\documentclass[11pt]{article}
\pdfoutput=1
\usepackage[top=1in, bottom=1in, left=1in, right=1in]{geometry}

\usepackage{euscript}
\usepackage{microtype}

\usepackage[utf8]{inputenc}
\usepackage{graphicx,amsmath,amsthm,amsfonts,amssymb}
\RequirePackage[OT1]{fontenc}
\usepackage[linktocpage,colorlinks,linkcolor=blue,anchorcolor=blue,citecolor=blue,urlcolor=blue,pagebackref]{hyperref}

\renewcommand*{\backref}[1]{\ifx#1\relax \else Page #1 \fi}
\renewcommand*{\backrefalt}[4]{%
  \ifcase #1 \footnotesize{(Not cited.)}%
  \or        \footnotesize{(Cited on page~#2.)}%
  \else      \footnotesize{(Cited on pages~#2.)}%
  \fi
}

\allowdisplaybreaks
\usepackage{booktabs, rotating, float}
\usepackage{mathtools}
\usepackage{bbm}
\usepackage{wasysym}
\usepackage{enumerate}
\usepackage{color}
\usepackage{listings}
\usepackage{epstopdf}
\usepackage{makecell}
\usepackage{tikz}
\usetikzlibrary{decorations.pathreplacing}
\usepackage{array}
\usepackage{caption}
\usepackage{subcaption}
\usepackage{soul}
\usepackage{hyperref}
\usepackage{verbatim}
\usepackage{siunitx}
\usetikzlibrary{calc} \tikzset{>=latex}
\usetikzlibrary{calc,decorations.markings}
\usetikzlibrary{positioning,shapes,decorations.text}

\theoremstyle{plain}
\newtheorem{thm}{Theorem}[section]
\newtheorem{cor}[thm]{Corollary}

\newtheorem{prop}[thm]{Proposition}
\newtheorem{lemma}[thm]{Lemma}
\newtheorem{assumption}{Assumption}[section]
\theoremstyle{definition}
\newtheorem{defn}[thm]{Definition}
\newtheorem{rem}[thm]{Remark}
\newtheorem{eg}[thm]{Example}
\newtheorem*{sinusoid}{Sinusoidal Covariance}

\newtheorem{fact}[thm]{Fact}

\newtheorem{observe}[thm]{Observation}

\newcommand{\rpm}{\sbox0{$1$}\sbox2{$\scriptstyle\pm$}
  \raise\dimexpr(\ht0-\ht2)/2\relax\box2 }

\tikzstyle{nd} = [anchor=base, inner sep=0pt]
\tikzstyle{ndpic} = [remember picture, baseline, every node/.style={nd}]

\newcommand{\E}{\mathbb E}

\usepackage{natbib}
\setcitestyle{authoryear,open={(},close={)}} 
\newcommand{\purple}[1]{\textcolor{black}{#1}}

\newcommand{\mb}{\mathbb}
\newcommand{\mc}{\mathcal}
\newcommand{\lv}{\left\lVert}
\newcommand{\rv}{\right\rVert}
\newcommand{\dtheta}{\Delta\theta}
\newcommand\topb{\mathrel{\stackrel{\makebox[0pt]{\mbox{\tiny p}}}{\to}}}
\newcommand\todbt{\mathrel{\stackrel{\makebox[0pt]{\mbox{\tiny d}}}{\to}}}

\newcommand{\MSE}{\mathsf{MSE}}
\newcommand{\PrE}{\mathsf{PE}}

\def\ba{\begin{enumerate}[(a)]}
\def\bei{\begin{enumerate}[(i)]}
\def\be{\begin{enumerate}[(1)]}
\def\ee{\end{enumerate}}
\def\bi{\begin{itemize}}
\def\ei{\end{itemize}}
\def\beg{\begin{eg}}
\def\eeg{\end{eg}}
\def\bd{\begin{defn}}
\def\ed{\end{defn}}
\def\bt{\begin{thm}}
\def\et{\end{thm}}
\def\bl{\begin{lemma}}
\def\el{\end{lemma}}
\def\bfac{\begin{fact}}
\def\efac{\end{fact}}
\def\Bar{\overline}
\def\bc{\begin{cor}}
\def\ec{\end{cor}}
\def\bp{\begin{prop}}
\def\ep{\end{prop}}
\def\bo{\begin{observe}}
\def\eo{\end{observe}}
\def\bas{\begin{assumption}}
\def\eas{\end{assumption}}

\def\beg{\begin{eg}}
\def\eeg{\end{eg}}
\title{High-dimensional scaling limits and fluctuations of online least-squares SGD with smooth covariance}

\author{
Krishnakumar Balasubramanian\thanks{Department of Statistics, University of California, Davis. Research of this author was supported in part by National Science Foundation (NSF) grant DMS-2053918. \texttt{kbala@ucdavis.edu}}
\and
Promit Ghosal \thanks{Department of Mathematics, Brandeis University. Research of this author was supported in part by National Science Foundation (NSF) grant DMS-2153661. \texttt{promit@brandeis.edu}}
\and
{Ye He} \thanks{School of Mathematics, Georgia Institute of Technology. Research of this author was supported in part by NSF TRIPODS grant CCF-1934568 awarded to UC Davis. \texttt{leohe@ucdavis.edu}}
}

\date{}
\begin{document}
\maketitle
{
\hypersetup{linkcolor=black}
}
\begin{abstract}
We derive high-dimensional scaling limits and fluctuations for the online least-squares Stochastic Gradient Descent (SGD) algorithm by taking the properties of the data generating model explicitly into consideration. Our approach treats the SGD iterates as an interacting particle system, where the expected interaction is characterized by the covariance structure of the input. Assuming smoothness conditions on moments of order up to eight orders, and without explicitly assuming Gaussianity, we establish the high-dimensional scaling limits and fluctuations in the form of infinite-dimensional Ordinary Differential Equations (ODEs) or Stochastic Differential Equations (SDEs). Our results reveal a precise three-step phase transition of the iterates; it goes from being ballistic, to diffusive, and finally to purely random behavior, as the noise variance goes from low, to moderate and finally to very-high noise setting. In the low-noise setting, we further characterize the precise fluctuations of the (scaled) iterates as infinite-dimensional SDEs. We also show the existence and uniqueness of solutions to the derived limiting ODEs and SDEs. Our results have several applications, including characterization of the limiting mean-square estimation or prediction errors and their fluctuations, which can be obtained by analytically or numerically solving the limiting equations.
\end{abstract}
{
\hypersetup{linkcolor=black}
\tableofcontents
}
\section{Introduction}\label{sec:Introduction}

Stochastic Gradient Descent (SGD) is an algorithm that was initially developed by \cite{robbins1951stochastic} for root-finding. Today, SGD and its variants are the most commonly used algorithms for training machine learning models, ranging from large-scale linear models to deep neural networks. One of the main challenges in understanding SGD is comprehending its convergence properties. In the case of fixed-dimensional problems, the learning theory and optimization communities have focused on providing non-asymptotic bounds, either in expectation or with high-probability, over the past two decades. However, such bounds often tend to be overly conservative in predicting the actual behavior of the SGD algorithm on large-scale statistical problems occurring in practice that are invariably based on specific data generating models. To address this, recent research has concentrated on characterizing the exact dynamics of SGD in large-scale high-dimensional problems. Specifically, the focus is on obtaining the precise asymptotic behavior of SGD and its fluctuations when the number of iterations or observations and the data dimension tend to infinity under appropriate scalings. The main idea behind this approach is to demonstrate that, under the considered scaling, the noise effects in SGD average out, so the exact asymptotic behavior and fluctuations are determined by a particular set of dynamical system equations.

Our goal in this work is to consider the SGD algorithm on a specific statistical problem, namely the linear regression problem, and provide a fine-grained analysis of its behavior under high-dimensional scalings. Specifically, we consider the linear regression model, $Y = X^\top \theta^* + \mathcal{E}$, where $\theta^*\in\mathbb{R}^d$ is the true regression coefficient, $X \in \mathbb{R}^d$ is the zero-mean input random vector with covariance matrix $\E[XX^\top]=\Sigma_d \in \mathbb{R}^{d\times d}$, $Y \in \mathbb{R}$ is the output or response random variable, and $\mathcal{E} \in \mathbb{R}$ is a zero-mean noise with at least finite-variance. For this statistical model, we consider minimizing the following population least-squares stochastic optimization problem
\begin{align*}
\min_{\theta \in \mathbb{R}^d}\mathbb{E}[(Y-\langle X,\theta\rangle)^2],
\end{align*}
using the online SGD with an initial guess $\theta^0 \in \mathbb{R}^d$, given by the following iterations
\begin{align}\label{eq:online SGD}
    \theta^{t+1} &= \theta^{t} +\eta \left(y^t - \langle x^t ,\theta^{t} \rangle \right) x^t,
    \end{align}
where $\theta^t\in\mb{R}^d$ is the output at time $t$. The sequence \purple{$\{x^t,y^t\}_{t\ge 0}$ corresponds to the observations, where $y^t=\langle x^t, \theta^* \rangle+\varepsilon^t$} and $\eta>0$ is the step-size parameter and plays a crucial in obtaining our scaling limits and fluctuations. Specific choices for $\eta$ will be detailed shortly in Section~\ref{sec:main}. The sequences $\{x^t\}_{t\ge 0}$ and $\{\varepsilon^t\}_{t\ge 0}$ are assumed to be independent and identical copies of the random vector $X$ and noise $\mathcal{E}$ respectively. Note in particular that for the online SGD in~\eqref{eq:online SGD}, the number of iterations is equal to the number of observations used.

In our analysis, we view the least-squares online SGD in~\eqref{eq:online SGD}, as a \emph{discrete} space-time interacting particle system, where the space-axis corresponds to the coordinates of the vector $\theta^{t}$ and the time-axis corresponds to the evolution of the algorithm. Specifically, note that the online SGD updates in~\eqref{eq:online SGD} can be viewed in the following coordinate-wise form. For any $1\le i\le d$,
\begin{align*}
    \theta^{t+1}_i 
    &= \theta^{t}_i + \eta x^t_i y^t- \eta \sum_{j=1}^{d} x^t_i x^t_j \theta^{t}_j \\
    &= \theta^t_i + \eta x^t_i \Big( \sum_{j=1}^d x^t_j \theta^*_j +\varepsilon{^t}-\sum_{j=1}^d x^t_j\theta^t_j \Big) \\
    &= \theta^t_i - \eta \sum_{j=1}^d x^t_ix^t_j (\theta^t_j-\theta^*_j)+\eta x^t_i \varepsilon^t.
\end{align*}
Now, defining the centralized iterates as $\dtheta^t\coloneqq\theta^t-\theta^*$ and letting $\dtheta^t_i$ denote its $i$-th coordinate for all $1\le i\le d$ and $t\ge 0$, the least-squares online SGD can be then alternatively be represented as the following interacting particle system:
\begin{align}\label{eq:centralized online SGD}
    \dtheta^{t+1}_i=\dtheta^t_i-\underbrace{\eta \sum_{j=1}^d x^t_ix^t_j \dtheta^t_j}_{\text{\emph{random} interaction}} +\eta x^t_i \varepsilon^t , \qquad 1\le i\le d,
\end{align}
where the particles $\{\dtheta^t_i\}_{1\le i\le d}$ are interacting and evolve over a discrete-time scale. In particular, the interaction among the particles $\{\dtheta^t_i\}_{1\le i\le d}$ is random for any $t$, and the \emph{expected interaction} is captured by the covariance matrix $\Sigma_d$ of the input vector $X$. Therefore, to analyze the high-dimensional asymptotic properties of the least-squares online SGD, we analyze the scaling limit and fluctuations of the interacting particle system given by \eqref{eq:centralized online SGD}. In particular, our limits are derived in the form of infinite-dimensional Ordinary Differential Equations (ODEs) and Stochastic Differential Equations (SDEs); for some background, we refer to~\cite{arnold1992ordinary, kallianpur1995stochastic,da2014stochastic}.

Our approach is also motivated by the larger literature available on analyzing interacting particle systems.  See, for example,~\cite{kipnis1998scaling, liggett1999stochastic, zeitouni2004random, sznitman2004topics, darling2008differential,spohn2012large}. The interacting particle system in~\eqref{eq:centralized online SGD} can either exhibit long-range or short-range interactions, depending on the structure of the covariance matrix $\Sigma_d$. The case when the covariance matrix $\Sigma_d$ is ``smooth'' in an appropriate sense, thereby prohibiting abrupt changes in the entries of the covariance matrix, corresponds to the regime of long-range interactions. Examples of such covariance matrices include bandable and circulant covariance matrices that arise frequently in practice~\citep{gray2006toeplitz}. 

\subsection{Preliminaries}\label{sec:Notations}
Before we proceed, we list the notations we make in this work. For a positive integer $a$, we let $[a]\coloneqq\{1,\ldots, a\}$. For vectors, superscripts denote time-index and subscripts denote coordinates. The space of square-integrable functions on $S$, a subset of Euclidean space, is denoted as $L^2(S)$ with the squared-norm $\lv g \rv^2_{L^2(S)}\coloneqq \smallint_S g(x)^2 dx$ for any $g\in L^2(S)$. The space of continuous functions in $[0,1]$, is denoted by $C([0,1])$, and is equipped with the topology of uniform convergence over $[0,1]$. The space of continuous functions with continuous derivatives up to $k^{\textrm{th}}$-order in $[0,1]$, is denoted by $C^k([0,1])$, and is equipped with the topology of uniform convergence in function value and derivatives up to $k^{\textrm{th}}$-order over $[0,1]$ (we mainly use $k=1,2$). For any topological space $\mc{H}$ and $\tau>0$,
$C([0,\tau];\mc{H})$ represents the space of $\mc{H}$-valued functions with continuous trajectories.

We also require the definition of a Gaussian random field or Gaussian random field process, that arise in characterizing the limiting behavior of the SGD iterates. We refer to~\cite{adler2007random} for additional background.
\begin{defn}
    A \textit{Gaussian random field} is defined as a random field $g$ on a parameter set $[0,1]$ for which the finite-dimensional distributions of $(g(x_1),\cdots, g(x_K))$ are multivariate Gaussian for each $1\le K<\infty$ and each $(x_1,\cdots,x_K)\in [0,1]^K$. A \textit{Gaussian random field process} is defined as a time-indexed random field $g$ on a parameter set $[0,\infty)\times [0,1]$ for which $\{(g(t,x_1),\cdots, g(t,x_K))\}_{t\ge 0}$ are multivariate Gaussian processes for each $1\le K<\infty$ and each $(x_1,\cdots,x_K)\in [0,1]^K$.
\end{defn}

\textbf{Space-time Interpolation.} 
Our approach starts by constructing a space-time stochastic process by performing a piecewise linear interpolation of the discrete particles in \eqref{eq:centralized online SGD}. This process, denoted by $\{\Bar{\Theta}^{d,T}(s,x)\}_{s\in [0,\tau],x\in [0,1]}$ is continuous both in time and in space. The spatial coordinate at the macroscopic scale is indexed by the set $[0,1]$ and is denoted by the spatial variable $x$. The spatial resolution at the microscopic scale is of order $1/d$. Let $T\in \mb{N}_+$ be a positive integer parameter. The parameters $\tau \in (0,\infty)$ and $1/T$ corresponds to the time-scale \purple{until} which we would like to observe the trajectory and the resolution of the time-axis respectively.  For any $\tau$, we consider the first $\lfloor \tau T \rfloor$ least-squares online SGD iterates. Specifically, $\lfloor\tau T\rfloor$ corresponds to the overall number of iterations, which also corresponds to the number of observations used.

We now describe how to construct the interpolation of the discrete particles in \eqref{eq:centralized online SGD}. Consider the function $\Theta^{d,T}(\cdot,\cdot):[0,\tau]\times [0,1]\to \mb{R}$ that satisfies the following conditions:
\begin{itemize}
    \item [(a)]\label{embedding 1 LLN} it is evaluated as the positions of SGD particles on the grid points with grid width $T^{-1}$ in the time variable and $d^{-1}$ in the space variable, i.e. for any $i\in [d]$ and any $0\le t\le \lfloor{\tau T}\rfloor$, $$\Theta^{d,T}\Big(\frac{t}{T},\frac{i}{d}\Big)=\dtheta^t_i.$$ We can artificially define $\dtheta^t_0=\dtheta^t_1$. 
    \item [(b)]\label{embedding 2 LLN} it is piecewise-constant in the time variable, i.e. for any $s\in [0,\tau]$,  $$\Theta^{d,T}(s,\cdot)=\Theta^{d,T}\Big(\frac{\lfloor{sT}\rfloor}{T},\cdot \Big).$$
    \item [(c)]\label{embedding 3 LLN} it is piecewise linear in the space variable, i.e. for any $x\in [0,1]$, 
    \[
    \Theta^{d,T}(\cdot,x)=(\lfloor{dx}\rfloor+1-dx)\Theta^{d,T}\Big(\cdot,\frac{\lfloor{dx}\rfloor}{d}\Big)+(dx-\lfloor{dx}\rfloor)\Theta\Big(\cdot,\frac{\lfloor{dx}\rfloor+1}{d}\Big).
    \]
\end{itemize}
From the above construction, we can express $\Theta^{d,T}$ explicitly based on $\{\dtheta^t_i\}_{i\in [d], 0\le t\le \lfloor{\tau T}\rfloor}$: for any $s\in [0,\tau]$ and any $x\in [0,1]$,
\begin{align*}
\begin{aligned}
    \Theta^{d,T} (s,x)&= (\lfloor{dx}\rfloor+1-dx)\Theta^{d,T}\Big(\frac{\lfloor{sT}\rfloor}{T},\frac{\lfloor{dx}\rfloor}{d}\Big)+(dx-\lfloor{dx}\rfloor)\Theta\Big(\frac{\lfloor{sT}\rfloor}{T},\frac{\lfloor{dx}\rfloor+1}{d}\Big) \\
    &=(\lfloor{dx}\rfloor+1-dx)\dtheta^{\lfloor{sT}\rfloor}_{\lfloor{dx}\rfloor}+(dx-\lfloor{dx}\rfloor)\dtheta^{\lfloor{sT}\rfloor}_{\lfloor{dx}\rfloor+1}.
\end{aligned}
\end{align*}
Condition (b) in particular implies that $\Theta^{d,T}$ is c\`adla\`g in the time variable. From condition (c), it is easy to observe that $\Theta^{d,T}$ is continuous with well-defined weak
derivative in space. Actually, the first order space derivative to $\Theta^{d,T}$ is well-defined for all $x$ except the grid points $\{0,\frac{1}{d},\cdots, \frac{d-1}{d},1 \}$. Next, we construct $\Bar{\Theta}^{d,T}$ which is piecewise linear both in time and  in space. For any $s\in [0,\tau]$ and $x\in [0,1]$,
\begin{align}\label{eq:approximation to Theta LLN}
    \Bar{\Theta}^{d,T}(s,x)\coloneqq\left( \lfloor{sT}\rfloor+1-sT \right)\Theta^{d,T}\Big(\frac{\lfloor{sT}\rfloor}{T},x\Big)+\left(sT-\lfloor{sT}\rfloor\right)\Theta^{d,T}\Big(\frac{\lfloor{sT}\rfloor+1}{T},x\Big).
\end{align}
From this construction we can immediately see that $\Bar{\Theta}^{d,T}\in C([0,\tau];C[0,1])$. Our main objective in this work, is to characterize the limiting behavior of the space-time stochastic process $\{\Bar{\Theta}^{d,T}(s,x)\}_{s\in [0,\tau],x\in [0,1]}$ when both $T$ and $d$ go to infinity, under appropriate scalings for appropriate choices of the step-size parameter $\eta$, and under various assumptions on the noise $\mathcal{E}$.

\subsection{Assumptions}\label{sec:assumptions LLN} 
We start with the following definition of a space-time stochastic process, which characterizes the data generating process for $X$ (or alternatively the sequence $\{ x^t\}_{t\geq 0}$). 
\begin{defn}[Data Generating Process]\label{def:white noise} Let $A, B$, and $E$ be real-valued symmetric functions defined on $[0,1]^2$, $[0,1]^4$ and $[0,1]^8$ respectively. $\{W(s,x)\}_{s\ge 0,x\in [0,1]}$ is a stochastic process which is white in the time variable, such that for all $s\ge 0$ and $x,x_1,\cdots,x_8\in [0,1]$:
\begin{align*}
   &\mb{E}\left[ W(s,x) \right]=0,\qquad \mb{E}\left[ W(s,x_1)W(s,x_2) \right]=A(x_1,x_2),\\
   \mb{E}\big[ \prod_{i=1}^4 W(&s,x_i) \big]=B(x_1,x_2,x_3,x_4), \qquad \mb{E}\big[ \prod_{i=1}^8 W(s,x_i)\big]=E(x_1,\cdots,x_8).
\end{align*}
\end{defn}
We now introduce the assumptions on the samples $\{x^t_i\}_{t\ge 0,i\in [d]}$ and random noises $\{\varepsilon^t\}_{t\ge 0}$.
\begin{assumption}\label{ass:tightness in space LLN}
The samples $\{x^t_i\}_{1\le i\le d,t\ge 0}$ and the noises $\{\varepsilon^t\}_{t\ge 0}$ are independent and satisfy
\begin{itemize}
    \item [\textup{(a)}]\label{ass:1.1-1} There exists a stochastic process $\{W(s,x)\}_{s\ge 0,x\in [0,1]}$, defined in Definition \ref{def:white noise} such that $$x^t_i=W\Big(\frac{t}{T},\frac{i}{d}\Big).$$
    \item [\textup{(b)}] \label{ass:1.1-4} There exists a universal constant $C_1>0$ such that for any $t\ge 0$, $$\mb{E}\big[\varepsilon^t\big]=0, \quad  \mb{E}\big[\big| \varepsilon^t \big|^2\big] \purple{=} \sigma_d^2\quad\text{and}\quad\mb{E}\big[\big| \varepsilon^t \big|^4\big] \purple{=} C_1 \sigma_d^4.
    $$
\end{itemize}
\end{assumption}
\noindent The noise-variance $\sigma_d^2$ plays a crucial role in our scaling limits. In particular, it is allowed to grow with $d$, with the growth rate determining the precise scaling limit of the SGD iterates. In the following, we drop the subscript $d$ for convenience. We now introduce the main assumption we make regarding smoothness of the covariance and higher-moments of the process $W$.

\begin{assumption}[Smoothness Conditions]\label{ass:continuous data} Stochastic process $\{W(s,x)\}_{s\ge 0,x\in [0,1]}$  satisfies the following smoothness conditions:
\begin{itemize}
    \item [1.]\label{ass:1.1-2} The function $A:[0,1]^2\to \mb{R}$ is such that $A(x,\cdot)\in C^1([0,1])$ for any $x\in [0,1]$. As a result, there exist constants $C_2,C_{3}$ such that for any $x,y,z\in [0,1]$, we have    \begin{align*}
        |A(x,y)|\le C_2\qquad \text{and} \qquad |A(x,z)-A(y,z)|\le C_{3}|x-y|.
    \end{align*}
    We further assume that there exists a  constant $C_4$ such that $$|A(x,x)+A(y,y)-2A(x,y)|\le C_{4}^2 |x-y|^2.$$
    \item [2.]\label{ass:1.1-3} For the function $B:[0,1]^4\to\mb{R}$, there exist constants $C_{5},C_{6},C_{7}>0$ such that for any $x_1,x_2,x_3,x_4\in [0,1]$, $|B(x_1,x_2,x_3,x_4)|\le C_{5}$ and 
    \begin{align*}
        &|B(x_1,x_3,x_1,x_3)+B(x_2,x_3,x_2,x_3)-2B(x_1,x_3,x_2,x_3)|\le C_{6}^2|x_1-x_2|^2,\\
        &\big| B(x_1,x_1,x_1,x_1)+B(x_2,x_2,x_2,x_2)+6B(x_1,x_1,x_2,x_2)-4B(x_1,x_1,x_1,x_2)\\
    &\qquad\qquad\qquad\qquad\qquad\qquad\qquad\quad-4B(x_1,x_2,x_2,x_2)\big|\le C_{7}^4 \big|x_1-x_2\big|^4.
    \end{align*}
    \item [3.] For the function $E:[0,1]^8\to \mb{R}$, there exist constants $C_8,C_9$ such that for all $x_1,\cdots,x_8\in [0,1]$, $E(x_1,\cdots,x_8)\le C_8$ and
    \begin{align*}
 \hspace{-0.5in}\big| & E(x_1,x_1,x_1,x_1,x_3,\cdots,x_6)+E(x_2,x_2,x_2,x_2,x_3,\cdots,x_6)\\
 &+6E(x_1,x_1,x_2,x_2,x_3,\cdots,x_6)-4E(x_1,x_1,x_1,x_2,x_3,\cdots,x_6)\\
 &\quad\qquad\qquad\qquad\qquad\qquad\qquad-4E(x_1,x_2,x_2,x_2,x_3,\cdots,x_6)\big|\le C_9^4 \big|x_1-x_2\big|^4.
    \end{align*}
\end{itemize}
\end{assumption}

\textbf{Motivation for the above assumptions.} Assumptions~\ref{ass:tightness in space LLN} and~\ref{ass:continuous data} are made to ensure tightness of the interpolated process in~\eqref{eq:approximation to Theta LLN}. In particular, the assumptions on the second and fourth-order moments are required to derive the tightness conditions needed to establish the scaling limits, with the second-order moment information actually showing up in the limit. Additionally, assumptions on the eighth-order moments are required to derive the tightness conditions needed to establish the fluctuations. In this case, the second-order moments information appears in the drift terms, and both the second and fourth-order moment information show up in the diffusion term.

From a statistical perspective, our assumptions above allow for a relative general class of distributions for the data and the noise sequence. \textcolor{black}{Specifically, we now characterize the relationship between the decay of the eigenvalues of $A$ and Assumption~\ref{ass:continuous data}. We proceed by considering the covariance function $A$ of the centered stochastic process $\{W(s,x)\}_{s\ge 0,x\in [0,1]}$ in the Mercer decomposition form as below,
\begin{align}\label{eq:covariance function extend example}
    A(x,y)= \sum_{k=1}^\infty \lambda_k \phi_k(x)\phi_k(y),
\end{align}
where $\{\lambda_k\}_{k\ge 1}$ is a sequence of positive coefficients (eigenvalues) in decreasing order and $\{\phi_k\}_{k\ge 1}$ is an orthonormal basis of $L^2([0,1])\cap C^1([0,1])$. In the following result, we provide a sufficient condition on the eigenvalues so that Assumption~\ref{ass:continuous data}-1 is satisfied.}

\begin{prop}\label{prop: sufficient condition equicts} \purple{Suppose $\{\phi_k\}_{k\ge 1}\subset L^2([0,1])\cap C^2([0,1])$ is an orthonormal basis of $L^2([0,1])$. 
Define 
$$
S_{N,0}\coloneqq \sum_{k=1}^N \lambda_k \phi_k(x)\phi_k(y) \quad \text{and}\quad 
S_{N,1}\coloneqq \sum_{k=1}^N \lambda_k \phi_k(x)\phi_k'(y),
$$
where $\phi_k'$ denotes the derivative of $\phi_k$. If $\{(\lambda_k,\phi_k)\}_{k\ge 1}$ satisfies that for all $k\ge 1$ and $i=0,1$,
\begin{align}\label{eq:covariance function sufficient condition}
    &\sup_{x\in [0,1]} |\phi_k^{(i)}(x)|\le C_{k,i} ,\qquad  \sum_{k=1}^\infty \lambda_k  C^2_{k,i}<\infty, \qquad \sum_{k=1}^\infty \lambda_k  C_{k,0}C_{k,2}<\infty
\end{align} 
then $\{S_{N,0}\}_{N\ge 1}$ is uniformly bounded and $\{S_{N,1}\}_{N\ge 1} 
$ is uniformly bounded and uniformly equicontinuous. Therefore, the process $\{W(s,x)\}_{s\ge 0,x\in [0,1]}$, with covariance function $A$ given in \eqref{eq:covariance function extend example}, satisfies Assumption \ref{ass:continuous data}-1.}
\end{prop}
\begin{proof}[Proof of Proposition \ref{prop: sufficient condition equicts}]\label{proof:sufficient condition equicts} 
\purple{The fact that $\{S_{N,0}\}_{N\ge 0}$ and $\{S_{N,1}\}_{N\ge 0}$ are uniformly bounded follows directly from \eqref{eq:covariance function sufficient condition}. Now, for any $x,y_1,y_2\in [0,1]$, 
 \begin{align*}
    | S_{N,1}(x,y_1)-S_{N,1}(x,y_2)|&\le  \sum_{k=1}^N \lambda_k  C_{k,0} |\phi_k'(y_1)-\phi_k'(y_2)|\le \big(\sum_{k=1}^\infty \lambda_k  C_{k,0} C_{k,2}\big) |y_1-y_2|.
 \end{align*}
 Therefore $\{S_{N,1}\}_{N\ge 0}$ is uniformly equicontinuous. Hence, according to the Arzel\`{a} Ascoli Theorem, $A(x,\cdot)\in C^1([0,1])$.}
 
 \purple{Due to symmetry, $A$ is bounded and Lipschitz in each variable and
 \begin{align*}
     &|A(x,x)+A(y,y)-2A(x,y)|=\lim_{N\to\infty} |S_{N,0}(x,x)+S_{N,0}(y,y)-2S_{N,0}(x,y)| \\
      \le & \lim_{N\to\infty} \sum_{k=1}^N \lambda_k \big( \phi_k(x)\phi_k(x)+\phi_k(y)\phi_k(y)-2\phi_k(x)\phi_k(y) \big)\le \big(\sum_{k=1}^\infty \lambda_k C_{N,1}^2\big) |x-y|^2.
 \end{align*}
 Therefore Assumption \ref{ass:continuous data}-1 is satisfied.}
\end{proof}


\begin{rem} \purple{We emphasize here for Proposition~\ref{prop: sufficient condition equicts}, we do not make either Gaussianity assumptions. While we mainly discussed Assumption \ref{ass:continuous data}-1, if we further assume that the centered process $\{W(s,x)\}_{s\ge 0,s\in [0,1]}$ is an elliptical process, Assumption \ref{ass:continuous data}-2 and Assumption \ref{ass:continuous data}-3 are also satisfied by an immediate consequence of the generalized Isserlis' Theorem. Alternatively, one could also proceed to do more explicit computations as above to related the eigenvalues and the smoothness conditions in Assumption \ref{ass:continuous data}-2 and Assumption \ref{ass:continuous data}-3. To summarize, our assumptions are based on the spatial structure of the $k^{\textrm{th}}$-order moments of the sample distribution for $k=2,4,8$, which in turn are satisfied as long as the eigenvalues of the covariance do not decay too slow.}  
\end{rem}
\begin{rem}
\purple{The eigen-decay assumptions are commonly made in the analysis of least-squares SGD algorithm in the fixed finite-dimensional and fixed infinite-dimensional setting; see, for example,~\cite{lin2017optimal,dieuleveut2017harder, pillaud2018statistical,mucke2019beating,bartlett2020benign}. In fact, they were originally made in the context of analyzing kernel ridge regression estimators (i.e., infinite-dimensional least-squares in reproducing kernel Hilbert spaces) by~\cite{caponnetto2007optimal}. From a practical statistical data analysis perspective, such assumptions are motivated by empirical observations that most big datasets exhibit decaying eigenvalues and are well-approximated by low-rank matrices~\citep{udell2019big}.}
\end{rem}

\textcolor{black}{In particular, \eqref{eq:covariance function extend example} and \eqref{eq:covariance function sufficient condition} could be used to characterize examples that cover a class of data sets that are widely used in practice. Suppose the data covariance matrix, $\Sigma_d$ has the following eigenvalue decomposition:
\begin{align}\label{eq:eigenvalue decomposition}
    \Sigma_d = \sum_{k=1}^\infty \lambda_k \mathbf{v}_k \mathbf{v}_k^\intercal, 
\end{align}
where $\{\mathbf{v}_k\}_{k\ge 1} $ is an orthonormal basis of $\mb{R}^d$ such that $\mathbf{v}_k^\intercal =(\phi_k(i/d))_{i\in [d]}$ for all $k\ge 1$. Then, Assumption \eqref{eq:covariance function sufficient condition} is satisfied when the sequence of eigenvalues $\{\lambda_k\}_{k\ge 1}$ decays not too slow enough. For example, in the case when $\{\phi_{k}\}_{k\ge 1}$ is the Fourier basis of $L^2([0,1])$, conditions in  \eqref{eq:covariance function sufficient condition} are satisfied when $\{\lambda_k\}_{k\ge 1}$ decay as fast as $k^{-5/2-\varepsilon}$ for arbitrary small positive $\varepsilon$. See, Section~\ref{sec:nonasym error} and Lemma \ref{lem:Sinusoidal graphon A} for an explicit computation in the Gaussian case for simplicity.}
%


We also emphasize that our analysis is general enough to allow for a certain degree of dependence within the data sequence $\{x^t\}_{t \geq 0}$ and within the error sequence $\{\varepsilon^t\}_{t \geq 0}$. Such assumptions are typically made in several applications, including reinforcement learning~\citep{meyn2022control} and sequential or online decision-making~\citep{chen2021statistical,khamaru2021near}. Furthermore, we could also allow for some level of dependency between the data and error sequences. We do not discuss these extensions in detail to keep our exposition simpler. 

\subsection{Main results}\label{sec:main}

\purple{In this section, we present our main results on the scaling limits of the least-squares online SGD trajectory and its fluctuation. The limits in our main results follow SDEs~\footnote{In some cases, the SDE limits degenerate to ODE limits.} that are driven by different Gaussian random field processes. In terms of convention, for any Gaussian random field process $\{\xi(s,x)\}_{s\ge 0,x\in [0,1]}$, $\mathrm{d} \xi(s,x)$ denotes the infinitesimal increment of the Gaussian random field process in the time variable $s$.} 

\purple{We also have the following convention regarding the initial conditions for the limiting processes come from the limit of space-time interpolated functions introduced in Section \ref{sec:Introduction}. At $t=0$, we only interpolate $\{\Delta \theta^0_i\}_{i\in [d]}$ in space and we denote the piecewise linear interpolation function by $\Theta^d$. Under our assumptions on $\{\Delta\theta^0_i\}_{i\in [d]}$ in each result below, the subsequential limit of $\Theta^d$ exists in $C([0,1])$ and we pick one such subsequential limit as the initial condition denoted as, $\Theta_0$.}  


\subsubsection{Scaling limits}
We now state our results on the scaling limits of the least-squares online SGD under different orders of noise variances. 
 
\begin{thm}[Scaling Limits]\label{thm:mainmean}
Let the initial conditions satisfy: For any $0\le i\le d$, there exist constants $R,L$ such that
\begin{align*}
\mb{E}\left[|\dtheta^0_i|^4 \right]\le R^4  \qquad\text{and}\qquad \mb{E}\left[ |\dtheta^0_i-\dtheta^0_{i-1}|^4  \right]\le L^4d^{-4}.
\end{align*}
Also, let Assumption \ref{ass:tightness in space LLN} and Assumption \ref{ass:continuous data} hold and let $\{\xi_1(s,x)\}_{s\in [0,\tau],x\in [0,1]}$ denote a centered Gaussian random field process \purple{such that for any $s_1,s_2\in [0,\tau]$ and $\ x,y\in [0,1]$, 
\begin{align}\label{eq:gaussian field variance high noise LLN}
       \mb{E}[\xi_1(s_1,x)\xi_1(s_2,y)]=\sigma_1(s_1,s_2,x,y)\quad\text{with}\quad \sigma_1(s_1,s_2,x,y)= (s_1\wedge s_2) A(x,y). 
    \end{align}}  Further, let there exist a uniform constant $C_{s,1}$ such that \textcolor{black}{for all $d\ge 1,T>0$ along the limiting sequence}, $\max(\eta d T, \sigma \eta T^{\frac{1}{2}})\le C_{s,1}$. Then, we have the following scaling limits. 

\begin{enumerate}
    \item \textbf{Low-noise, i.e., $\sigma^2/d^2T \to 0$:} Assume further that there exists a uniform positive constant $\alpha$ such that $\lim_{d,T\to\infty} \eta d T=\alpha$. Then for any $\tau\in (0,\infty)$, $\{\Bar{\Theta}^{d,T}\}_{d\ge 1,T>0}$ converges weakly to a function $\Theta$ as $d,T\to \infty$. Furthermore, the limit $\Theta\in C^1([0,\tau];C^1([0,1]))$ is the unique continuous solution to the following ODE:
\begin{align}\label{eq:mainmean ODE LLN}
    \partial_s \Theta(s,x)=-\alpha \int_0^1 A(x,y) \Theta(s,y) \mathrm{d}y, \quad\purple{s\in [0,\tau], x\in [0,1]}.
\end{align}
    \item \textbf{Moderate-noise, i.e., $\sigma^2/d^2T \to (0,\infty)$:} Assume further that there exist uniform positive constants $\alpha, \beta$ such that $\lim_{d,T\to\infty} \sigma^2/(d^2 T)=\beta^2 $ and $\lim_{d,T\to\infty} \eta d T=\alpha$ \textcolor{black}{(hence, as a consequence $\lim_{d,T\to\infty} \eta \sigma T^{\frac{1}{2}}=\alpha\beta$)}. Then for any $\tau\in (0,\infty)$, $\{\Bar{\Theta}^{d,T}\}_{d\ge 1,T>0}$ converges weakly to a process $\{\Theta(s,\cdot)\}_{s\in [0,\tau]}$ as $d,T\to \infty$. Furthermore, the limit $\Theta$ is the unique solution in $ C([0,\tau];C([0,1]))$ to the following SDE:
\begin{align}\label{eq:mainmean SDE LLN}
    \mathrm{d} \Theta(s,x)=-\alpha \int_0^1 A(x,y) \Theta(s,y) \mathrm{d}y \mathrm{d}s+ \alpha \beta \mathrm{d}\xi_1(s,x), \quad\purple{s\in [0,\tau], x\in [0,1]}.
\end{align}
    \item \textbf{High-noise, i.e., $\sigma^2/d^2T \to \infty$:} Assume further that there exists a uniform positive constant $\alpha$ such that $\lim_{d,T\to\infty} \eta \sigma T^{\frac{1}{2}}=\alpha$. Then for any $\tau\in (0,\infty)$, $\{\Bar{\Theta}^{d,T}\}_{d\ge 1,T>0}$ converges weakly to a process $\{\Theta(s,\cdot)\}_{s\in [0,\tau]}$ as $d,T\to \infty$. Furthermore, the limit $\Theta$ is the unique solution in $ C([0,\tau];C([0,1]))$ to the following SDE:
\begin{align}\label{eq:mainmean BM LLN}
    \mathrm{d} \Theta(s,x)= \alpha  \mathrm{d}\xi_1(s,x), \quad\purple{s\in [0,\tau], x\in [0,1]}.
\end{align}
\end{enumerate}
\end{thm}

\begin{rem}
Theorem~\ref{thm:mainmean} shows that under high-dimensional scalings, the limiting behavior of the SGD trajectory exhibits a three-step phase transition: It goes from being ballistic (i.e., characterized by an infinite-dimensional ODE) in the low-noise setting, to diffusive (i.e., characterized by an infinite-dimensional SDE) in the moderate-noise setting, to purely random in the high-noise setting. The boundaries of this three-step phase transition are precisely characterized, and explicit dependencies on the order of dimension, iterations and step-size choices are identified. In the moderate and high-noise setting, the covariance of the diffusion term $\xi_1$ is determined by the second moment function $A$ from Assumption~\ref{ass:continuous data}. We also remark that our initial conditions are made coordinate-wise and are rather mild.  
\end{rem}

\begin{rem}\label{rem:LLN smoothness} In the low-noise setting, according to \eqref{eq:mainmean ODE LLN}, we see that $\Theta(s,\cdot)$ has the same order of smoothness as $A(\cdot,y)$ for any $s\in[0,\tau]$ and $y\in [0,1]$. This phenomenon is more general, i.e., if we further assume that $A(\cdot,y)\in C^k([0,1])$, for some finite positive integer $k$ and for any $y\in [0,1]$, then it can be shown that $\Theta\in C^1([0,\tau]; C^k([0,1]))$. 
\end{rem}

\begin{rem}
 Assumption \ref{ass:continuous data} is required to show tightness results in Proposition \ref{prop:L2 tightness in time LLN} and Proposition \ref{prop:L2 tightness in time and space LLN}, and hence to prove existence of the weak limit. It is worth mentioning that in the low-noise setup when the noise variance satisfies $\sigma^2=O(d^2)$, Assumption \ref{ass:continuous data} can be relaxed by dropping the last condition on $B$ and the condition on $E$. The tightness results under the relaxed assumptions can be proved in the same way by considering second moments bounds rather than the fourth moments bounds in Proposition \ref{prop:L2 tightness in time LLN} and Proposition \ref{prop:L2 tightness in time and space LLN}. For the sake of simplicity and consistency of our analysis, the proof of Theorem \ref{thm:mainmean} is based on Assumption \ref{ass:continuous data}.
\end{rem}

\subsubsection{Fluctuations}

We now study the fluctuation of $\{\Delta\theta^t_i\}_{i\in [d],0\le t\le N}$ in the low-noise setting. In order to do so, we look at a re-scaled difference between $\{\Delta\theta^t_i\}_{i\in [d],0\le t\le N}$ and its scaling limit. Specifically, for any $s\in [0,\tau]$, $x\in [0,1]$, define
\begin{align}\label{eq:fluctuation}
    U^{d,T}(s,x)\coloneqq \gamma \big( \Bar{\Theta}^{d,T}(s,x)-\Theta(s,x) \big),
\end{align}
where $\gamma$ is the scaling parameter, and we expect $\gamma\to\infty$ as $d,T\to\infty$. We now state our fluctuation results.

\begin{thm}[Fluctuations]\label{thm:mainfluc} Let the initial conditions follow: For any $0\le i\le d$, there exist \purple{uniform} constants $D, M$ such that
\begin{align*}
\mb{E}\left[|U^{d,T}(0,\frac{i}{d})|^4 \right]\le D^4 \qquad \text{and}\qquad \mb{E}\left[ |U^{d,T}(0,\frac{i}{d})-U^{d,T}(0,\frac{i-1}{d})|^4  \right]\le M^4d^{-4}.
\end{align*}
\purple{Let $\{\xi_{2}(s,x)\}_{s\in [0,\tau],x\in [0,1]}$ and $\{\xi_{3}(s,x)\}_{s\in [0,\tau],x\in [0,1]}$ denote two independent centered Gaussian field processes such that for any $s_1,s_2\in [0,\tau]$ and $x,y\in [0,1]$,
    \begin{align}
     &\mb{E}[\xi_2(s_1,x)\xi_2(s_2,y)]=\sigma_2(s_1,s_2,x,y)\quad\text{with}\quad \sigma_2(s_1,s_2,x,y)= (s_1\wedge s_2) A(x,y),
     \label{eq:gaussian field covariance CLT}\\
     \begin{split}
         &\mb{E}[\xi_3(s_1,x)\xi_3(s_2,y)]= \sigma_3(s_1,s_2,x,y)\quad \text{with}\\
 &\sigma_3(s_1,s_2,x,y) =\int_0^{s_1\wedge s_2}\int_0^1\int_0^1  \Theta(s,z_1)\Theta(s,z_2)\Tilde{B}(x,z_1,x,z_2)\mathrm{d}z_{1}\mathrm{d}z_{2}\mathrm{d}s ,
     \end{split}\label{eq:gaussian parameters}
    \end{align}
  where  $\Tilde{B}(x_1,x_2,x_3,x_4)=B(x_1,x_2,x_3,x_4)-A(x_1,x_3)A(x_2,x_4)$ for any $x_1,x_2,x_3,x_4\in[0,1]$.} Furthermore, let the assumptions made in Theorem \ref{thm:mainmean} in the low-noise setup, i.e. \textcolor{black}{$\sigma^2/d^2T\to 0$}, hold. Let $\eta=\frac{\alpha}{dT}$ and assume that there exists a uniform constant $C_{s,2}$ such that for all $d\ge 1,T>0$ \textcolor{black}{along the limiting sequence}, we have $\max(\gamma T^{-\frac{1}{2}},\gamma d^{-1},\gamma \sigma d^{-1}T^{-\frac{1}{2}} )\le C_{s,2}$ . Then we have the following fluctuation results.
\begin{enumerate}
    \item \textbf{Particle interaction dominates:} Assume further that, as $d,T\to\infty$, we have 
    \begin{align*}
    T=o(d^2), \quad \sigma=O(d), \quad \gamma T^{-\frac{1}{2}}\to \zeta\in(0,\infty), \quad \text{and}\quad \gamma \sigma d^{-1}T^{-\frac{1}{2}}\to\beta\in[0,\infty),
    \end{align*}
    for some uniform constants $\zeta,\beta$. Then for any $\tau\in (0,\infty)$, the fluctuation of SGD particles, $\{U^{d,T}\}_{d\ge 1,T>0}$ converges in distribution to a function $U \in C([0,\tau];C([0,1]))$ as $d,T\to \infty$. Furthermore, for any $s\in [0,\tau]$, $x\in [0,1]$, the limit $U$ is the unique solution in $C([0,\tau];C([0,1]))$ to the SDE 
    \begin{align}\label{eq:mainthm fluc SDE CLT}
    \mathrm{d} U(s,x)=&-\alpha\int_0^1 A(x,y)U(s,y)\mathrm{d}y \mathrm{d}s+\alpha \beta \mathrm{d}\xi_2(s,x)\\
    &+ \alpha \zeta \mathrm{d}\xi_3(s,x), \quad \purple{s\in [0,\tau], x\in [0,1]} \nonumber.
    \end{align}
    \item \textbf{Noise dominates:} Assume further that, as $d,T\to\infty$, we have
    \begin{align*}
    \max(d,T^{\frac{1}{2}})\ll \sigma\ll dT^{\frac{1}{2}}, \qquad\text{and}\qquad \gamma \sigma d^{-1}T^{-\frac{1}{2}}\to\beta\in(0,\infty),
    \end{align*}
    for a uniform constant $\beta$. Then for any $\tau\in (0,\infty)$, the fluctuation of SGD particles, $\{U^{d,T}\}_{d\ge 1,T>0}$ converges in distribution to a function $U \in C([0,\tau];C([0,1]))$ as $d,T\to \infty$. Furthermore, for any $s\in [0,\tau]$, $x\in [0,1]$, the limit $U$ is the unique solution in $C([0,\tau];C([0,1]))$ to the SDE 
    \begin{align}\label{eq:mainthm fluc SDE CLT relative large variance}
    \mathrm{d} U(s,x)=-\alpha\int_0^1 A(x,y)U(s,y)\mathrm{d}y \mathrm{d}s+\alpha \beta \mathrm{d} \xi_2(s,x), \quad\purple{s\in [0,\tau], x\in [0,1]}.
    \end{align}
\end{enumerate}
\end{thm}

\begin{rem}
 \textcolor{black}{The setting in Theorem~\ref{thm:mainfluc} further splits the low-noise regime of Theorem~\ref{thm:mainfluc} into two sub-regimes. The first two regimes correspond to the case when the dominating terms leading to the fluctuations are due to the particle interaction (whose expectation is characterized by the covariance function $A$), and the noise variance respectively. In particular, when the particle interaction dominates, we have two independent diffusion terms; while the covariance of the process $\xi_2$ is determined only by the second-moment function $A$ in Assumption~\ref{ass:continuous data}, the covariance of the process $\xi_3$ appearing in the second diffusion term is determined by both $A$ and the fourth-moments function $B$. The entire limit identification result is provided in Theorem~\ref{thm:masterclttheorem}. }
\end{rem}

\begin{rem}\label{rem:interpolation error dominates} Besides the two sub-regimes introduced in Theorem \ref{thm:mainfluc}, there is another sub-regime, where the \textbf{interpolation error dominates}. Specifically,  under the assumptions in Theorem \ref{thm:mainfluc}, if we assume further that, as $d,T\to\infty $, we have
    \begin{align*}
     d=O(T^{\frac{1}{2}}), \qquad \sigma=O(T^{\frac{1}{2}})\, \qquad\text{and}\qquad   \gamma d^{-1}\to 0,
    \end{align*}
    then for any $\tau\in (0,\infty)$, the fluctuation of SGD particles, $\{U^{d,T}\}_{d\ge 1,T>0}$ converges in distribution to a function $U \in C([0,\tau];C([0,1]))$ as $d,T\to \infty$. Furthermore, for any $s\in [0,\tau]$, $x\in [0,1]$, the limit $U$ is the unique continuous solution to the following ODE (with random initial conditions)
    \begin{align}\label{eq:mainthm fluc SDE CLT large discretization error}
    \mathrm{d} U(s,x)=-\alpha\int_0^1 A(x,y)U(s,y)\mathrm{d}y \mathrm{d}s,, \quad\purple{s\in [0,\tau], x\in [0,1]}.
    \end{align}

This sub-regime comes from having to deal with the approximation of the integral on the right-hand side of \eqref{eq:mainmean ODE LLN} with its Riemann sum, and hence we this regime corresponds to the case when the interpolation error dominates. In this regime, we choose a small order of $\gamma$ to ensure the interpolation error vanishes. Doing so, fluctuations from the particle interaction and noises are suppressed in the limit. Therefore, we obtain the degenerate convergence to the ODE  as in~\eqref{eq:mainthm fluc SDE CLT large discretization error} for the fluctuations. 

However, due to the small order of $\gamma$ and assumptions in Theorem \ref{thm:mainfluc}, the initialization of \eqref{eq:mainthm fluc SDE CLT large discretization error} is trivial, and it makes the solution of \eqref{eq:mainthm fluc SDE CLT large discretization error} also trivial. This indicates that the choice of $\gamma$ of order $o(d)$ is too small, such that it suppresses all the fluctuations in the limit. It is worth mentioning that the interesting case to look at is when $\gamma=\Theta(d)$, in which the interpolation error does not vanish. However, to identify the limit in this case amounts to identifying the limit of the numerical integral error of the Riemann sum, which is beyond the scope the current work.   
\end{rem}

\begin{rem}\label{rem:weak to almost sure} \purple{We emphasize that  for each equation in Theorem \ref{thm:mainmean}, Theorem \ref{thm:mainfluc} and Remark \ref{rem:interpolation error dominates}, there exists a probability space in which the equation holds. The existence of such a probability space follows from the Skorokhod's Representation Theorem. For more details, we refer the readers to Section \ref{sec:proofidea}.}    
\end{rem}

\begin{rem}
Scaling limits and fluctuations developed in Theorems~\ref{thm:mainmean} and~\ref{thm:mainfluc} respectively, also hold (with slight modifications) for the online multiplier bootstrap version of SGD developed in~\cite[Equations (7) and (8)]{fang2018online}; such results may be leveraged for practical high-dimensional statistical inference. 
\end{rem}

\subsection{Related work}\label{sec:comparison} We now place our results in the context of the larger literature on SGD analysis.\\

\textbf{High-dimensional scaling limits of SGD.} \cite{arous2022high} studied the scaling limits of online SGD in the high-dimensional setting for a class of non-convex problems \textcolor{black}{in a statistical, model-based setup}. Their scaling limits are derived for finite-dimensional summary statistics of the online SGD, and no fluctuation results are provided. Furthermore, the precise scaling relationship between the dimension and the number of iterations, and the impact of the data generating process, in particular the covariance structure, is left unexplored. \cite{wang2017scaling} analyzed the online SGD algorithms for least-squares regression and principal component analysis, and derived the scaling limit of the empirical densities as solutions to PDEs. However, their analysis was restricted to the special case of isotropic covariance matrices and no fluctuation results are provided. See also~\cite{wang2017scalinga,veiga2022phase}, for related works on specific models. \textcolor{black}{High-dimensional dynamics of multi-pass (and online) SGD for least-squares problem was considered in \cite{paquette2021dynamics,paquette2022homogenization,paquetteimplicit2022}. Without assuming a true underlying statistical model, it was shown that a certain homogenized
stochastic differential equation characterizes the mean-squared loss precisely under a fairly general (non-Gaussian) assumptions on the input data. However, such works do not identify any phase transition phenomenon (from ballistic to diffusive behavior), are applicable only for a specific quadratic statistics of the SGD iterates, and do not characterize the fluctuations. }

The works by \cite{celentano2021high} and \cite{gerbelot2022rigorous} also characterized the asymptotic behaviors of variants of SGD for a class of optimization problems in the high-dimensional setting. Their approach was based on the so-called \emph{dynamical mean-field theory} from statistical physics. Different from our work, which considers online SGD on expected objective functions, both \cite{celentano2021high} and \cite{gerbelot2022rigorous} considered mini-batch SGD on finite-sum objective functions (or empirical risk minimization). \cite{celentano2021high} require isotropic sub-Gaussian inputs for their analysis. While \cite{gerbelot2022rigorous}  allows for \purple{non-isotropic} covariance, they required Gaussianity assumptions on the inputs. Finally, they only track a real-valued functional of the trajectory as the dimension grows, and no fluctuation results are provided.  

To our knowledge, our work provides the first result on characterizing the entire infinite-dimensional trajectorial limit and related fluctuations of the online SGD, for the specific problem of least-squares regression with smooth covariance structures.
\vspace{0.1in}

\textbf{Other high-dimensional analysis of SGD.} \textcolor{black}{Random matrix theory is also used to analyze gradient-based iterative algorithms ~\citep{deift2019universality,ding2022conjugate}. In the context of machine learning, \cite{paquette2021sgd,paquette2022halting} analyzed full and mini-batch SGD for high-dimensional least-squares linear regression, and established the limits as solutions to the Volterra equation. Results using random matrix theory for full-batch gradient descent dynamics for neural networks has been considered in \cite{bahigh2022}.} \cite{chandrasekher2022alternating,chandrasekher2021sharp} studied mini-batch SGD for certain high-dimensional non-convex problems using Gaussian process techniques. Their work relies heavily on the isotropy and Gaussianity assumptions. State-space approaches for high-dimensional analysis of online SGD were carried out in~\cite{tan2019online} under isotropic Gaussianity assumptions. The work of~\cite{arous2021online} also used a similar approach to establish high-probability bounds in a signal-recovery setup. Recently, high-dimensional normal approximation results and tail bounds are also established in~\cite{agrawalla2023high} and~\cite{durmus2021tight, durmus2022finite} respectively for online least-squares SGD. 

Mean-field analysis of SGD for overparametrized neural networks is also explored intensely in the recent past. While assuming growing parameter dimension, such works, however, assume the data-dimension is fixed. Due to the flurry of recent works in this direction, it is impossible to list them all here. We refer to \cite{chizat2018global, mei2018mean, sirignano2020meana, sirignano2020meanb, pham2021limiting, rotskoff2022trainability, ding2022overparameterization, sirignano2022mean,abbe2022merged,gess2022conservative} for a sampling of such works. 
\vspace{0.1in}

\textbf{Fixed-dimensional analyses of SGD.} 
The study of diffusion limits of SGD in the fixed-dimensional setting is a classical topic. We refer to \cite{kushner2003stochastic,kushner2012stochastic,benveniste2012adaptive,ljung2012stochastic} for a textbook treatment of this topic. The main idea behind such works is to show that appropriately time-interpolated SGD iterates converge to a multi-dimensional Ornstein-Uhlenbeck process, under specific scalings. Recently, \cite{li2019stochastic} developed a framework to approximate the dynamics of a relative general class of stochastic gradient algorithms by stochastic differential equations. See also,~\cite{krichene2017acceleration,li2018near,gupta2020some,fontaine2021convergence} for a partial list of other recent related works. 

Almost sure convergence and central limit theorems for SGD in the fixed-dimensional setting are also well-studied. See~\cite{mertikopoulos2020almost,sebbouh2021almost,liu2022almost} and references therein for almost-sure convergence results. In regard to CLT, we refer to \cite{polyak1992acceleration,ruppert1988efficient,duchi2016local,toulis2017asymptotic, asi2019stochastic,yu2020analysis,dieuleveut2020bridging, barakat2021stochastic, davis2023asymptotic} and references therein for a partial list of related works. We also highlight the works of~\cite{anastasiou2019normal} and~\cite{shao2022berry}, where non-asymptotic normal approximation for SGD is established. For a survey of expectation and high-probability bounds for SGD and its variants, see~\cite{bottou2018optimization}, and~\cite{lan2020first}. 
\vspace{0.1in}

\textbf{Scaling limits of other algorithms.} We remark that high-dimensional scaling limits of iterative sampling algorithms like the Random-Walk Metropolis (RWM) algorithm, Unadjusted Langevin Algorithm (ULA),  and Metropolis Adjusted Langevin Algorithm (MALA) are well-studied. For example,~\cite{pillai2012optimal} and~\cite{mattingly2012diffusion} characterize the scaling limits in the form of infinite-dimensional SDE (or equivalently as stochastic PDEs); see also the references therein for other related works in this direction. While being morally related to our approach, the sampling algorithms studied in those works correspond to a different setup than us, as the interactions are not characterized by any data generating process. More recently, \purple{focusing} on high-dimensional linear regression~\cite{qiu2022tap} and~\cite{mukherjee2022variational} study scaling limits of variational inference algorithms, which are not directly related to SGD. 
\vspace{0.1in}

\textbf{\textcolor{black}{Additional discussion.}} \purple{As discussed previously, in this work, we work in the long-range interaction regime. Alternatively, the case when covariance matrix is ``rough'', which allows for certain degree of abrupt changes between the entries (allowing, for example, various patterns of structured sparsity in $\Sigma_d$) corresponds to the regime of short-range interactions. For last few decades, a large number of works on interacting particle system focused on the connection between the scaling limit of the fluctuation in interacting particle systems and the Kardar-Parisi-Zhang (KPZ) equation. See, for example, \cite{BG97, BS10,quastel2011introduction, corwin2012kardar, CST18,G17,CGST18}. Many of those works demonstrated that the fluctuation of the height function of the associated particle system converges to the KPZ equation under weak noise scaling. A handful set of those works including some recent breakthroughs \citep{MQR21,DOV22,QS23} showed convergence towards the KPZ fixed point under the so called KPZ scaling of space, time and fluctuation. Our fluctuation results in the current work for the long-range interaction case does not reproduce the KPZ equation or the KPZ fixed point in the limit. However, we do believe that in the short-range regime, the scaling limits would lead to those unprecedented limiting behavior even for the seemingly simple problem of solving least-squares with online SGD. In a forthcoming work, we investigate the case of short-range interactions in detail. }

\section{Applications and examples}\label{sec:applications}

As an application of our main results in Theorem~\ref{thm:mainmean} and~\ref{thm:mainfluc}, in this section, we show how they can be leveraged to study certain specific properties of the least-squares online SGD, like the Mean Square Error ($\MSE$) and Predictive Error ($\PrE$). 
\vspace{.2cm}

\noindent \textbf{Mean Squared Error ($\MSE$) and \textbf{Predictive Error ($\PrE$)}:} For a given $d$ and $T$, the time-interpolated mean square error and prediction error of the least-squares online SGD estimator at time $s\in[0,\tau]$ are defined respectively as 
\begin{align*}
&\MSE^{d,T}(s)\coloneqq\frac{1}{d}\sum_{i=1}^d \Bar{\Theta}^{d,T}(\frac{\lfloor sT \rfloor}{T},\frac{i}{d})^2,\\
\PrE^{d,T}(s)&\coloneqq\frac{1}{d^2}\sum_{i,j=1}^d A\big(\frac{i}{d},\frac{j}{d}\big)\Bar{\Theta}^{d,T}\big(\frac{\lfloor sT \rfloor}{T},\frac{i}{d}\big)\Bar{\Theta}^{d,T}\big(\frac{\lfloor sT \rfloor}{T},\frac{j}{d}\big).
\end{align*}

\noindent Specifically, we first calculate the high-dimensional scaling limits of $\MSE$ and $\PrE$ in terms of the solutions to \eqref{eq:mainmean ODE LLN},\eqref{eq:mainmean SDE LLN} and \eqref{eq:mainmean BM LLN}, and show the decay properties of the limiting $\MSE$ and limiting $\PrE$ in the low-noise setup. We next calculate the fluctuations of the $\MSE$ and $\PrE$ in terms of the solutions to \eqref{eq:mainmean ODE LLN}, and \eqref{eq:mainthm fluc SDE CLT},\eqref{eq:mainthm fluc SDE CLT relative large variance} and \eqref{eq:mainthm fluc SDE CLT large discretization error}.

\subsection{Scaling limits and fluctuations of $\MSE$ and $\PrE$}\label{subsec:scaling limit}

Leveraging Theorem \ref{thm:mainmean}, we have the following result on the scaling limits of $\MSE$ and $\PrE$ of the least-squares online SGD. In particular, it exhibits the three-step phase-transition depending on the noise level. To proceed, we defined the limiting $\MSE$ and $\PrE$ as follows: 
\begin{align}\label{eq:msepre}
    \MSE(s)\coloneqq\int_0^1 \Theta(s,x)^2 dx,\qquad \PrE(s)\coloneqq\int_0^1\int_0^1 \Theta(s,x)A(x,y)\Theta(s,y)dxdy.
\end{align}


\begin{prop}\label{prop:general scaling limits} Assume that \textcolor{black}{the assumptions} in Theorem \ref{thm:mainmean} hold. Then, we have the following convergences,  \textrm{(i)} $\MSE^{d,T}(s)\to  \MSE(s)$  and  \textrm{(ii)} $\PrE^{d,T}(s)\to \PrE(s)$ (in probability in the low-noise setting and in distribution in the other two settings), provided one of the following scaling condition is satisfied:
\begin{enumerate}
    \item [\textup{(1)}] \textbf{Low-noise:} $\sigma d^{-1}T^{-\frac{1}{2}}\to 0$, $\eta d T\to \alpha\in (0,\infty)$ as $d,T\to\infty$ and $\Theta$ is the solution to \eqref{eq:mainmean ODE LLN}.
    \item [\textup{(2)}] \textbf{Moderate-noise:} $\sigma d^{-1}T^{-\frac{1}{2}}\to\beta\in (0,\infty)$, $\eta d T\to\alpha\in (0,\infty)$ as $d,T\to\infty$ and $\Theta$ is the solution to \eqref{eq:mainmean SDE LLN}.
    \item [\textup{(3)}] \textbf{High-noise:} $\sigma d^{-1}T^{-\frac{1}{2}}\to \infty$, $\eta \sigma T^{\frac{1}{2}}\to \alpha\in (0,\infty)$ as $d,T\to\infty$ and $\Theta$ is the solution to \eqref{eq:mainmean BM LLN}.
\end{enumerate}
\end{prop}

\noindent We next characterize the fluctuations of $\MSE^{d,T}$ and $\PrE^{d,T}$, i.e., $\gamma\left( \MSE^{d,T}-\MSE \right)$  and $\gamma \left( \PrE^{d,T}-\PrE\right)$, respectively.
\begin{prop}\label{prop:fluctuation of errors} Assume that 
\textcolor{black}{the assumptions} in Theorem \ref{thm:mainmean} and Theorem \ref{thm:mainfluc} hold. Then
\begin{enumerate}
    \item [\textup{(i)}] $\gamma\left( \MSE^{d,T}(s)-\MSE(s) \right)\to 2\int_0^1 \Theta(s,x)U(s,x)\mathrm{d} x$ in distribution,
    \item [\textup{(ii)}] $\gamma\left( \PrE^{d,T}(s)-\PrE(s)\right)\to 2\int_0^1 \int_0^1 \Theta(s,x)A(x,y)U(s,y)\mathrm{d}x\mathrm{d}y$ in distribution,
\end{enumerate}
provided one of the following scaling is satisfied:
\begin{enumerate}
    \item [\textup{(1)}] $\gamma=\zeta T^{\frac{1}{2}}$ for some $\zeta\in (0,\infty)$. $T=o(d^2), \sigma=O(d)$ as $d,T\to\infty$ and $\Theta$, $U$ are the solutions to \eqref{eq:mainmean ODE LLN} and \eqref{eq:mainthm fluc SDE CLT} respectively.  
    \item [\textup{(2)}] $\gamma=\beta \sigma^{-1}dT^{\frac{1}{2}}$ for some $\beta\in(0,\infty)$. $\max(d,T^{\frac{1}{2}})\ll \sigma \ll dT^{\frac{1}{2}}$ as $d,T\to\infty$ and $\Theta$, $U$ are the solutions to \eqref{eq:mainmean ODE LLN} and \eqref{eq:mainthm fluc SDE CLT relative large variance} respectively.
    \item [\textup{(3)}] $1\ll \gamma \ll d$, $d=O(T^{\frac{1}{2}})$, $\sigma=O(T^{\frac{1}{2}})$ as $d,T\to\infty$ and $\Theta$, $U$ are the solutions to \eqref{eq:mainmean ODE LLN} and \eqref{eq:mainthm fluc SDE CLT large discretization error} respectively.
\end{enumerate}
\end{prop}

\subsection{Additional insights and specific example}\label{sec:nonasym error} In this section, we show that under Assumption \ref{ass:tightness in space LLN} (part (a)) and Assumption \ref{ass:continuous data}, the functions $A,B,E$ in Definition \ref{def:white noise} are the limiting descriptions of the corresponding moments of finite-dimensional data. We first show that a piecewise-constant function induced by the finite-dimensional covariance matrix $\Sigma_d$ of the data $\{x^t_i\}_{1\le i\le d}$ converges uniformly to the function $A$ defined in Definition \ref{def:white noise}. Before we state the convergence result, we introduce some necessary definitions and notations.

\begin{defn}\label{def:matrix embedding function}
    Given a symmetric $d\times d$ matrix $\Sigma$ with real entries, define a piecewise-constant function on $[0,1]^2$ by dividing $[0,1]^2$ into $d^2$ smaller squares each of length $1/d$ and set
    \begin{align*}
        W_{\Sigma}(x,y)\coloneqq\Sigma(i,j)\ \text{if}\ \lceil dx \rceil=i,\lceil py \rceil=j. 
    \end{align*}
\end{defn}
\noindent Recalling that $\Sigma_d$ denotes the covariance matrix of data $\{x^t_{i}\}_{1\le i\le d}$, it is easy to see that as $d\to\infty$, under Assumption \ref{ass:tightness in space LLN} and Assumption \ref{ass:continuous data}, we have 
$$
\lv W_{\Sigma_d}-A \rv_\infty \coloneqq \sup_{x,y\in [0,1]} \left| W_{\Sigma_d}(x,y)-A(x,y)\right| \to 0.
$$ 
Indeed, from the Definition \ref{def:matrix embedding function}, we have
\begin{equation*}
    W_{\Sigma_d}(x,y)-A(x,y)=\left\{
    \begin{aligned}
     &0  , \qquad &\lceil dx \rceil,\lceil dy \rceil \in \{0,1,\cdots, d\} \\
     & A\Big(\frac{\lceil dx \rceil}{d},\frac{\lceil dy \rceil}{d}\Big)-A(x,y) , \qquad &\text{otherwise}.
    \end{aligned}
    \right.
\end{equation*}
 Due to Assumption \ref{ass:continuous data}, we have
 \begin{align*}
    &\Big| A\Big(\frac{\lceil dx \rceil}{d},\frac{\lceil dy \rceil}{d}\Big)-A(x,y)\Big|\\
    \le &\Big|A\Big(\frac{\lceil dx \rceil}{d},\frac{\lceil dy \rceil}{d}\Big)-A\Big(x,\frac{\lceil dy \rceil}{d}\Big)  \Big|+\big|A\Big(x,\frac{\lceil dy \rceil}{d}\Big)-A(x,y)\big|\\
    \le& \frac{2C_{3}}{d}.
 \end{align*}
 Therefore $\left| W_{\Sigma_d}(x,y)-A(x,y)\right|\le 2C_{3} d^{-1}$ for all $x,y\in [0,1]$. The claimed convergence result hence follows from definition of $\lv\cdot\rv_\infty$. Similarly, we can extend Definition \ref{def:matrix embedding function} to any $m$-tensor and study the uniform convergence for functions defined $[0,1]^m$ with any $m\in \mb{N}$. In that way we can interpret $B$ and $E$ in Definition \ref{def:white noise} as the $C([0,1]^4)$-limit and the $C([0,1]^8)$-limit of the functions corresponding to the fourth moment tensor and the eighth moment tensor of $x^t$ respectively.

We now provide a concrete example of a model that satisfies our Assumption \ref{ass:continuous data}. Recall that we consider the data $\{x^t_i\}_{t\ge 1,i\in[d]}$ as being generated as the discretizations of the process $\{W(s,x)\}_{s\in[0,\tau],x\in [0,1]}$ defined in Definition \ref{def:white noise}. 
\begin{sinusoid}\label{ass:data example} For any $s\ge 0$, let $\{W(s,x)\}_{x\in [0,1]}$ be a centered stochastic process with the covariance function $A$ given by 
  \begin{align}\label{eq:Sinusoidal graphon A}
  A(x,y)=a_0+  \sum_{k=1}^\infty b_k \cos\left( 2\pi k \left( x-y \right) \right),\qquad \forall\ x,y\in [0,1].
 \end{align}
 such that there exists $\varepsilon>0$ such that
\begin{align}\label{eq:Sinusoidal graphon A parameter}
    \sum_{k=1}^\infty k^{5+\varepsilon} b_k^2<\infty.
\end{align}
\end{sinusoid}
\noindent Define the finite dimensional data covariance matrix as    \begin{align}\label{eq:Sinusoidal data covariance}        \Sigma_d(i,j)=a_0+\sum_{k=1}^\infty b_k \cos\left( 2\pi k (i-j)/d \right),\qquad \forall \ i,j\in [d].
    \end{align}
Let $A$ be as defined in \eqref{eq:Sinusoidal graphon A} and $W_{\Sigma_d}$ be as defined in Definition \ref{def:matrix embedding function}. When \eqref{eq:Sinusoidal graphon A parameter} holds, it is easy to see that 
$\lv A-W_{\Sigma_d} \rv_{\infty} \to 0$ as $d\to\infty$. Condition \eqref{eq:Sinusoidal graphon A parameter} guarantees that $A$ defined in \eqref{eq:Sinusoidal graphon A} satisfies Assumption \ref{ass:continuous data}.
    
If we further assume that $W(s,\cdot)$ is a Gaussian process, we have relatively easy expressions for the fourth and eighth moments, thanks to  Isserlis' Theorem. That is, for any $k\in \mb{N}$, we have 
    \begin{align}\label{eq:Isserlis thm}
        \mb{E}\big[ \prod_{i=1}^{2k} W(s,x_i) \big]=\sum_{p\in P_{2k}^2} \prod_{(i,j)\in p} \mb{E}\big[ W(s,x_i)W(s,x_j) \big],
    \end{align}
    where $P_{2k}^2$ is the set of all pairings of $\{1,\cdots, 2k\}$. Based on this, we have the following result.
\begin{lemma}\label{lem:Sinusoidal graphon A} The function $A$ defined by \eqref{eq:Sinusoidal graphon A} and \eqref{eq:Sinusoidal graphon A parameter} satisfies Assumption \ref{ass:continuous data}. Furthermore, if $W$ is Gaussian, its fourth moment $B$ and eighth moment $E$ satisfy Assumption \ref{ass:continuous data}. 
\end{lemma}
We emphasize here that the Gaussian assumption in Lemma~\ref{lem:Sinusoidal graphon A} is purely made for the sake of simplicity. Based on generalizations of Isserlis' theorem, computations similar to those required to prove Lemma~\ref{lem:Sinusoidal graphon A} could be carried out in the elliptical setting~\citep{zuo2021expressions}, mixture of Gaussian setting~\citep{vignat2012generalized} and beyond.

Next, we provide more explicit decay properties of the scaling limits that govern the limiting behavior of the finite-dimensional and finite-iteration $\MSE$ and $\PrE$ in the low-noise setup.

\begin{prop}\label{prop:exp decay of MSE} 
Invoke the assumptions in 
Theorem \ref{thm:mainmean}, under the low-noise set up.
\begin{itemize} 
    \item [\textup{(a)}]  
We have that $\PrE(\tau)\le \MSE(0)/\alpha \tau$, for all $\tau\in(0,\infty)$.
    \item [\textup{(b)}] If we further assume that for an orthonormal basis $\{\phi_i\}_{i=1}^\infty$ of $L^2([0,1])$, $A:[0,1]^2\to\mb{R}$ admits the following decomposition \textcolor{black}{$A(x,y)=\sum_{i=1}^K \lambda_i \phi_i(x) \phi_i(y)$}, where the sum converges in $L^2([0,1]^2)$ and \textcolor{black}{$\lambda_1\ge \lambda_2\ge \cdots \ge \lambda_K$, with $\lambda\coloneqq\lambda_K >0$}, we have that $\MSE(\tau)\le \MSE(0) \exp(-2\alpha \lambda \tau )$, for all $\tau\in (0,\infty)$.
\end{itemize}
\end{prop}

\noindent In the following, we provide two specific examples illustrating the above propositions.
\begin{enumerate}
    \item [(1)] Let covariance matrix be defined as $$\Sigma_d(i,j)=1+\cos\left( 2\pi (i-j)/d \right)\quad\text{for all}\quad i,j\in [d].$$ 
    Then, the function $A$ is given by $A(x,y)=1+\cos(2\pi (x-y))$ for all $x,y\in [0,1]$. 
    Note that, $A$ could also be represented as
    \begin{align*}
        A(x,y)=1+\frac{1}{2}\big(\sqrt{2}\cos(2\pi x)\big)\big(\sqrt{2}\cos(2\pi y)\big)+\frac{1}{2}\big(\sqrt{2}\sin(2\pi x)\big)\big(\sqrt{2}\sin(2\pi y)\big).
    \end{align*}
    Therefore $A$ satisfies the conditions in part (b) of Proposition \ref{prop:exp decay of MSE} with $\lambda_1=1,\lambda_2=\lambda_3={1}/{2}$. Hence, we  look at the $\MSE$ as defined in~\eqref{eq:msepre}. According to Proposition \ref{prop:exp decay of MSE}, the scaling limit of the $\MSE$ satisfies $\MSE(\tau)\le \MSE(0)\exp(-\alpha \tau)$ for any $\tau\in (0,\infty)$.
    \item [(2)]  Now, let the covariance matrix be given by 
    $$
    \Sigma_d(i,j)=\Big( \frac{\left| i-j \right|}{d}-\frac{1}{2} \Big)^2-2\Big( \frac{\left| i-j \right|}{d}-\frac{1}{2} \Big)^4\quad \text{for all}\quad i,j\in [d].
    $$
    Then the function $A$ is given by $A(x,y)=\left( \left| x-y \right|-\frac{1}{2} \right)^2-2\left( \left| x-y \right|-\frac{1}{2} \right)^4$ for all $x,y\in [0,1]$. By a Fourier series expansion, we also have that 
    \begin{align*}
        A(x,y)&=a_0+\sum_{k=1}^\infty \frac{b_k}{2} \big(\sqrt{2}\cos(2\pi x)\big)\big(\sqrt{2}\cos(2\pi y)\big)\\
        &\qquad+\sum_{k=1}^\infty \frac{b_k}{2} \big(\sqrt{2}\sin(2\pi x)\big)\big(\sqrt{2}\sin(2\pi y)\big),
    \end{align*}
where $a_0={7}/{120}$ and $b_k={6}/{\pi^2 k^4}$ for all $k\ge 1$. It can also be checked that \eqref{eq:Sinusoidal graphon A parameter} is satisfied with $\varepsilon=1$. As there is no positive uniform lower bound of the eigenvalues, we look at the $\PrE$ as defined in~\eqref{eq:msepre}. 
    According to part (a) of Proposition \ref{prop:exp decay of MSE}, we have for any $\tau\in (0,\infty)$, $\PrE(s)\le \MSE(0)/\alpha \tau$.
\end{enumerate}
\noindent More generally, to compute the scaling limits and the fluctuations, in particular in the moderate and high-noise setups, one invariably has to resort to numerical procedures. We refer, for example, to~\cite{lord2014introduction} regarding the details, and leave a thorough investigation of this as future work.

\section{Proof overview}\label{sec:proofidea}

\purple{The high level idea behind our proofs is to study the interpolated SGD iterates/fluctuations in~\eqref{eq:approximation to Theta LLN} as sequences of random variables in the space $C([0,\tau];C([0,1]))$. Under appropriate assumptions, we first prove that the sequences of random variables are tight in $C([0,\tau];C([0,1]))$; this forms a major portion of our analysis. As a consequence, we can interpret different terms in the algorithm as sequences of random variables and each subsequence has a further subsequence that converges jointly in distribution with the sequential weak limit in $C([0,\tau];C([0,1]))^{\otimes{m}}$($m$ is the number of terms in the algorithm).}

\purple{Next, we identify the jointly sequential weak limits by examining the sequential weak limit of each term. Thanks to the Skorokhod representation theorem, the joint sequential weak limits can be represented by an ODE/SDE whose solution equals to sequential weak limit of the interpolated SGD iterates/fluctuations in distribution. Last, we prove that all the sequential weak limits are the same and therefore the sequences of interpolated SGD iterates/fluctuations converge weakly to unique limits by showing the uniqueness of path-wise solutions to the ODE/SDEs. For the convenience of readers, we stated the Skorokhod Representation Theorem in the following.
\begin{thm}[Skorokhod's Representation Theorem (Theorem 6.7 in \cite{billingsley2013convergence})]\label{them:sko thm} Suppose $\{\mb{P}_n\}_{n\ge 1}$ and $\mb{P}$ are probability measures on $\mb{R}^d$ (provided with its Borel $\sigma$-algebra) such that $\{\mb{P}_n\}_{n\ge 1}$ converges to $\mb{P}$ weakly.
Then there is a probability space $(\Omega, \mc{F}, \mb{P})$ on which are defined $\mb{R}^d$-valued random
variables $\{X_n\}_{n\ge 1}$ and $X$ with distributions $\{\mb{P}_n\}_{n\ge 1}$ and $\mb{P}$ respectively, such
that $X_n$ converges to $X$ almost surely.    
\end{thm}
}

\subsection{For Theorem~\ref{thm:mainmean} }\label{sec:KTC}

Our first result shows the tightness of $ \{\Bar{\Theta}^{d,T}(\cdot,\cdot)\}_{d\ge 1, T> 0}$ in the space ${C}\left([0,\tau]; {C}([0,1])\right)$. It is based on two technical results (see Proposition \ref{prop:L2 tightness in time LLN} and Proposition \ref{prop:L2 tightness in time and space LLN}) which forms the major part of our analysis.
\begin{thm}[Tightness]\label{thm:tightness of continuous interpolation LLN}
Let Assumptions \ref{ass:tightness in space LLN} and \ref{ass:continuous data} hold, and the initial conditions in Theorem \ref{thm:mainmean} are satisfied. Further, suppose that there is a uniform constant $C_{s,1}>0$ such that $\max(\eta d T,\eta \sigma T^{\frac{1}{2}})\le C_{s,1}$. Then $ \{\Bar{\Theta}^{d,T}\}_{d\ge 1, T> 0}$ defined in \eqref{eq:approximation to Theta LLN} is tight in ${C}\left([0,\tau]; {C}([0,1])\right)$. Hence, any subsequence of $ \{\Bar{\Theta}^{d,T}\}_{d\ge 1, T> 0}$ has a further weakly convergent subsequence with its limit in ${C}\left([0,\tau]; {C}([0,1])\right)$.
\end{thm}


\begin{thm}[Limit Identification]\label{thm:LLN} Invoke the conditions in Theorem \ref{thm:tightness of continuous interpolation LLN}. Then for any $\tau>0$ and any subsequence $\{\Theta^{d_k,T_k}\}_{k\ge 1}$ of $\{\Theta^{d,T}\}_{d\ge 1,T>0}$, there further exists a  subsequence converging weakly to a function $\Theta\in C([0,\tau];C([0,1]))$ as $d_k,T_k\to \infty$. Furthermore, for any bounded smooth function $f:[0,\tau]\to \mb{R}$, for any $s\in [0,\tau]$, $x\in [0,1]$, we have 
\begin{align*}
    &\qquad -\int_0^s \Theta( u ,x)  f'(u) \mathrm{d}u+ f(s){\Theta}(s,x)-f(0){\Theta}(0,x) \nonumber \\
    &=-\eta d T\big( \int_0^{s}\int_0^1 f(u)A(x,y)\Theta(u,y)\mathrm{d}y\mathrm{d}u+o(1)\big)+\sigma^2 \eta^2 T \big( \int_0^s f(u)\mathrm{d}\xi_1(u,x)+o(1)\big).
\end{align*}
where $\{\xi_1(s,x)\}_{s\in [0,\tau],x\in [0,1]}$ is a Gaussian field process with covariance given by \eqref{eq:gaussian field variance high noise LLN}.
\end{thm}
\hspace{0.1in}

\noindent The above theorem is proved in Section \ref{sec:LLN proofs}. As an immediate consequence, we have that:
\begin{itemize}
    \item [(a)] When $\eta d T\to\alpha$ and $\sigma^2/d^2 T\to 0$ as $d,T\to\infty$, we get
    \begin{align}\label{eq:LLN weak form}
        \hspace{-0.2in}-\int_0^s \Theta( u ,x)  f'(u) \mathrm{d}u+ f(s){\Theta}(s,x)-f(0){\Theta}(0,x) =-\alpha \int_0^{s}\int_0^1 f(u)A(x,y)\Theta(u,y)\mathrm{d}y\mathrm{d}u.
    \end{align}
    \item [(b)] When $\eta d T\to\alpha$ and $\sigma^2/d^2T\to \beta^2$ as $d,T\to\infty$, we get
    \begin{align*}
       &\qquad -\int_0^s \Theta( u ,x)  f'(u) \mathrm{d}u+ f(s){\Theta}(s,x)-f(0){\Theta}(0,x) \\
    &=-\alpha\int_0^{s}\int_0^1 f(u)A(x,y)\Theta(u,y)\mathrm{d}y\mathrm{d}u+\alpha^2\beta^2 \int_0^s f(u)\mathrm{d}\xi_1(u,x).
    \end{align*}
    \item [(c)] When $\eta \sigma T^{\frac{1}{2}}\to\alpha$ and $\sigma^2/d^2T\to\infty$ as $d,T\to\infty$, we get
        \begin{align*}
       & -\int_0^s \Theta( u ,x)  f'(u) \mathrm{d}u+ f(s){\Theta}(s,x)-f(0){\Theta}(0,x) =\alpha^2 \int_0^s f(u)\mathrm{d}\xi_1(u,x).
    \end{align*}
\end{itemize}
Our main result in Theorem~\ref{thm:mainmean} follows from the above results; proof is provided in Section \ref{sec:LLN proofs}.

\subsection{For Theorem~\ref{thm:mainfluc}}\label{sec:fluctuation scaling limit}

We next turn to proving the fluctuation results. From the definition of our interpolation process in Section~\ref{sec:Notations}, we immediately have that $U^{d,T}\in C([0,\tau];C([0,1]))$. The idea of deriving the limit of $\{U^{d,T}\}_{d\ge 1,T>0}$ is similar to the idea of deriving the limit of $\{\Bar{\Theta}^{d,T}\}_{d\ge 1,T>0}$. Specifically, we first have the following tightness result.

\begin{thm}[Tightness]\label{thm:tightness of continuous interpolation CLT}
Let Assumptions \ref{ass:tightness in space LLN} and \ref{ass:continuous data} hold, and further suppose the initial conditions in Theorem \ref{thm:mainmean} and Theorem \ref{thm:mainfluc} are satisfied. If $\eta=\frac{\alpha}{dT}$, $\sigma=o(dT^{\frac{1}{2}})$ and there exists a uniform constant $C_{s,2}$ such that $\max(\gamma T^{-\frac{1}{2}},\gamma d^{-1},\gamma \sigma d^{-1}T^{-\frac{1}{2}} )\le C_{s,2}$ for all $d\ge 1,T>0$, then $ \{U^{d,T}(\cdot,\cdot)\}_{d\ge 1, T> 0}$ as defined in \eqref{eq:fluctuation} is tight in ${C}\left([0,\tau]; {C}([0,1])\right)$. Hence, any subsequence of $ \{U^{d,T}(\cdot,\cdot)\}_{d\ge 1, T> 0}$ has a further weakly convergent subsequence with limit in ${C}\left([0,\tau]; {C}([0,1])\right)$.
\end{thm}

\begin{thm}[Limit Identification]\label{thm:masterclttheorem}
Invoke the conditions in Theorem \ref{thm:tightness of continuous interpolation CLT}. For any $\tau>0$ and any subsequence $\{U^{d_k,T_k}\}_{k\ge 1}$ of $\{U^{d,T}\}_{d\ge 1,T>0}$, there exists a further subsequence converging weakly to a stochastic function $U\in C([0,\tau];C([0,1]))$ as $d_k,T_k\to\infty$. Furthermore, for any bounded smooth function $f:[0,\tau]\to\mb{R}$, 
 \purple{if there exist some uniform non-negative constants $\zeta,\beta$ such that $\gamma \sigma d^{-1}T^{-\frac{1}{2}}\to \beta$ and $\gamma T^{-\frac{1}{2}}\to \zeta$, $U$ satisfies that
for any $s\in [0,\tau],x\in [0,1]$,
    \begin{align*}
        &\qquad - \int_0^s U(u,x)f'(u)\mathrm{d}u+f(s)U(s,x)-f(0)U(0,x)\\
    &=-\alpha \int_0^s \int_0^1 f(u)A(x,y)U(u,y)\mathrm{d}y\mathrm{d}u+\alpha \beta \int_0^s f(u) \mathrm{d}\xi_2(u,x) \\
    &\quad+\alpha \zeta \int_0^s f(u) \mathrm{d}\xi_3(u,x) 
    + O(\gamma d^{-1}+\gamma T^{-1}) +o(1).
    \end{align*}}
\end{thm}
\vspace{0.1in}

\noindent The above theorem is proved in Section~\ref{sec:cltproof}. As an immediate consequence, we have that:

\begin{itemize}
    \item [(a)] When $T=o(d^2)$, $\sigma=O(d)$ and $\gamma T^{-\frac{1}{2}}\to\zeta\in (0,\infty) $, we can show that $O(\gamma d^{-1}+\gamma T^{-1})=o(1)$. Therefore, we get 
    \begin{align*}
     &-\alpha^{-1} \big(\int_0^s U(u,x)f'(u)\mathrm{d}u+f(s)U(s,x)-f(0)U(0,x)\big) \nonumber \\
    =&-\int_0^s \int_0^1 f(u)A(x,y)U(u,y)\mathrm{d}y\mathrm{d}u+ \beta \int_0^s f(u) \mathrm{d}\xi_2(u,x)  \\
    &+\zeta \int_0^s f(u) \mathrm{d}\xi_3(u,x) .
    \end{align*}
    \item [(b)] When $\max({d,T^{\frac{1}{2}}})\ll \sigma \ll dT^{\frac{1}{2}}$ and $\gamma \sigma d^{-1}T^{-\frac{1}{2}}\to \beta\in(0,\infty)$, we can show that $\gamma T^{-\frac{1}{2}}\to 0$ and $O(\gamma d^{-1}+\gamma T^{-1})=o(1)$. Therefore, we get
    \begin{align*}
        &\qquad - \alpha^{-1} \big(\int_0^s U(u,x)f'(u)\mathrm{d}u+f(s)U(s,x)-f(0)U(0,x)\big) \nonumber \\
    &=- \int_0^s \int_0^1 f(u)A(x,y)U(u,y)\mathrm{d}y\mathrm{d}u+\beta \int_0^s f(u) \mathrm{d}\xi_2(u,x) .
    \end{align*}
    \item [(c)] When $d=O(T^{\frac{1}{2}})$, $\sigma=O(T^{\frac{1}{2}})$ and $\gamma d^{-1}\to 0$, we can show that $\gamma \sigma d^{-1} T^{-\frac{1}{2}}\to 0$, $\gamma T^{-\frac{1}{2}}\to 0$ and $O(\gamma d^{-1}+\gamma T^{-1})=o(1)$. Therefore, we get
    \begin{align}\label{eq:large discretization error}
        &- \int_0^s U(u,x)f'(u)\mathrm{d}u+f(s)U(s,x)-f(0)U(0,x) \\
        =&-\alpha \int_0^s \int_0^1 f(u)A(x,y)U(u,y)\mathrm{d}y\mathrm{d}u .\nonumber
    \end{align}
\end{itemize}
\textcolor{black}{Our main results in Theorem~\ref{thm:mainfluc} and Remark \ref{rem:interpolation error dominates} follows from the above result, whose proof is provided in Section \ref{sec:FluctuationProofs}}. 

\subsection{Existence and uniqueness of the solutions}
To fully complete the proofs of our main results in Theorems~\ref{thm:mainmean} and~\ref{thm:mainfluc}, we also need to provide the following existence and uniqueness results on the solutions to the corresponding SDEs. The ODE case follows by Picard-Lindel\"{o}f theorem (see, for example,~\cite{arnold1992ordinary}). Furthermore, since \eqref{eq:mainmean SDE LLN}, \eqref{eq:mainmean BM LLN}, \eqref{eq:mainthm fluc SDE CLT relative large variance} and \eqref{eq:mainthm fluc SDE CLT large discretization error} can be considered as certain degenerate forms of \eqref{eq:mainthm fluc SDE CLT}, we specifically focus on the existence and uniqueness of solution to \eqref{eq:mainthm fluc SDE CLT}. Results similar to Proposition~\ref{prop:existence and uniqueness of SDE} below also hold for \eqref{eq:mainmean SDE LLN}, \eqref{eq:mainmean BM LLN}, \eqref{eq:mainthm fluc SDE CLT relative large variance} and \eqref{eq:mainthm fluc SDE CLT large discretization error}.

\begin{thm}\label{prop:existence and uniqueness of SDE} Let Assumption \ref{ass:continuous data} hold.
\begin{enumerate}
    \item If the initial condition of \eqref{eq:mainthm fluc SDE CLT} satisfies $U(0,\cdot)\in C([0,1])$, then there exists a unique solution $\{U(s,x)\}_{s\in [0,\tau],x\in[0,1]}\in C([0,\tau];C([0,1]))$ to the SDE \eqref{eq:mainthm fluc SDE CLT}.
    \item If the initial condition of \eqref{eq:mainthm fluc SDE CLT} satisfies $U(0,\cdot)\in L^2([0,1])$, then there exists a unique solution $\{U(s,x)\}_{s\in [0,\tau],x\in[0,1]}\in C([0,\tau];L^2([0,1]))$ to the SDE \eqref{eq:mainthm fluc SDE CLT}. Furthermore, the solution satisfies the following stability property:
 \begin{align}\label{eq:L2 estimation SDE solution}
\begin{aligned}
    \mb{E}\big[ \sup_{s\in [0,\tau]} \lv U(s,\cdot) \rv_{L^2([0,1])}^2 \big] &\le 4\bigg(\mb{E}\big[ \lv U(0,\cdot) \rv_{L^2([0,1])}^2\big] +C_2\alpha^2 \beta^2 \tau \\
     &+(C_{5}+C_2^2)\alpha^2\zeta^2 \int_0^\tau  \lv \Theta(s,\cdot) \rv_{L^2([0,1])}^2 \mathrm{d}s  \bigg)\exp\left(4C_2^2\alpha^2 \tau^2\right).
\end{aligned}
\end{align}
\end{enumerate}
\end{thm}
We remark that the above results not only provide theoretical support but also provide the necessary conditions for provably computing the solutions numerically; see, for example,~\cite{lord2014introduction} for details.

\section{Proofs for the scaling limits}\label{sec: Proofs}

\subsection{Tightness of $\{\Bar{\Theta}^{d,T}\}_{d\ge 1,T>0}$}\label{sec:tightness proofs LLN}

We first start with the following two main propositions. 

\begin{prop}\label{prop:L2 tightness in time LLN} Assume that Assumption \ref{ass:tightness in space LLN} and Assumption \ref{ass:continuous data} hold and there is a uniform constant $C_{s,1}>0$ such that \textcolor{black}{for all $d\ge 1,T>0$ along the limiting sequence},  $\max(\eta d T,\eta \sigma T^{\frac{1}{2}})\le C_{s,1}$. Then there exists a positive constant $C$ depending on $\tau$ and $C_{s,1}$ such that for any $d,T>0$ and $i\in [d]$ and any $0\le t_1<t_2\le \lfloor{\tau T}\rfloor$,
\begin{align}\label{eq:L2 tightness in time LLN}
    \mb{E}\big[\left|\dtheta^{t_2}_i-\dtheta^{t_1}_i\right|^4\big]\le C  \left(\frac{t_2-t_1}{T}\right)^2 .
\end{align}
\end{prop}

\begin{proof}[Proof of Proposition \ref{prop:L2 tightness in time LLN}]
 For any $0\le t_1 < t_2 \le \lfloor{\tau T}\rfloor\coloneqq N$, we have
\begin{align*}
&~\dtheta^{t_2}_i-\dtheta^{t_1}_i\\
 =&-\eta \sum_{t=t_1}^{t_2-1} \sum_{j=1}^d \mathbb{E}[x^t_ix^t_j] \dtheta^t_j  -\eta \sum_{t=t_1}^{t_2-1} \sum_{j=1}^d \left( x^t_ix^t_j-\mathbb{E}[x^t_ix^t_j]\right) \dtheta^t_j+\eta \sum_{t=t_1}^{t_2-1} x^t_i \varepsilon^t,
\end{align*}
which implies that
\begin{align}\label{eq:L2 difference in time}
\begin{aligned}
    \left| \dtheta^{t_2}_i-\dtheta^{t_1}_i \right|^4
 &\le 27\eta^4 \underbrace{\big(\sum_{t=t_1}^{t_2-1} \sum_{j=1}^d \mathbb{E}[x^t_ix^t_j] \dtheta^t_j\big)^4}_{S_{i,1}}\\
 &\quad+27\eta^4\underbrace{\big( \sum_{t=t_1}^{t_2-1} \sum_{j=1}^d \left( x^t_ix^t_j-\mathbb{E}[x^t_ix^t_j]\right) \dtheta^t_j\big)^4}_{S_{i,2}}+27\eta^4 \underbrace{ \big(\sum_{t=t_1}^{t_2-1} x^t_i \varepsilon^t \big)^4}_{S_{i,3}}.
\end{aligned}
\end{align}  
Summing over the space index $i$ and taking expectation, we get
\begin{align}\label{eq:difference in time and sum in space}
   \frac{1}{d}\sum_{i=1}^d  \mb{E}\big[ |\dtheta^{t_2}_i-\dtheta^{t_1}_i|^4 \big] & \le \frac{27\eta^4}{d} \sum_{i=1}^d \mb{E}\big[ S_{i,1}\big] +\frac{27\eta^4}{d}  \sum_{i=1}^d\mb{E}\big[ S_{i,2}\big]+\frac{27\eta^4}{d}\sum_{i=1}^d \mb{E}\big[S_{i,3}\big].
\end{align}
Next we will estimate the right-hand side of \eqref{eq:difference in time and sum in space} term by term. Define $$m_t^d\coloneqq\frac{1}{d}\sum_{i=1}^d \mb{E}\left[|\dtheta^t_i|^4\right].$$ First we have
\begin{align*}
   \mb{E}\big[S_{i,1}\big]&= \mb{E}\big[ \sum_{r_1,r_2,r_3,r_4=t_1}^{t_2-1}\sum_{j_1,j_2,j_3,j_4=1}^d \prod_{k=1}^4 A(\frac{i}{d},\frac{j_k}{d})\prod_{k=1}^4 \dtheta^{r_k}_{j_k} \big]\\
   &\le C_2^4 \mb{E}\big[ \sum_{r_1,r_2,r_3,r_4}\sum_{j_1,j_2,j_3,j_4} \prod_{k=1}^4 \purple{|}\dtheta^{r_k}_{j_k} \purple{|}\big] \\
   &\le C_2^4 d^4 (t_2-t_1)^3 \sum_{t=t_1}^{t_2-1} m^d_t ,
\end{align*}
where the first inequality follows from Assumption \ref{ass:continuous data} and the last inequality follows from $abcd\le \frac{1}{4}(a^4+b^4+c^4+d^4)$ for any $a,b,c,d\in\mb{R}$. We also have that 
\begin{align*}
    &\quad \mb{E}\big[ S_{i,2}\big]= \mb{E}\big[ \sum_{r_1,r_2,r_3,r_4=t_1}^{t_2-1}\sum_{j_1,j_2,j_3,j_4=1}^d \prod_{k=1}^4 \big(x^{r_k}_{i}x^{r_k}_{j_k}-\mb{E}[x^{r_k}_{i}x^{r_k}_{j_k}]\big)\prod_{k=1}^4 \dtheta^{r_k}_{j_k}\big]. 
\end{align*}
Most terms in the above sum are zeros due to the independence between $x^t$ and $x^{t'}$ when $t\neq t'$. There are two types of terms would be preserved. To proceed, we let $\mc{F}_k^{\purple{d}}$ to be the $\sigma$-algebra generated by $\left\{ \{x^t_i\}_{i\in [d], 0\le t\le k}, \{ \varepsilon^t \}_{0\le t\le k}, \{\dtheta^{\purple{0}}_i\}_{i\in [d]} \right\}$.
\vspace{0.1in}

\textbf{Type 1}: When $r_1=r_2=r_3=r_4$, nonzero terms are in the form of 
\[
\sum_{j_1,j_2,j_3,j_4=1}^d \mb{E}\big[\prod_{k=1}^4 \left(x^{t}_{i}x^{t}_{j_k}-\mb{E}\left[x^{t}_{i}x^{t}_{j_k}\right]\right)\prod_{k=1}^4\dtheta^{t}_{j_k}\big], 
\]
and there are $3(t_2-t_1)$ such terms. The sum of such terms can be upper bounded as
\begin{align*}
  &\qquad 3\sum_{t=t_1}^{t_2-1}\sum_{j_1,j_2,j_3,j_4=1}^d \mb{E}\big[\prod_{k=1}^4 \left(x^{t}_{i}x^{t}_{j_k}-\mb{E}\left[x^{t}_{i}x^{t}_{j_k}\right]\right)\prod_{k=1}^4\dtheta^{t}_{j_k}\big]  \\
  &=3\sum_{t=t_1}^{t_2-1} \sum_{j_1,j_2,j_3,j_4=1}^d \mb{E}\bigg[\mb{E}\big[\prod_{k=1}^4 \left(x^{t}_{i}x^{t}_{j_k}-\mb{E}\left[x^{t}_{i}x^{t}_{j_k}\right]\right)\prod_{k=1}^4\dtheta^{t}_{j_k}\big] \big| \purple{\mc{F}_{t-1}^{d}} \bigg] \\
  &\le 3C_8' \sum_{t=t_1}^{t_2-1}\sum_{j_1,j_2,j_3,j_4=1}^d   \mb{E}\big[\prod_{k=1}^4\purple{|}\dtheta^{t}_{j_k}\purple{|}\big]\le C_8' d^4 \sum_{t=t_1}^{t_2-1}m^d_t,
\end{align*}
where the first inequality follows from Assumption \ref{ass:continuous data} and $C_8'=16C_8+32C_2^2C_{5}+16C_2^4$. The last inequality follows from $abcd\le \frac{1}{4}(a^4+b^4+c^4+d^4)$ for any $a,b,c,d\in\mb{R}$.
\vspace{0.1in}

\textbf{Type 2}: When $r_1,r_2,r_3,r_4$ are pairwise equal, nonzero terms are in the form of
\[
\sum_{j_1,j_2,j_3,j_4=1}^d \mb{E}\big[\prod_{k=1}^2 \big(x^{t}_{i}x^{t}_{j_k}-\mb{E}\big[x^{t}_{i}x^{t}_{j_k}\big]\big)\prod_{k=3}^4 \big(x^{t'}_{i}x^{t'}_{j_k}-\mb{E}\big[x^{t'}_{i}x^{t'}_{j_k}\big]\big)\prod_{k=1}^2 \dtheta^{t}_{j_k}\prod_{k=3}^4 \dtheta^{t'}_{j_k}\big] 
\]
with $t\neq t'$ and there are $3(t_2-t_1)^2-3(t_2-t_1)$ such terms. For the simplicity of notations, we denote $\Delta X_{ij}^t\coloneqq x^{t}_{i}x^{t}_{j}-\mb{E}\big[x^{t}_{i}x^{t}_{j}\big]$ for all $i,j\in [d]$ and $t\le \lfloor \tau T \rfloor$. Then, the sum of such terms can be upper bounded as 
\begin{align*}
    &\qquad3\sum_{t\neq t'}\sum_{j_1,j_2,j_3,j_4=1}^d \mb{E}\big[\prod_{k=1}^2 \Delta X_{ij_k}^t\prod_{k=3}^4 \Delta X_{ij_k}^{t'}\prod_{k=1}^2 \dtheta^{t}_{j_k}\prod_{k=3}^4 \dtheta^{t'}_{j_k}\big] \\
    &=3\sum_{t\neq t'}\sum_{j_1,j_2,j_3,j_4=1}^d\mb{E}\bigg[  \mb{E}\big[\prod_{k=1}^2 \Delta X_{ij_k}^t\prod_{k=3}^4 \Delta X_{ij_k}^{t'}\prod_{k=1}^2 \dtheta^{t}_{j_k}\prod_{k=3}^4 \dtheta^{t'}_{j_k}\big] | \purple{\mc{F}_{t\vee t'-1}^{d}} \bigg]\\
    &\le 3C_{5}'^2\sum_{t\neq t'} \sum_{j_1,j_2,j_3,j_4=1}^d \mb{E}\big[\prod_{k=1}^2 \purple{|}\dtheta^{t}_{j_k}\purple{|}\prod_{k=3}^4 \purple{|}\dtheta^{t'}_{j_k}\purple{|}\big]\le 3C_{5}'^2 d^4 (t_2-t_1)\sum_{t=t_1}^{t_2-1}  m^d_t,
\end{align*}
where the first inequality follows from Assumption \ref{ass:continuous data} and $C_{5}'=C_{5}+C_2^2$. The last inequality follows from $4(abcd)\le (a^4+b^4+c^4+d^4)$ for any $a,b,c,d\in\mb{R}$. Combining the two upper bounds, we have
\begin{align*}
    \mb{E}\big[ S_{i,2}\big]\le 3\big( C_8'+C_{5}'^2 (t_2-t_1) \big)d^4\sum_{t=t_1}^{t_2-1}  m^d_t \le \big( C_8'+C_{5}'^2 \big) (t_2-t_1)d^4\sum_{t=t_1}^{t_2-1}  m^d_t .
\end{align*}
Last, due to Assumption \ref{ass:continuous data}, we estimate $\mb{E}[S_{i,3}]$ as follows:
\begin{align*}
    \mb{E}[S_{i,3}]&= \sum_{t=t_1}^{t_2-1}B(\frac{i}{d},\frac{i}{d},\frac{i}{d},\frac{i}{d}) \mb{E}\big[ (\varepsilon^t)^4 \big]+\sum_{t\neq t'} A(\frac{i}{d},\frac{i}{d})A(\frac{i}{d},\frac{i}{d}) \mb{E}\left[ (\varepsilon^t)^2 \right]\mb{E}\big[ (\varepsilon^{t'})^2 \big] \\
    &\le C_{5}C_1\sigma^4 (t_2-t_1) +C_2^2 \sigma^4(t_2-t_1)^2\le \big( C_{5}C_1+C_2^2 \big)\sigma^4(t_2-t_1)^2 .
\end{align*}
Plug our estimations of $\mb{E}[S_{i,,1}]$, $\mb{E}[S_{i,2}]$, $\mb{E}[S_{i,3}]$ into \eqref{eq:difference in time and sum in space} and we have
\begin{align}\label{eq:upper bound for differenc in time LLN}
\begin{aligned}
     \frac{1}{d}\sum_{i=1}^d  \mb{E}\big[ |\dtheta^{t_2}_i-\dtheta^{t_1}_i|^4 \big] & \le 27\eta^4 d^4 \big[ C_2^4(t_2-t_1)^3+3\big( C_8'+C_{5}'^2 \big)(t_2-t_1) \big] \\
     &\quad \times \sum_{t=t_1}^{t_2-1} m^d_t +27 \big( C_{5}C_1+C_2^2 \big)\eta^4 \sigma^4(t_2-t_1)^2.
\end{aligned}
\end{align}
Pick $t_1=0,\ t_2=t\le N $ in \eqref{eq:upper bound for differenc in time LLN}. We have for any $t\le N$,
\begin{align*}
    m_{t}^d &\le 8m^d_0+216\eta^4 d^4 \big[ C_2^4 t^3+3\big( C_8'+C_{5}'^2 \big)t \big] \sum_{k=0}^{t-1} m^d_k +216 \big( C_{5}C_1+C_2^2 \big)\eta^4 \sigma^4 t^2\\
    &\le 8m^d_0+216 \big(C_2^4+ 3C_8'+3C_{5}'^2 \big)\eta^4 d^4t^3 \sum_{k=0}^{t-1} m^d_k +216 \big( C_{5}C_1+C_2^2 \big)\eta^4 \sigma^4 t^2.
\end{align*}
According to the discrete Gronwall's inequality, we have for any $t\le N$,
\begin{align*}
    m_{t}^d &\le 216 \big( C_{5}C_1+C_2^2 \big)\eta^4 \sigma^4 t^2\\
    &\quad+216^2\big( C_{5}C_1+C_2^2 \big)\eta^8d^4\sigma^4 t^3\sum_{k=0}^{t-1} k^2\exp\big( 216\big(C_2^4+ 3C_8'+3C_{5}'^2 \big)\eta^4 d^4t^3(t-k-1) \big)\\
    &\le 216 \big( C_{5}C_1+C_2^2 \big)\eta^4 \sigma^4 t^2\\
    &\quad+216^2\big( C_{5}C_1+C_2^2 \big)\eta^8d^4\sigma^4 t^6 \exp\big( 216\big(C_2^4+ 3C_8'+3C_{5}'^2 \big)\eta^4 d^4t^4 \big) \\
    &\le 216 \big( C_{5}C_1+C_2^2 \big)C_{s,1}^4 \tau^2\\
    &\quad +216^2\big( C_{5}C_1+C_2^2 \big)C_{s,1}^8 \tau^6 \exp\big( 216\big(C_2^4+ 3C_8'+3C_{5}'^2 \big)C_{s,1}^4 \tau^4\big), 
\end{align*}
where the last inequality follows from $\max(\eta dT,\eta\sigma T^{\frac{1}{2}})\le C_{s,1}$ and $N\le \tau T$. We simplify the upper bound of $m^d_t$ as $m^d_t\le C_{\tau} $ for some positive constant $C_\tau$ independent of $d,T,\sigma$. Apply the upper bound of $m^d_t$ to \eqref{eq:L2 difference in time} along with estimations of $\mb{E}[S_{i,,1}]$, $\mb{E}[S_{i,2}]$, $\mb{E}[S_{i,3}]$ and we get
\begin{align*}
    &\qquad \mb{E}\left[|\dtheta_i^{t_2}-\dtheta_i^{t_1}|^2\right] \\
    &\le 27\eta^4 d^4 \big( C_2^4(t_2-t_1)^3+3(C_8'+C_{5}'^2)(t_2-t_1) \big)\sum_{t=t_1}^{t_2-1} m^d_t+27(C_{5}C_1+C_2^2)\eta^4\sigma^4(t_2-t_1)^2\\
    &\le 27C_{\tau}\big( C_2^4+3C_8'+3C_{5}'^2 \big) \eta^4d^4 (t_2-t_1)^4 +27(C_{5}C_1+C_2^2)\eta^4\sigma^4(t_2-t_1)^2\\
    &\le 27C_{\tau}\big( C_2^4+3C_8'+3C_{5}'^2 \big) C_{s,1}^4 \left(\frac{t_2-t_1}{T}\right)^4 +27(C_{5}C_1+C_2^2)C_{s,1}^4 \left(\frac{t_2-t_1}{T}\right)^2,
\end{align*}
where the last inequality follows from $\max(\eta dT,\eta\sigma T^{\frac{1}{2}})\le C_{s,1}$. Last due to $\big(\frac{t_2-t_1}{T}\big)^2\le \tau^2$ for any $0\le t_1<t_2\le N$, \eqref{eq:L2 tightness in time LLN} is proved. 
\end{proof}

\begin{prop}\label{prop:L2 tightness in time and space LLN} Assume that Assumption \ref{ass:tightness in space LLN} and Assumption \ref{ass:continuous data} hold and there exists a uniform constant $C_{s,1}>0$ such that \textcolor{black}{for all $d\ge 1,T>0$ along the limiting sequence}, $\max(\eta d T,\eta \sigma T^{\frac{1}{2}})\le C_{s,1}$. Then there exists a positive constant $C$ depending on $\tau$ and $C_{s,1}$ such that for any $d,T>0$, $i,j\in [d]$ and $0\le t_1<t_2\le \lfloor{\tau T}\rfloor$,
\begin{align}\label{eq:L2 tightness in time and space LLN}
    \mb{E}\big[\big| (\dtheta^{t_2}_i-\dtheta^{t_2}_j)-(\dtheta^{t_1}_i-\dtheta^{t_1}_j)\big|^4\big]\le C\left(\frac{i-j}{d}\right)^4 \left( \frac{t_2-t_1}{T}\right)^2.
\end{align}
Furthermore, if initial conditions in Theorem \ref{thm:mainmean} hold, then for any $i,j\in [d]$ and any $0\le t\le \lfloor{\tau T}\rfloor$,
\begin{align}\label{eq:L2 tightness in space LLN}
    \mb{E}\big[\big| \dtheta^{t}_i-\dtheta^{t}_j\big|^4\big]\le \left(8L^4+8\tau^2 C \right) \left(\frac{i-j}{d}\right)^4.
\end{align}
where $L$ is the uniform constant in the initial condition in Theorem \ref{thm:mainmean}.
\end{prop}

\begin{proof}[Proof of Proposition \ref{prop:L2 tightness in time and space LLN}]
According to \eqref{eq:centralized online SGD}, for any $i,j\in [d]$ and any $0\le t\le \lfloor{\tau T}\rfloor-1\coloneqq N-1$, we have
\begin{align*}
    \dtheta^{t+1}_i-\dtheta^{t+1}_j&=(\dtheta^t_i-\dtheta^t_j)-\eta \sum_{l=1}^d \left(\mb{E}\left[x^t_ix^t_l\right]-\mb{E}\left[x^t_jx^t_l\right]\right)\dtheta^t_l \\
    &\quad -\eta \sum_{l=1}^d \left( x^t_ix^t_l-\mb{E}\left[x^t_ix^t_l\right]-x^t_jx^t_l+\mb{E}\left[x^t_j x^t_l\right] \right)\dtheta^t_l +\eta (x^t_i-x^t_j)\varepsilon^t.
\end{align*}
Summing over the time index and taking expectation of the absolute value of both sides, we get for any $0\le t_1<t_2\le N$,

\begin{align}\label{eq:L2 difference in space LLN}
\begin{aligned}
    &~\mb{E}\big[\big| (\dtheta^{t_2}_i-\dtheta^{t_2}_j)-(\dtheta^{t_1}_i-\dtheta^{t_1}_j)\big|^4\big]\\
\le&~27\eta^4 \mb{E}\big[\underbrace{\big( \sum_{t=t_1}^{t_2-1}  \sum_{l=1}^d \left(\mb{E}\left[x^t_ix^t_l\right]-\mb{E}\left[ x^t_jx^t_l\right]\right)\dtheta^t_l \big)^4}_{\mb{E}[S_{i,j,1}]}\big]  \\
 &~+ 27\eta^4 \mb{E}\big[\underbrace{\big( \sum_{t=t_1}^{t_2-1}  \sum_{l=1}^d \left( x^t_ix^t_l-\mb{E}\left[x^t_ix^t_l\right]-x^t_jx^t_l+\mb{E}\left[x^t_j x^t_l\right] \right)\dtheta^t_l \big)^4}_{\mb{E}[S_{i,j,2}]} \big]  \\
    &~+27\eta^4 \mb{E}\big[\underbrace{\big( \sum_{t=t_1}^{t_2-1} (x^t_i-x^t_j)\varepsilon^t \big)^4}_{\mb{E}[S_{i,j,3}]}\big].
\end{aligned}
\end{align}  

Next, we will estimate the right-hand side of \eqref{eq:L2 difference in space LLN} term by term. First we have 
\begin{align*}
     \mb{E}\big[ S_{i,j,1}\big]&=\mb{E}\big[ \sum_{r_1,r_2,r_3,r_4=t_1}^{t_2-1}\sum_{l_1,l_2,l_3,l_4=1}^d \prod_{k=1}^4 \big(A(\frac{i}{d},\frac{j_k}{d})-A(\frac{j}{d},\frac{j_k}{d})\big) \prod_{k=1}^4 \dtheta^t_l \big] \\
     &\le C_{3}^4 \big( \frac{i-j}{d} \big)^4 (t_2-t_1)^3 d^4\sum_{t=t_1}^{t_2-1} m^d_t,
\end{align*}
where the inequality follows from Assumption \ref{ass:continuous data} and $4(abcd)\le a^4+b^4+c^4+d^4$ for all $a,b,c,d\in\mb{R}$. Similar to $\mb{E}[S_{i,2}]$ in the proof of Proposition \ref{prop:L2 tightness in time LLN}, in the estimation of $\mb{E}[S_{i,j,2}]$, there are two types of nonzero terms.
\vspace{0.1in}

\textbf{Type 1}: The first type of nonzero terms are in the form of
\begin{align*}
    \sum_{l_1,l_2,l_3,l_4=1}^d \mb{E}\big[ \prod_{k=1}^4 \left( x^t_ix^t_{l_k}-\mb{E}[x^t_ix^t_{l_k}]-x^t_jx^t_{l_k}+\mb{E}[x^t_jx^t_{l_k}] \right)  \big|\prod_{k=1}^4 \dtheta^t_{l_k} \big| \big],
\end{align*}
and there are $3(t_2-t_1)$ such terms. The sum of such terms can be bounded by 
\begin{align*}
    &\qquad3\sum_{t=t_1}^{t_2-1}\sum_{l_1,l_2,l_3,l_4=1}^d \mb{E}\big[ \prod_{k=1}^4 \left( x^t_ix^t_{l_k}-\mb{E}[x^t_ix^t_{l_k}]-x^t_jx^t_{l_k}+\mb{E}[x^t_jx^t_{l_k}] \right)  \big|\prod_{k=1}^4 \dtheta^t_{l_k} \big| \big] \\
    &= 3\sum_{t=t_1}^{t_2-1}\sum_{l_1,l_2,l_3,l_4=1}^d \mb{E}\bigg[\mb{E}\big[ \prod_{k=1}^4 \left( x^t_ix^t_{l_k}-\mb{E}[x^t_ix^t_{l_k}]-x^t_jx^t_{l_k}+\mb{E}[x^t_jx^t_{l_k}] \right)  \big|\prod_{k=1}^4 \dtheta^t_{l_k} \big | \big] \big|\purple{\mc{F}_{t-1}^{d}} \bigg] \\
    &\le 3C_9'^4 \big( \frac{i-j}{d} \big)^4 \sum_{t=t_1}^{t_2-1}\sum_{l_1,l_2,l_3,l_4=1}^d \mb{E}\big[ \big|\prod_{k=1}^4 \dtheta^t_{l_k} \big| \big] \le  C_9'^4 \big( \frac{i-j}{d} \big)^4 d^4 \sum_{t=t_1}^{t_2-1} m^d_t, 
\end{align*}
where the first inequality follows from Assumption \ref{ass:continuous data} and $C_9'^4=C_9^4+2C_{3}^2C_{6}^2+C_{3}^4$. The last inequality follows from  $4(abcd)\le a^4+b^4+c^4+d^4$ for all $a,b,c,d\in\mb{R}$.

\textbf{Type 2}: For the simplicity of notations, we denote $\Delta X_{ij}^t\coloneqq x^{t}_{i}x^{t}_{j}-\mb{E}\big[x^{t}_{i}x^{t}_{j}\big]$ for all $i,j\in [d]$ and $t\le \lfloor \tau T \rfloor$. Then, the second type nonzero terms are in the form of
\begin{align*}
    &\sum_{l_1,l_2,l_3,l_4=1}^d  \mb{E}\Bigg[ \underbrace{\prod_{k=1}^2 \left( \Delta X_{il_k}^t-\Delta X_{jl_k}^t\right)  \prod_{k=3}^4 \left( \Delta X_{il_k}^{t'}-\Delta X_{jl_k}^{t'} \right)}_{P^{t,t'}_{i,j,l_1,l_2,l_3,l_4}} \times \big|\prod_{k=1}^2 \dtheta^t_{l_k} \prod_{k=3}^4  \dtheta^{t'}_{l_k}  \big| \Bigg],
\end{align*}
with $t\neq t'$ and there are $3(t_2-t_1)^2-3(t_2-t_1)$ such terms. The sum of such terms can be upper bounded by
\begin{align*}
    &\qquad \sum_{t\neq t'}\sum_{l_1,l_2,l_3,l_4=1}^d \mb{E}\bigg[ \mb{E}\big[ P^{t,t'}_{i,j,l_1,l_2,l_3,l_4} \prod_{k=1}^2 \big|\dtheta^t_{l_k}\big| \prod_{k=3}^4  \big|\dtheta^{t'}_{l_k}  \big|  \big] \big| \purple{\mc{F}_{t\vee t'-1}^{d}} \bigg] \\
    &\le C_{6}'^4 \big( \frac{i-j}{d} \big)^4 \sum_{t\neq t'}\sum_{l_1,l_2,l_3,l_4=1}^d \mb{E}\big[ \prod_{k=1}^2 \big|\dtheta^t_{l_k}\big| \prod_{k=3}^4  \big|\dtheta^{t'}_{l_k}  \big|  \big] \le C_{6}'^4 \big( \frac{i-j}{d} \big)^4 d^4 (t_2-t_1)\sum_{t=t_1}^{t_2-1} m^d_t,
\end{align*}
 where the first inequality follows from Assumption \ref{ass:continuous data} and $C_{6}'^4=C_{3}^2+C_{6}^2$. The last inequality follows from  $abcd\le \frac{a^4+b^4+c^4+d^4}{4}$ for all $a,b,c,d\in\mb{R}$. Therefore, $\mb{E}[S_{i,j,2}]$ can be upper bounded as
 \begin{align*}
     \mb{E}[S_{i,j,2}]&\le  \big(C_9'^4 +C_{6}'^4 (t_2-t_1)\big) \big( \frac{i-j}{d} \big)^4 d^4 \sum_{t=t_1}^{t_2-1} m^d_t\\
     &\le \big(C_9'^4 +C_{6}'^4 \big) (t_2-t_1) \big( \frac{i-j}{d} \big)^4 d^4 \sum_{t=t_1}^{t_2-1} m^d_t.
 \end{align*}
 Due to Assumption \ref{ass:continuous data}, we can bound $\mb{E}[S_{i,j,3}]$ as
 \begin{align*}
    &\quad \mb{E}[S_{i,j,3}] \\
    &=\sum_{t=t_1}^{t_2-1}\mb{E}\big[(x^t_i-x^t_j)^4\big]\mb{E}\big[ \left(\varepsilon^t\right)^4 \big]+\sum_{t\neq t'} \mb{E}\big[ (x^t_i-x^t_j)^2\big]\mb{E}\big[ (x^{t'}_i-x^{t'}_j)^2\big] \mb{E}\big[ \left(\varepsilon^t\right)^2 \big]\mb{E}\big[ \big(\varepsilon^{t'}\big)^2 \big]\\
     &\le C_1 C_{7}^4 (t_2-t_1) \big( \frac{i-j}{d} \big)^4 \sigma^4 +C_{4}^4 (t_2-t_1)^2 \big( \frac{i-j}{d} \big)^4 \sigma^4\\
     &\le \big( C_1 C_{7}^4+C_{4}^4 \big) (t_2-t_1)^2 \big( \frac{i-j}{d} \big)^4 \sigma^4 .
 \end{align*}
 We have shown $m^d_t\le C_{\tau}$ for any $0\le t\le N$ in the proof of Proposition \ref{prop:L2 tightness in time LLN}. With \eqref{eq:L2 difference in space LLN} and the estimations on $\mb{E}[S_{i,j,1}]$, $\mb{E}[S_{i,j,2}]$, $\mb{E}[S_{i,j,3}]$, we have
 \begin{align*}
     &\quad \mb{E}\big[\big| (\dtheta^{t_2}_i-\dtheta^{t_2}_j)-(\dtheta^{t_1}_i-\dtheta^{t_1}_j)\big|^4\big]\\
     &\le 27(C_9'^4+C_{6}'^4+C_{3}^4 )(t_2-t_1)^3 \big( \frac{i-j}{d} \big)^4 \eta^4 d^4 \sum_{t=t_1}^{t_2-1} m^d_t\\
     &\quad+27\big( C_1 C_{7}^4+C_{4}^4 \big) (t_2-t_1)^2 \big( \frac{i-j}{d} \big)^4 \eta^4\sigma^4 \\
     &\le 27C_{\tau}(C_9'^4+C_{6}'^4+C_{3}^4 )C_{s,1}^4\big(\frac{t_2-t_1}{T}\big)^4\big( \frac{i-j}{d} \big)^4 +27C_{s,1}^4 \big(\frac{t_2-t_1}{T}\big)^2\big( \frac{i-j}{d} \big)^4,
 \end{align*}
 where the last inequality follows from $\max(\eta dT,\eta \sigma T^{\frac{1}{2}})\le C_{s,1}$. Therefore, \eqref{eq:L2 tightness in time and space LLN} is proved. Last \eqref{eq:L2 tightness in space LLN} follows from \eqref{eq:L2 tightness in time and space LLN} and the initial conditions in Theorem \ref{thm:mainmean}.
\end{proof}

We are now ready to prove Theorem \ref{thm:tightness of continuous interpolation LLN} based on the above two propositions. 
\begin{proof}[Proof of Theorem \ref{thm:tightness of continuous interpolation LLN}] Tightness can be proved by the Kolmogorov tightness criteria \cite[Chapter 4]{karatzas2012brownian}. The last statement simply follows from tightness property. To apply the Kolmogorov tightness criteria, we need to verify the following two conditions:
\begin{enumerate}
    \item [(a)] $\{\Bar{\Theta}^{d,T}(0,0)\}_{d\ge 1, T> 0}$ is tight in the probability space.
    \item [(b)] There exists a positive constant $C_{\textbf{tight}}$ such that for any $s_1,s_2\in [0,\tau]$ and $x_1,x_2\in [0,1]$, we have $$\sup_{d,T} \mb{E}\big[ \big| \Bar{\Theta}^{d,T}(s_1,x_1)-\Bar{\Theta}^{d,T}(s_2,x_2)  \big|^4 \big]\le  C_{\textbf{tight}} \big( \left|s_1-s_2\right|^2+\left|x_1-x_2\right|^4 \big).$$ 
\end{enumerate}
To verify $(a)$, it is easy to see that for any $d\ge 1,T>0$ and $N>0$, 
\begin{align*}
    \mb{P}\big( \big| \Bar{\Theta}^{d,T}(0,0) \big|>N \big)\le N^{-2} \mb{E}\big[\big|\dtheta^0_{0}\big|^2\big]\le N^{-2}R^2\to 0  \qquad \text{as }N\to \infty.
\end{align*}
To verify $(b)$, first we notice that $\Bar{\Theta}^{d,T}(\cdot,\cdot)$ is piecewise linear in both variables due to \eqref{eq:approximation to Theta LLN}. Without loss of generality, we assume that $0\le s_1<s_2\le \tau $, $0\le x_1<x_2\le 1$ and $\lfloor{Ts_1}\rfloor=t_1< \lfloor{Ts_2}\rfloor=t_2$, $\lfloor{dx_1}\rfloor=i<\lfloor{dx_2}\rfloor=j$. We have
\begin{align*}
    {\mb{E}\big[ \big| \Bar{\Theta}^{d,T}(s_1,x_1)-\Bar{\Theta}^{d,T}(s_2,x_2)\big|^4 \big]}&\le  {8\mb{E}\big[ \big| \Bar{\Theta}^{d,T}(s_1,x_1)-\Bar{\Theta}^{d,T}(s_2,x_1)\big|^4 \big]}\\
    &\qquad\qquad\qquad+ {8\mb{E}\big[ \big| \Bar{\Theta}^{d,T}(s_2,x_1)-\Bar{\Theta}^{d,T}(s_2,x_2)\big|^4 \big]}.
\end{align*}
 According to Proposition \ref{prop:L2 tightness in time LLN}, we have
\begin{align*}
   &\quad\mb{E}\big[ \big| \Bar{\Theta}^{d,T}(s_1,x_1)-\Bar{\Theta}^{d,T}(s_2,x_1)\big|^4 \big]\\
   &\le 27C  \left( \big(\frac{t_1+1}{T}-s_1\big)^2+\big(\frac{t_2-t_1-1}{T}\big)^2+\big(s_2-\frac{t_2}{T}\big)^2  \right) \\
   &\le 27C \left(s_2-s_1\right)^2.
\end{align*}
and according to Proposition \ref{prop:L2 tightness in time and space LLN}, we have
\begin{align*}
   &\quad\mb{E}\big[ \big| \Bar{\Theta}^{d,T}(s_2,x_1)-\Bar{\Theta}^{d,T}(s_2,x_2)\big|^4 \big]\\
   &\le 27(8M^4+8\tau^2 C) \left( \big(\frac{i+1}{d}-x_1\big)^4+\big(\frac{j-i-1}{d}\big)^4+\big(x_2-\frac{j}{d}\big)^4  \right)\\
   &\le 216(M^4+\tau^2 C) \left(x_2-x_1\right)^4.  
\end{align*}
Therefore $(b)$ holds with $C_{\textbf{tight}}=216C+1728(M^4+\tau^2 C)$.
\end{proof}

\subsection{Limit identification}\label{sec:LLN proofs}

\begin{proof}[Proof of Theorem \ref{thm:LLN}.] According to Theorem \ref{thm:tightness of continuous interpolation LLN}, any subsequence 
\begin{align*}
\{\Bar{\Theta}^{d_k,T_k}(\cdot,\cdot)\}_{k\ge 1}~~~~ \text{of}~~~~ \{\Bar{\Theta}^{d,T}(\cdot,\cdot)\}_{d\ge 1, T> 0}
\end{align*}
has a further weakly convergent subsequence with limit $\Theta \in {C}([0,\tau]; {C}([0,1]))$ as $d_k,T_k\to\infty$. For the simplicity of notations, we denote the convergent subsequence of $ \{\Bar{\Theta}^{d_k,T_k}\}_{k\ge 1}$ by $ \{\Bar{\Theta}^{d,T}\}_{d\ge 1, T> 0}$ in the proof.

To identify the limit, first we rewrite \eqref{eq:centralized online SGD} in terms of $\Bar{\Theta}^{d,T}$. For any $0\le t\le \lfloor{\tau T}\rfloor-1$ and any $i\in [d]$:
\begin{align}\label{eq:online SGD embedding form}
\begin{aligned}
    &~\Bar{\Theta}^{d,T}(\frac{t+1}{T},\frac{i}{d})-\Bar{\Theta}^{d,T}(\frac{t}{T},\frac{i}{d})\\
    =&~-\eta \sum_{j=1}^d W(\frac{t}{T},\frac{i}{d})W(\frac{t}{T},\frac{j}{d})\Bar{\Theta}^{d,T}(\frac{t}{T},\frac{j}{d}) +\eta W(\frac{t}{T},\frac{i}{d})\varepsilon^t.
\end{aligned}
\end{align}
Therefore for any bounded smooth function $f:[0,\tau]\to\mb{R}$, we have for any $s\in [0,\tau]$, $i\in [d-1]$ and $x \in [\frac{i}{d},\frac{i+1}{d}] $,
{\small
\begin{align}\label{eq:test function applied sum}
    &\quad \sum_{t=0}^{\lfloor sT \rfloor-1} f(\frac{t}{T})\big(\Bar{\Theta}^{d,T}(\frac{t+1}{T},x)-\Bar{\Theta}^{d,T}(\frac{t}{T},x) \big) \nonumber \\
    &= -\eta (i+1-dx)\sum_{t=0}^{\lfloor sT \rfloor-1}\sum_{j=1}^d f(\frac{t}{T}) W(\frac{t}{T},\frac{i}{d})W(\frac{t}{T},\frac{j}{d})\Bar{\Theta}^{d,T}(\frac{t}{T},\frac{j}{d})\\
    &\quad+\eta(i+1-dx) \sum_{t=0}^{\lfloor sT \rfloor-1} f(\frac{t}{T})W(\frac{t}{T},\frac{i}{d})\varepsilon^t \nonumber \\
    &\quad  -\eta (dx-i)\sum_{t=0}^{\lfloor sT \rfloor-1}\sum_{j=1}^d f(\frac{t}{T}) W(\frac{t}{T},\frac{i+1}{d})W(\frac{t}{T},\frac{j}{d})\Bar{\Theta}^{d,T}(\frac{t}{T},\frac{j}{d})\nonumber \\
    &\quad+\eta(dx-i) \sum_{t=0}^{\lfloor sT \rfloor-1} f(\frac{t}{T})W(\frac{t}{T},\frac{i+1}{d})\varepsilon^t \nonumber.
\end{align}
}
We can rewrite the left-hand side of \eqref{eq:test function applied sum} as
\begin{align*}
     &\quad\text{LHS}\eqref{eq:test function applied sum}\\
     &= \sum_{t=1}^{\lfloor sT \rfloor-1} \Bar{\Theta}^{d,T}(\frac{t}{T},x)\big( f(\frac{t-1}{T})-f(\frac{t}{T}) \big)+f(\frac{\lfloor sT \rfloor-1}{T})\Bar{\Theta}^{d,T}(\frac{\lfloor sT \rfloor}{T},x)-f(0)\Bar{\Theta}^{d,T}(0,x).
\end{align*}
When $d,T\to \infty$, since $f$ is bounded and smooth, for any $s\in (0,\tau)$, we have
\[
f(s\pm T^{-1})=f(s)+O(T^{-1}),\qquad  T\left( f(s)-f(s-T^{-1})\right)=f'(s)+O(T^{-1}).
\]
Since $\{\Bar{\Theta}^{d,T}\}_{d\ge 1,T>0}$ converges weakly to $\Theta$ and $f,f'$ are continuously bounded, we have
\begin{align*}
    \frac{1}{T}\sum_{t=1}^{\lfloor sT \rfloor-1} \Bar{\Theta}^{d,T}(\frac{t}{T},x)\left( f'(\frac{t}{T})+O(T^{-1})  \right)=\int_0^{\frac{\lfloor sT \rfloor}{T}} \Theta( u ,x)  f'(u) \mathrm{d}u +o(1).
\end{align*}
Therefore 
\begin{align}\label{eq:LHS of test function applied sum}
      \text{LHS}\eqref{eq:test function applied sum}=-\int_0^{\frac{\lfloor sT \rfloor}{T}} \Theta( u ,x)  f'(u) \mathrm{d}u+ f(s){\Theta}(\frac{\lfloor sT\rfloor}{T},x)-f(0){\Theta}(0,x) +o(1).
\end{align}
Next, we look at the right-hand side of \eqref{eq:test function applied sum}. For any $i\in [d]$, $\text{RHS}\eqref{eq:test function applied sum}=-(i+1-dx) I_i^{\purple{s}}-(dx-i)I_{i+1}^s$,
\begin{align*}
    &\quad I_i^s\\
    =&\underbrace{- \frac{\eta dT}{d T}\sum_{t=0}^{\lfloor sT \rfloor-1}\sum_{j=1}^d f(\frac{t}{T})A(\frac{i}{d},\frac{j}{d})\Bar{\Theta}^{d,T}(\frac{t}{T},\frac{j}{d})}_{N_{i,1}^s}+ \underbrace{ \frac{\eta dT}{d T} \sum_{t=0}^{\lfloor sT \rfloor-1} f(\frac{t}{T})W(\frac{t}{T},\frac{i}{d})\varepsilon^t}_{N_{i,2}^s} \\
    & \underbrace{- \frac{\eta dT}{d T}\sum_{t=0}^{\lfloor sT \rfloor-1}\sum_{j=1}^d f(\frac{t}{T})\big( W(\frac{t}{T},\frac{i}{d})W(\frac{t}{T},\frac{j}{d})-\mb{E}\big[W(\frac{t}{T},\frac{i}{d})W(\frac{t}{T},\frac{j}{d})\big] \big)\Bar{\Theta}^{d,T}(\frac{t}{T},\frac{j}{d})}_{N_{i,3}^s} .
\end{align*}
\purple{Based on the same space-time interpolation introduced in Section~\ref{sec:Introduction}, for each $l=1,2,3$, $\{N_{i,l}^s\}$ can be interpolated as a function in $C([0,\tau];C([0,1]))$, denoted as $N^{d,T}_l$. According to Lemma \ref{lem:joint tightness LLN}, $(\Bar{\Theta}^{d,T}, N_1^{d,T}, N_2^{d,T}, N_3^{d,T})$ is tight in $C([0,\tau];C([0,1]))^{\otimes 4}$. Therefore, the sequential weak limit exists in $C([0,\tau];C([0,1]))^{\otimes 4}$.
 The sequential weak limit of $(\Bar{\Theta}^{d,T}, N_1^{d,T}, N_2^{d,T}, N_3^{d,T})$ in $C([0,\tau];C([0,1]))^{\otimes 4}$ is derived based on the weak limit of each term. Note that for $\Bar{\Theta}^{d,T}$, the weak limit, denoted as $\Theta$, satisfies an equation based on the limits of other terms, which are identified as follows:
\begin{itemize}
    \item  For $N_1^{d,T}$, weak limit is identified based on the continuous mapping theorem~\cite[Theorem 2.7]{billingsley2013convergence};
    \item  For $N_2^{d,T}$, to identify weak limit of an infinite dimensional process, it suffices to identify the weak limits for each finite dimensional weak limit. We identify the finite dimensional limit via pointwise central limit theorem.
    \item For $N_3^{d,T}$, we use the fact that convergence in probability to zero implies convergence in distribution to zero. 
\end{itemize}
}

We start with term $N_1^{d,T}$. Due to the facts that $f, A$ are continuously bounded, and $\{\Bar{\Theta}^{d,T}\}_{d\ge 1,T>0}$ converges weakly to $\Theta$ (see Theorem~\ref{thm:tightness of continuous interpolation LLN}), we have
\begin{align*}
    N_{i,1}^s&= -\eta d T\big( \frac{1}{d T}\sum_{t=0}^{\lfloor sT \rfloor-1}\sum_{j=1}^d f(\frac{t}{T})A(\frac{i}{d},\frac{j}{d})\Theta(\frac{t}{T},\frac{j}{d})+o(1) \big)\\
    &= -\eta d T \big( \int_0^{\frac{\lfloor sT\rfloor}{T}}\int_0^1 f(u)A(\frac{i}{d},y)\Theta(u,y)\mathrm{d}y\mathrm{d}u+o(1)\big).
\end{align*}
For $N_{i,2}^s= \frac{\eta dT}{d T} \sum_{t=0}^{\lfloor sT \rfloor-1} f(\frac{t}{T})W(\frac{t}{T},\frac{i}{d})\varepsilon^t$, note that $\mb{E}[N_{i,2}^s]=0$. Since $\{W(\frac{t}{T},\frac{i}{d})\varepsilon^t\}_{t=1}^T$ is a sequence of i.i.d. random variables, the limit of $N_{i,2}^s$ can be studied via standard Central Limit Theorems. In particular, we have
\begin{align}\label{eq:covariance Ni2}
       \mb{E}[N_{i,2}^{s_1}N_{j,2}^{s_2}]= \sigma^2 \eta^2 T A(\frac{i}{d},\frac{j}{d}) \big( \int_0^{s_1\wedge s_2} f(u)^2 \mathrm{d}u+o(1) \big).
    \end{align}
\purple{Therefore, there exists a centered Gaussian field process $\{\xi_1(s,x)\}_{s\in [0,\tau],x\in [0,1]}$ such that for any $s_1,s_2\in [0,\tau]$ and $x,y\in [0,1]$,
\begin{align*}
       \mb{E}[\xi_1(s_1,x)\xi_1(s_2,y)]=\sigma_1(s_1,s_2,x,y)\quad\text{with}\quad \sigma_1(s_1,s_2,x,y)= (s_1\wedge s_2) A(x,y), 
    \end{align*}
   and $N_{i,2}^s=\sigma^2\eta^2 T \big( \int_0^s f(u)\mathrm{d}\xi_1(u,\frac{i}{d}) +o(1)\big)$ in distribution for all $i\in [d]$. By checking the covariance of the process $\{(N_{1,3}^s,N_{2,3}^s,\cdots,N_{d,3}^s)\}_{s\ge 0}$ via equation \eqref{eq:covariance Ni2}, we have 
 $\{(N_{i,2}^s)_{i\in [d]}\}_{s\ge 0}= \{(\sigma^2\eta^2 T  \int_0^s f(u)\mathrm{d}\xi_1(u,\frac{i}{d})+o(1) )_{i\in [d]}\}_{s\ge 0}$. }
 For $N_{i,3}^s$, we have 
\[
N_{i,3}^s= -(\eta d T) \sum_{j=1}^d \sum_{t=0}^{\lfloor sT \rfloor-1}Z_{t,j},
\]
where $Z_{t,j}\coloneqq\frac{1}{dT}f(\frac{t}{T})\big( W(\frac{t}{T},\frac{i}{d})W(\frac{t}{T},\frac{j}{d})-\mb{E}\big[W(\frac{t}{T},\frac{i}{d})W(\frac{t}{T},\frac{j}{d})\big] \big)\Bar{\Theta}^{d,T}(\frac{t}{T},\frac{j}{d})$. Notice that
\begin{align*}
    &\mb{E}\left[Z_{t,j}\right]\\
    =&\frac{1}{T}f(\frac{t}{T})\mb{E}\bigg[\mb{E}\big[\big( W(\frac{t}{T},\frac{i}{d})W(\frac{t}{T},\frac{j}{d})-\mb{E}\big[W(\frac{t}{T},\frac{i}{d})W(\frac{t}{T},\frac{j}{d})\big] \big)| \purple{\mc{F}_{t-1}^{d}} \big]\Bar{\Theta}^{d,T}(\frac{t}{T},\frac{j}{d}) \bigg]=0.
\end{align*}
Furthermore, we can check that for any $t_1\neq t_2$ and any $j,l\in [d]$, 
\begin{align*}
    &\mb{E}\left[Z_{t_1,j}Z_{t_2,l}\right]\\
    =&\frac{1}{d^2 T^2}f(\frac{t_1}{T})f(\frac{t_2}{T})\big(B(\frac{i}{d},\frac{j}{d},\frac{i}{d},\frac{l}{d})-A(\frac{i}{d},\frac{j}{d})A(\frac{i}{d},\frac{l}{d})\big)\mb{E}\big[\Bar{\Theta}^{d,T}(\frac{t_1}{T},\frac{j}{d})\Bar{\Theta}^{d,T}(\frac{t_2}{T},\frac{l}{d})  \big] .
\end{align*}
Therefore, we have $\mb{E}\left[N_{i,3}^s\right]=0$ and 
\begin{align*}
    &\text{Var}( N_{i,3})\\
    =&\eta ^2 \sum_{j,l=1}^d \sum_{t=0}^{\lfloor sT \rfloor-1} f(\frac{t}{T})^2\big(B(\frac{i}{d},\frac{j}{d},\frac{i}{d},\frac{l}{d})-A(\frac{i}{d},\frac{j}{d})A(\frac{i}{d},\frac{l}{d})\big) \mb{E}\big[\Bar{\Theta}^{d,T}(\frac{t}{T},\frac{j}{d})\Bar{\Theta}^{d,T}(\frac{t}{T},\frac{l}{d})\big] \\
    \le &  \left(C_2^2+C_{5}\right) \lv f \rv_\infty^2 s \eta^2 d T \sum_{j=1}^d \mb{E}\big[ |\Bar{\Theta}^{d,T}(\frac{t}{T},\frac{j}{d})|^2 \big]\\
    \le &  C_{s,1}^2 \left(C_2^2+C_{5}\right) \lv f \rv_\infty^2 \frac{s}{T}\big( \frac{1}{d}\sum_{j=1}^d \mb{E}\big[ |\Delta \theta^0_j|^2 \big]+\frac{1}{d}\sum_{j=1}^d \mb{E}\big[ |\Delta \theta^t_j-\dtheta^0_j|^2 \big] \big)\\
    &\to 0, \qquad \text{as }d,T\to\infty ,
\end{align*}
where the first inequality follows from Assumption \ref{ass:continuous data}. The last limit follows from Proposition \ref{prop:L2 tightness in time LLN} and initial conditions in Theorem \ref{thm:mainmean}. Therefore, we have shown that $N_{i,3}^s\to 0$ in probability, and we write it as $N_{i,3}^s=o(1)$ in the following calculation. Combine our approximations on $N_{i,1}^s,N_{i,2}^s,N_{i,3}^s$ for any $i\in [d]$ and $s\in [0,\tau]$, we can write the right-hand side of \eqref{eq:test function applied sum} as 
\begin{align}\label{right side of test function applied sum}
    &\quad\text{RHS}\eqref{eq:test function applied sum} \nonumber\\
    &=-(i+1-dx) (N_{i,1}^s+N_{i,2}^s+N_{i,3}^s)-(dx-i)(N_{i+1,1}^s+N_{i+1,2}^s+N_{i+1,3}^s) \nonumber \\
    &= -\eta d T \int_0^{\frac{\lfloor sT \rfloor}{T}}\int_0^1 f(u)((i+1-dx)A(\frac{i}{d},y)+(dx-i)A(\frac{i+1}{d},y))\Theta(u,y)\mathrm{d}y\mathrm{d}u   \nonumber \\
    &\quad +\sigma^2 \eta^2 T \big( (i+1-dx)\int_0^s f(u)\mathrm{d}\xi_1(u,\frac{i}{d})+(dx-i)\int_0^s f(u)\mathrm{d}\xi_1(u,\frac{i+1}{d}) \big)\nonumber \\
    &\quad+o(\eta d T )+ o(\sigma^2 \eta^2 T) \nonumber
    \\
    &=-\eta d T \int_0^{\frac{\lfloor sT \rfloor}{T}}\int_0^1 f(u)A(x,y)\Theta(u,y)\mathrm{d}y\mathrm{d}u  +\sigma^2 \eta^2 T  \int_0^s f(u)\mathrm{d}\xi_1(u,x)\\
    &\quad+o(\eta d T )+ o(\sigma^2 \eta^2 T)\nonumber,
\end{align}
\purple{where the last identity follows from the facts that $A\in C^2([0,1])$ and for any $x\in [\frac{i}{d},\frac{i+1}{d}) $, $y\in [\frac{j}{d},\frac{j+1}{d})$, 
\begin{align*}
&\quad (i+1-dx)(j+1-dy)A(\frac{i}{d},\frac{j}{d})+(i+1-dx)(dy-j)A(\frac{i}{d},\frac{j+1}{d})\\
&+(dx-i)(j+1-dy)A(\frac{i+1}{d},\frac{j}{d})+(dx-i)(dy-j)A(\frac{i+1}{d},\frac{j+1}{d})\\
&=A(x,y)+o(1).
\end{align*}} Finally, with \eqref{eq:LHS of test function applied sum} and\eqref{right side of test function applied sum}, we have for any $s\in [0,\tau]$,

\begin{align*}
   & \quad -\int_0^{\frac{\lfloor sT \rfloor}{T}} \Theta( u ,x)  f'(u) \mathrm{d}u+ f(s){\Theta}(\frac{\lfloor sT\rfloor}{T},x)-f(0){\Theta}(0,x) +o(1)\\
   &=-\eta d T \int_0^{\frac{\lfloor sT \rfloor}{T}}\int_0^1 f(u)A(x,y)\Theta(u,y)\mathrm{d}y\mathrm{d}u  +\sigma^2 \eta^2 T  \int_0^s f(u)\mathrm{d}\xi_1(u,x)\\
   &\quad+o(\eta d T )+ o(\sigma^2 \eta^2 T).
\end{align*}

Hence, letting $d,T\to\infty$ we obtain that for any bounded smooth test function $f$ and for any $s\in [0,\tau]$, $x\in [0,1]$,
\begin{align*}
    &-\int_0^s \Theta( u ,x)  f'(u) \mathrm{d}u+ f(s){\Theta}(s,x)-f(0){\Theta}(0,x) \\
    =&-\eta d T\big( \int_0^{s}\int_0^1 f(u)A(x,y)\Theta(u,y)\mathrm{d}y\mathrm{d}u+o(1)\big)\\
    &~~+\sigma^2 \eta^2 T \big( \int_0^s f(u)\mathrm{d}\xi_1(u,x)+o(1)\big).
\end{align*}

\end{proof}

\begin{proof}[Proof of Theorem \ref{thm:mainmean}.] The proof follows from the results established in Theorem \ref{thm:LLN}.

The moderate-noise setup and the high-noise setup simply follow from integration by parts according to the consequences 2 and 3 of Theorem \ref{thm:LLN} respectively, as we discussed in Section \ref{sec:KTC}. The uniqueness and existence of solution in $C([0,\tau];C([0,1]))$ follows from part (a) of Theorem \ref{prop:existence and uniqueness of SDE}.

For the low-noise setup, since $f$ is smooth, we know that $$
\int_0^{(\cdot)} \Theta( u ,x)  f'(u) \mathrm{d}u\in C^1([0,\tau])\quad\text{and}\quad \int_0^{(\cdot)} f(u)A(x,y)\Theta(u,y)\mathrm{d}u\mathrm{d}y\in C^1([0,\tau]).$$ 
Therefore according to \eqref{eq:LLN weak form}, $f(\cdot)\Theta(\cdot,x)\in C^1([0,\tau])$ for any $x\in[0,1]$ which implies that $\Theta(\cdot,x)\in C^1([0,\tau])$ for any $x\in [0,1]$. We can then apply integration by parts to the left side of \eqref{eq:LLN weak form}. Therefore, $\Theta$ satisfies \eqref{eq:mainmean ODE LLN}.

Since $A$ satisfies Assumption \ref{ass:continuous data}, for any $\Theta_1,\Theta_2\in C([0,\tau];C([0,1]))$, we have
\begin{align*}
     &\quad\alpha\sup_{x\in [0,1]} \bigg|  \int_0^1 A(x,y) \left( \Theta_1(s,y)-\Theta_2(s,y) \right)\mathrm{d}y \bigg| \\
     &\le \alpha \sup_{x,y\in [0,1]} \left|A(x,y)\right| \sup_{y\in [0,1]} \left| \Theta_1(s,y)-\Theta_2(s,y)  \right|\\
    &\le \alpha C_2  \sup_{y\in [0,1]} \left| \Theta_1(s,y)-\Theta_2(s,y)  \right|.
\end{align*}
Therefore the right-hand side of \eqref{eq:mainmean ODE LLN} is Lipschitz in $\Theta(s,\cdot)$ for any $s\in [0,\tau]$.  According to the Picard-Lindel\"{o}f theorem (see, for example,~\cite{arnold1992ordinary}), there exists a unique solution in $C([0,\tau];C([0,1]))$ to \eqref{eq:mainmean ODE LLN} with any initial condition $\Theta(0,\cdot)=\Theta_0(\cdot)\in C([0,1])$. 

With the uniqueness of solution to \eqref{eq:mainmean ODE LLN} and Theorem \ref{thm:LLN}, we know that every subsequence of $\{\Theta^{d,T}\}_{d\ge 1,T>0}$ has a further subsequence converging weakly to a unique $\Theta\in C([0,\tau];C([0,1]))$. Therefore, $\{\Theta^{d,T}\}_{d\ge 1,T>0}$ converges weakly to $\Theta$ as $d,T\to\infty$.
\end{proof}

\begin{lemma}\label{lem:joint tightness LLN} Let $N^{d,T}_l \in C([0,\tau];C([0,1]))$ be the interpolated function defined in the proof of Theorem \ref{thm:LLN} for $l=1,2,3$. Under the conditions in Theorem \ref{thm:LLN}, the sequence $\{(\Bar{\Theta}^{d,T}, N^{d,T}_1, N^{d,T}_2, N^{d,T}_3)\}_{d\ge 1, T>0}$ is tight in the space $C([0,\tau];C([0,1]))^{\otimes 4}$.      
\end{lemma}
\begin{proof}[Proof of Lemma \ref{lem:joint tightness LLN}] Due to the product space structure, it suffices to show tightness for each one of $\Bar{\Theta}^{d,T}, N^{d,T}_1, N^{d,T}_2, N^{d,T}_3$ in $C([0,\tau];C([0,1]))$. The tightness of $\{\Bar{\Theta}^{d,T}\}$ is proved in Theorem \ref{thm:tightness of continuous interpolation LLN}. In this proof, we show tightness for $\{N^{d,T}_l\}$ for $l=1,2,3$.

\purple{\textbf{Tightness of $N_1^{d,T}$.} First, we have $N^{d,T}_1(0,0)=-\eta f(0) \sum_{j=1}^d A(0,\frac{j}{d})\Bar{\Theta}^{d,T}(0,\frac{j}{d})$. Under assumptions in Theorem \ref{thm:LLN}, we have \begin{align*}
    \mb{P}(|N^{d,T}_1(0,0)| >N )&\le N^{-2} \eta^2 f(0)^2 \big( \sum_{j=1}^d A(0,\frac{j}{d})^2 \big) \big( \sum_{j=1}^d \mb{E}[ |\Delta \theta^{0}_j|^2 ] \big) \\
    &\le C_2^2\lv f \rv_{\infty}^2  N^{-2}\eta^2 d^2 R^2 .
\end{align*}
Since $\eta d T\le C_{s,1}$ for a uniform constant $C_{s,1}$, $\mb{P}(|N^{d,T}_1(0,0)| >N ) \lesssim N^{-2}\to 0$ as $N\to\infty$. Therefore, $\{N^{d,T}_1(0,0)\}_{d\ge 1, T>0}$ is tight in the probability space. Next, we observe that for any $0\le t_1<t_2\le \lfloor sT \rfloor-1$ and $i_1,i_2\in [d]$,
\begin{align*}
    N^{d,T}_1(\frac{t_2}{T},\frac{i_2}{d})-N^{d,T}_1(\frac{t_1}{T},\frac{i_1}{d}) &= \eta \sum_{t=t_1}^{t_2-1}\sum_{j=1}^d f(\frac{t}{T})A(\frac{i_2}{d},\frac{j}{d})\Bar{\Theta}^{d,T}(\frac{t}{T},\frac{j}{d})\\
    &\ +\eta \sum_{t=0}^{t_1-1}\sum_{j=1}^d f(\frac{t}{T})\big(A(\frac{i_2}{d},\frac{j}{d})-A(\frac{i_1}{d},\frac{j}{d})\big)\Bar{\Theta}^{d,T}(\frac{t}{T},\frac{j}{d})
\end{align*}
Therefore, we have
\begin{align*}
    &\quad\mb{E}[| N^{d,T}_1(\frac{t_2}{T},\frac{i_2}{d})-N^{d,T}_1(\frac{t_1}{T},\frac{i_1}{d}) |^4]\\
    &\lesssim \eta^4\mb{E}\big[ \sum_{r_1,r_2,r_3,r_4=t_1}^{t_2-1}\sum_{j_1,j_2,j_3,j_4=1}^d \prod_{k=1}^4 |A(\frac{i_2}{d},\frac{j_k}{d})|\prod_{k=1}^4 |f(\frac{r_k}{T})\dtheta^{r_k}_{j_k}| \big]\\
    &\ +\eta^4\mb{E}\big[ \sum_{r_1,r_2,r_3,r_4=0}^{t_1-1}\sum_{j_1,j_2,j_3,j_4=1}^d \prod_{k=1}^4 |A(\frac{i_2}{d},\frac{j_k}{d})-A(\frac{i_1}{d},\frac{j_k}{d})|\prod_{k=1}^4 |f(\frac{r_k}{T})\dtheta^{r_k}_{j_k}| \big]\\
   &\le C_2^4 \lv f \rv_{\infty}^4 \eta^4 \mb{E}\big[ \sum_{r_1,r_2,r_3,r_4=t_1}^{t_2-1}\sum_{j_1,j_2,j_3,j_4=1}^d \prod_{k=1}^4 \purple{|}\dtheta^{r_k}_{j_k} \purple{|}\big] \\
   &\ +C_2^4 \lv f \rv_{\infty}^4 \eta^4 \big(\frac{i_2-i_1}{d}\big)^4\mb{E}\big[ \sum_{r_1,r_2,r_3,r_4=0}^{t_1-1}\sum_{j_1,j_2,j_3,j_4=1}^d \prod_{k=1}^4 \purple{|}\dtheta^{r_k}_{j_k} \purple{|}\big] \\
   &\le C_2^4\lv f \rv_{\infty}^4 \eta^4 d^4 (t_2-t_1)^3 \sum_{t=t_1}^{t_2-1} m^d_t \\
   &\ +C_2^4C_3^4\lv f \rv_{\infty}^4 \eta^4 d^4 t_1^3 \big(\frac{i_2-i_1}{d}\big)^4 \sum_{t=0}^{t_1-1} m^d_t\\
   &\lesssim (\eta d T)^4 \big(\frac{t_2-t_1}{T}\big)^4 + (\eta dT)^4 \big(\frac{i_2-i_1}{d}\big)^4,
\end{align*}
where the third inequality follows from Assumption \ref{sec:assumptions LLN} and the last inequality follows from the upper bound of $m^d_t$ derived in the proof of Proposition \ref{prop:L2 tightness in time LLN}. Since $N^{d,T}_1$ is piecewise linear in both variables and $\eta d T\le C_{s,1}$ for all $d\ge 1, T>0$, for all $0\le s_1<s_2\le \tau$ and $0\le x_1<x_2\le 1$, we have
\begin{align*}
    {\mb{E}\big[ \big| N_1^{d,T}(s_1,x_1)-N_1^{d,T}(s_2,x_2)\big|^4 \big]}
    &\lesssim  (s_2-s_1)^4 + (x_1-x_2)^4.
\end{align*}
The tightness of $N^{d,T}_1$ in $C([0,\tau];C([0,1]))$ follows from the Kolmogorov's tightness criterion. }

\purple{\textbf{Tightness of $N_2^{d,T}$.} First, we have $N^{d,T}_2(0,0)=-\eta f(0) W(0,0)\varepsilon^0$. Under assumptions in Theorem \ref{thm:LLN}, we have \begin{align*}
    \mb{P}(|N^{d,T}_2(0,0)| >N )\le N^{-2} \eta^2 f(0)^2 \sigma^2 A(0,0)  \le C_2 \lv f \rv_{\infty}^2 N^{-2}\eta^2 \sigma^2  .
\end{align*}
Since $\sigma \eta T^{\frac{1}{2}}\le C_{s,1}$ for a uniform constant $C_{s,1}$, $\mb{P}(|N^{d,T}_2(0,0)| >N ) \lesssim N^{-2}\to 0$ as $N\to\infty$. Therefore, $\{N^{d,T}_2(0,0)\}_{d\ge 1, T>0}$ is tight in the probability space. Next, we observe that for any $0\le t_1<t_2\le \lfloor sT \rfloor-1$ and $i_1,i_2\in [d]$,
\begin{align*}
    N^{d,T}_2(\frac{t_2}{T},\frac{i_2}{d})-N^{d,T}_2(\frac{t_1}{T},\frac{i_1}{d}) &= \eta \sum_{t=t_1}^{t_2-1} f(\frac{t}{T})W(\frac{t}{T},\frac{i_2}{d})\varepsilon^t\\
    &\ +\eta \sum_{t=0}^{t_1-1}f(\frac{t}{T})\big(W(\frac{t}{T},\frac{i_2}{d})-W(\frac{t}{T}, \frac{i_1}{d})\big)\varepsilon^t
\end{align*}
It follows from the arguments of bounding $\mb{E}[S_{i,3}]$ and $\mb{E}[S_{i,j,3}]$ in the proofs of Proposition \ref{prop:L2 tightness in time LLN} and Proposition \ref{prop:L2 tightness in time and space LLN} that
\begin{align*}
    \mb{E}[| N^{d,T}_2(\frac{t_2}{T},\frac{i_2}{d})-N^{d,T}_2(\frac{t_1}{T},\frac{i_1}{d}) |^4]&\lesssim C_2^2\eta^4\lv f \rv_{\infty}^4 \sigma^4 (t_2-t_1)^2  +C_5\eta^4\lv f \rv_{\infty}^4 \sigma^4 (t_2-t_1) \\
    &\ +C_7^4\eta^4\lv f \rv_\infty^4 \sigma^4 t_1^4 \big(\frac{i_2-i_1}{d}\big)^4  \\
   &\lesssim (\sigma\eta T^{\frac{1}{2}})^4 \big(\frac{t_2-t_1}{T}\big)^2 + (\sigma\eta T^{\frac{1}{2}})^4 \big(\frac{i_2-i_1}{d}\big)^4.
\end{align*}
Since $N^{d,T}_2$ is piecewise linear in both variables and $\sigma\eta T^{\frac{1}{2}}\le C_{s,1}$ for all $d\ge 1, T>0$, for all $0\le s_1<s_2\le \tau$ and $0\le x_1<x_2\le 1$, we have
\begin{align*}
    {\mb{E}\big[ \big| N_2^{d,T}(s_1,x_1)-N_2^{d,T}(s_2,x_2)\big|^4 \big]}
    &\lesssim  (s_2-s_1)^2 + (x_1-x_2)^4.
\end{align*}
The tightness of $N^{d,T}_2$ in $C([0,\tau];C([0,1]))$ follows from the Kolmogorov's tightness criterion.}

\purple{\textbf{Tightness of $N_3^{d,T}$.} First, we have $N^{d,T}_3(0,0)=-\eta f(0) \sum_{j=1}^d \big(W(0,0)W(0,\frac{j}{d})-A(0,\frac{j}{d})\big) \Bar{\Theta}^{d,T}(0,\frac{j}{d})$. Under assumptions in Theorem \ref{thm:LLN}, we have \begin{align*}
    \mb{P}(|N^{d,T}_3(0,0)| >N )&\le N^{-2} \eta^2 f(0)^2 \sum_{j=1}^d \big( B(0,\frac{j}{d},0,\frac{j}{d})-A(0,\frac{j}{d})^2 \big) \sum_{j=1}^d \mb{E}[|\Delta\theta^{0}_j|^2] \\
    &\le (C_2^2+C_5) \lv f \rv_{\infty}^2 N^{-2}\eta^2 d^2 R^2  .
\end{align*}
Since $\eta d T\le C_{s,1}$ for a uniform constant $C_{s,1}$, $\mb{P}(|N^{d,T}_3(0,0)| >N ) \lesssim N^{-2}\to 0$ as $N\to\infty$. Therefore, $\{N^{d,T}_3(0,0)\}_{d\ge 1, T>0}$ is tight in the probability space. Next, we observe that for any $0\le t_1<t_2\le \lfloor sT \rfloor-1$ and $i_1,i_2\in [d]$,
\begin{align*}
    &\qquad N^{d,T}_3(\frac{t_2}{T},\frac{i_2}{d})-N^{d,T}_3(\frac{t_1}{T},\frac{i_1}{d}) \\
    &= \eta \sum_{t=t_1}^{t_2-1}\sum_{j=1}^d f(\frac{t}{T})\big(W(\frac{t}{T},\frac{i_2}{d})W(\frac{t}{T},\frac{j}{d})-A(\frac{i_2}{d},\frac{j}{d})\big)\Theta^{d,T}(\frac{t}{T},\frac{j}{d})\\
    &\ +\eta \sum_{t=0}^{t_1-1}\sum_{j=1}^df(\frac{t}{T})\bigg(W(\frac{t}{T},\frac{i_2}{d})W(\frac{t}{T},\frac{j}{d})-W(\frac{t}{T}, \frac{i_1}{d})W(\frac{t}{T},\frac{j}{d})\\
    &\qquad\qquad\qquad\qquad\qquad-A(\frac{i_2}{d},\frac{j}{d})+A(\frac{i_1}{d},\frac{j}{d})\bigg)\Theta^{d,T}(\frac{t}{T},\frac{j}{d}).
\end{align*}
It follows from the arguments of bounding $\mb{E}[S_{i,2}]$ and $\mb{E}[S_{i,j,2}]$ in the proofs of Proposition \ref{prop:L2 tightness in time LLN} and Proposition \ref{prop:L2 tightness in time and space LLN} that
\begin{align*}
    \mb{E}[| N^{d,T}_3(\frac{t_2}{T},\frac{i_2}{d})-N^{d,T}_3(\frac{t_1}{T},\frac{i_1}{d}) |^4]&\lesssim \eta^4\lv f \rv_{\infty}^4 d^4 (t_2-t_1)^2  +\eta^4\lv f \rv_{\infty}^4 d^4 t_1^2 \big(\frac{i_2-i_1}{d}\big)^4 \\
   &\lesssim (\eta d T)^4 \big(\frac{t_2-t_1}{T}\big)^2 + (\eta d T)^4 \big(\frac{i_2-i_1}{d}\big)^4.
\end{align*}
Since $N^{d,T}_3$ is piecewise linear in both variables and $\eta d T\le C_{s,1}$ for all $d\ge 1, T>0$, for all $0\le s_1<s_2\le \tau$ and $0\le x_1<x_2\le 1$, we have
\begin{align*}
    {\mb{E}\big[ \big| N_3^{d,T}(s_1,x_1)-N_3^{d,T}(s_2,x_2)\big|^4 \big]}
    &\lesssim  (s_2-s_1)^2 + (x_1-x_2)^4.
\end{align*}
The tightness of $N^{d,T}_3$ in $C([0,\tau];C([0,1]))$ follows from the Kolmogorov's tightness criterion.
}    
\end{proof}

\section{Proofs for the fluctuations}\label{sec:FluctuationProofs}

\purple{Before we introduce the proofs for Theorem \ref{thm:masterclttheorem}, we first look at the iteration formula for the fluctuations and establish the tightness of the fluctuations. According to the definition of $\Theta^{d,T}$, under the low-noise setup in Theorem \ref{thm:mainmean}, for any $0\le t\le \lfloor \tau T \rfloor-1\coloneqq N-1$, $i\in [d]$, we have}

\purple{\begin{align}\label{eq:recurssion for fluctuation}
\begin{aligned}
    &~~U^{d,T}(\frac{t+1}{T},\frac{i}{d})-U^{d,T}(\frac{t}{T},\frac{i}{d}) \\
    =&- \eta \sum_{j=1}^d W(\frac{t}{T},\frac{i}{d})W(\frac{t}{T},\frac{j}{d})U^{d,T}(\frac{t}{T},\frac{j}{d}) \\
     &-\eta \gamma \sum_{j=1}^d \big(W(\frac{t}{T},\frac{i}{d})W(\frac{t}{T},\frac{j}{d})-\mb{E}\big[W(\frac{t}{T},\frac{i}{d})W(\frac{t}{T},\frac{j}{d})\big]\big)\Theta(\frac{t}{T},\frac{j}{d}) \\
     &-\gamma \big(\Theta(\frac{t+1}{T},\frac{i}{d})- \Theta(\frac{t}{T},\frac{i}{d})+\eta \sum_{j=1}^d A(\frac{i}{d},\frac{j}{d})\Theta(\frac{t}{T},\frac{j}{d}) \big) \\
     &+ \eta \gamma W(\frac{t}{T},\frac{i}{d})\varepsilon^t.
\end{aligned}
\end{align}}
\subsection{Tightness of $\{U^{d,T}\}_{d\ge 1,T>0}$}\label{sec:Tightness proofs CLT}

\begin{prop}\label{prop:L2 tightness in time CLT} \purple{Let Assumptions \ref{ass:tightness in space LLN} and \ref{ass:continuous data} hold and further suppose the initial conditions in Theorem \ref{thm:mainmean} and Theorem \ref{thm:mainfluc} are satisfied.} If there exists a uniform positive constant $C_{s,2}$ such that \textcolor{black}{for all $d\ge 1,T>0$ along the limiting sequence}, $\max(\gamma T^{-\frac{1}{2}}, \gamma d^{-1}, \gamma \sigma d^{-1}T^{-\frac{1}{2}})\le C_{s,2}$, then there exists a uniform constant $C>0$ such that for any $d\ge 1,T>0$, any $i\in [d]$ and any $0\le t_1<t_2\le \lfloor{\tau T}\rfloor$,
\begin{align}\label{eq:L2 tightness in time CLT}
    \mb{E}\big[\big|U^{d,T}(\frac{t_2}{T},\frac{i}{d})-U^{d,T}(\frac{t_1}{T},\frac{i}{d})\big|^4\big]\le C  \left(\frac{t_2-t_1}{T}\right)^2  .
\end{align}
\end{prop}

\begin{proof}[Proof of Proposition \ref{prop:L2 tightness in time CLT}.]
From \eqref{eq:recurssion for fluctuation}, we have that for any $0\le t_1<t_2\le N\coloneqq\lfloor \tau T \rfloor$, 
\begin{align}\label{eq:difference in time fluctuation}
    & \quad U^{d,T}(\frac{t_2}{T},\frac{i}{d})-U^{d,T}(\frac{t_1}{T},\frac{i}{d}) \nonumber\\
    &=- \underbrace{\eta \sum_{t=t_1}^{t_2-1}\sum_{j=1}^d A(\frac{i}{d},\frac{j}{d})U^{d,T}(\frac{t}{T},\frac{j}{d})}_{M_{i,1}} \\
    &\quad-\underbrace{ \eta  \sum_{t=t_1}^{t_2-1} \sum_{j=1}^d \big(W(\frac{t}{T},\frac{i}{d})W(\frac{t}{T},\frac{j}{d})-\mb{E}\big[W(\frac{t}{T},\frac{i}{d})W(\frac{t}{T},\frac{j}{d})\big]\big)U^{d,T}(\frac{t}{T},\frac{j}{d})}_{M_{i,2}} \nonumber\\
    &\quad -\underbrace{\eta \gamma \sum_{t=t_1}^{t_2-1} \sum_{j=1}^d \big(W(\frac{t}{T},\frac{i}{d})W(\frac{t}{T},\frac{j}{d})-\mb{E}\big[W(\frac{t}{T},\frac{i}{d})W(\frac{t}{T},\frac{j}{d})\big]\big)\Theta(\frac{t}{T},\frac{j}{d}) }_{M_{i,3}}\nonumber\\
    &\quad-\underbrace{\gamma \big(\Theta(\frac{t_2}{T},\frac{i}{d})- \Theta(\frac{t_1}{T},\frac{i}{d})+\eta \sum_{t=t_1}^{t_2-1}\sum_{j=1}^d A(\frac{i}{d},\frac{j}{d})\Theta(\frac{t}{T},\frac{j}{d}) \big)}_{M_{i,4}} +\underbrace{ \eta \gamma \sum_{t=t_1}^{t_2-1} W(\frac{t}{T},\frac{i}{d})\varepsilon^t }_{M_{i,5}} \nonumber.
\end{align}
Therefore we have
\begin{align}\label{eq:4th moment difference fluctuation}
     \mb{E}\big[\big|U^{d,T}(\frac{t_2}{T},\frac{i}{d})-U^{d,T}(\frac{t_1}{T},\frac{i}{d})\big|^4\big]\le 125\sum_{j=1}^5 \mb{E}[M_{i,j}^4].
\end{align}
Next we will bound the expectation of the right-hand side term by term. Many terms can be estimated using the proof of Proposition \ref{prop:L2 tightness in time LLN}. Define $n^d_t\coloneqq\frac{1}{d}\sum_{i=1}^d \mb{E}\big[\big| U^{d,T}(\frac{t}{T},\frac{i}{d}) \big|^4 \big] $. $\mb{E}[M_{i,1}^4]$ can be estimated similar to $\mb{E}[S_{i,1}]$ in the proof of Proposition \ref{prop:L2 tightness in time LLN}, and we have
\begin{align*}
    \mb{E}\big[ M_{i,1}^4 \big]&=\eta^4\mb{E}\big[ \sum_{r_1,r_2,r_3,r_4=t_1}^{t_2-1}\sum_{j_1,j_2,j_3,j_4=1}^d \prod_{k=1}^4 A(\frac{i}{d},\frac{j_k}{d})\prod_{k=1}^4 U^{d,T}(\frac{r_k}{T},\frac{j_k}{d})\big] \\
    &\le  C_2^4 \eta^4 d^4 (t_2-t_1)^3\sum_{t=t_1}^{t_2-1} n^d_t.
\end{align*}
$\mb{E}[M_{i,2}^4]$ can be estimated similar to $\mb{E}[S_{i,2}]$ in the proof of Proposition \ref{prop:L2 tightness in time LLN}, and we have
\begin{align*}
    \mb{E}\big[ \purple{M}_{i,2}^4 \big]&\le 3\left( C_8'+C_{5}'^2 \right)\eta^4 d^4 (t_2-t_1)\sum_{t=t_1}^{t_2-1} n^d_t .
\end{align*}
Similar to how we bound $\mb{E}[M_{i,2}^4]$, we have
\begin{align*}
    \mb{E}\big[ M_{i,3}^4 \big]&\le 3\big(C_8' +C_{5}'^2\big) (t_2-t_1) \eta^4 \gamma^4\mb{E} \big[ \sum_{t=t_1}^{t_2-1}\sum_{j_1,j_2,j_3,j_4=1}^d \prod_{k=1}^4 \Theta(\frac{t}{T},\frac{j_k}{d})\big] \\
    &\le 3\left( C_8'+C_{5}'^2 \right)\lv \Theta \rv_\infty^4  \gamma^4 \eta^4 d^4 (t_2-t_1)^2\\
    &=O(\gamma^4\eta^4 d^4(t_2-t_1)^2),
\end{align*}
where the second inequality follows from the fact that $$\lv \Theta \rv_\infty\coloneqq\sup_{s\in [0,\tau],x\in [0,1]}|\Theta(s,x)|<\infty.$$ Next due to the fact that $A,\Theta$ are $C^1$ in all variables, we have
\begin{align*}
    \mb{E}\left[ M_{i,4}^2 \right]&=\gamma^4 \eta^4 d^4 T^4 \big( -\int_{\frac{t_1}{T}}^{\frac{t_2}{T}} \int_0^1 A(\frac{i}{d},y)\Theta(s,y)\mathrm{d}y\mathrm{d}s +\frac{1}{dT} \sum_{t=t_1}^{t_2-1}\sum_{j=1}^d A(\frac{i}{d},\frac{j}{d})\Theta(\frac{t}{T},\frac{j}{d}) \big)^4 \\
    &=O( \gamma^4 \eta^4 {(t_2-t_1)^4} )+O(\gamma^4 \eta^4 d^4 {(t_2-t_1)^4}{T^{-4}}).
\end{align*}
$\mb{E}[M_{i,5}^4]$ can be estimated similar to $\mb{E}[S_{i,3}]$ in the proof of Proposition \ref{prop:L2 tightness in time LLN}, and we have
\begin{align*}
    \mb{E}\left[ M_{i,5}^4 \right]&\le (C_{5} C_1+C_2^2 )\sigma^4\eta^4 \gamma^4  (t_2-t_1)^2=O(\gamma^4 \eta^4 \sigma^4 (t_2-t_1)^2 ).
\end{align*}
Combining all the estimations and pick $t_2=t,t_1=0$ in \eqref{eq:4th moment difference fluctuation},
\begin{align*}
    n^d_t &=\frac{1}{d}\sum_{i=1}^d \mb{E}\big[ \big| U^{d,T}(\frac{t}{T},\frac{i}{d}) \big|^4 \big] \le 8n^d_0+ \frac{8}{d}\sum_{i=1}^d \mb{E}\big[ \big| U^{d,T}(\frac{t}{T},\frac{i}{d})-U^{d,T}(0,\frac{i}{d}) \big|^4 \big] \\
    &\le 8 n^d_0+2000\big( C_2^4 +3( C_8+C_{5}^2)  \big)\eta^4 d^4t^3\sum_{k=0}^{t-1} n^d_k+C_{d,T}(t),
\end{align*}
where $C_{d,T}(t)\coloneqq O(\gamma^4\eta^4 d^4t^2+ \gamma^4 \eta^4 {t^4} +\gamma^4 \eta^4 d^4 {t^4}{T^{-4}}+\gamma^4 \eta^4 \sigma^4 t^2 )$. According to the discrete Gronwall's inequality, we have
\begin{align*}
    n^d_t&\le C_{d,T}(t)+\Tilde{C}\eta^4 d^4t^3\sum_{k=0}^{t-1}C_{d,T}(k) \exp\big(\Tilde{C}\eta^4 d^4t^3(t-k-1)\big), 
\end{align*}
where $\Tilde{C}=2000( \purple{C_2^4}+3( C_8+C_{5}^2) )$. Since $\eta d T\le C_{s,1}$, $2000\big( C_2^4+3( C_8+C_{5}^2) \big)\eta^4 d^4t^3(t-k-1)\le C_{\tau}C_{s,1}^4\tau^4$ for any $0\le k<t$, $t\le N$ and $C_{\tau,1}$ is a constant independent of $d,T,\sigma$. Since $C_{d,T}(k)$ is increasing with $k$, there exists a constant $C_{\tau,2}$ independent of $d,T,\sigma$ such that
\begin{align}\label{eq:upper bound for L4 average CLT}
    n^d_t&\le C_{d,T}(t)+ 2000\big( C_2^4 +3( C_8+C_{5}^2) \big)\exp\left( C_{\tau,1}C_{s,1}^4\sigma^{-4}\tau^4 \right) \eta^4 d^4 t^3  \sum_{k=0}^{t-1} C_{d,T}(k)\nonumber\\
    &\le C_{\tau,2} C_{d,T}(t).
\end{align}
Plug \eqref{eq:upper bound for L4 average CLT} into \eqref{eq:4th moment difference fluctuation} and take expectations. Then, we get
\begin{align*}
    &\mb{E}\big[ \big| U^{d,T}(\frac{t_2}{T},\frac{i}{d})-U^{d,T}(\frac{t_1}{T},\frac{i}{d}) \big|^4 \big]\\
    \le&  125\big( C_2^4 +3( C_8'+C_{5}'^2)  \big)\eta^4 d^4(t_2-t_1)^3\sum_{t=t_1}^{t_2-1} n^d_t +C_{d,T}(t_2-t_1)\\
    \le& 125C_{\tau,2}\big( C_2^4 +3( C_8'+C_{5}'^2)  \big)\eta^4 d^4(t_2-t_1)^3\sum_{t=t_1}^{t_2-1} C_{d,T}(t) +C_{d,T}(t_2-t_1).
\end{align*}
Observe that since $\max(\eta dT,\eta \sigma T^{\frac{1}{2}})\le C_{s,1}$, we have,
\begin{align*}
    C_{d,T}(t)=\big( \frac{t}{T} \big)^2   O( \gamma^4 T^{-2}+\gamma^4 d^{-4}+{ \gamma^4 \sigma^4 d^{-4}T^{-2}} ).
\end{align*}
 Furthermore, note that 
 $$
 \max(\gamma T^{-\frac{1}{2}}, \gamma d^{-1},\gamma \sigma d^{-1}T^{-\frac{1}{2}})\le C_{s,2}\quad\text{and}\quad
 C_{d,T}(t)\le C_{\tau,3}C_{s,2}^4\big( \frac{t}{T} \big)^2
 $$ 
 for some positive constant $C_{\tau,3}$ independent of $d,T,\sigma,\gamma$. Therefore,
\begin{align*}
     &\mb{E}\big[ \big| U^{d,T}(\frac{t_2}{T},\frac{i}{d})-U^{d,T}(\frac{t_1}{T},\frac{i}{d}) \big|^4 \big]\\
     \le& 125C_{\tau,2}\big( C_2^4 +3( C_8'+C_{5}'^2)  \big)\eta^4 d^4(t_2-t_1)^4 C_{\tau,3}C_{s,2}^4 \tau^2 +C_{\tau,3} \big( \frac{t_2-t_1}{T}\big)^2\\
     \le &125C_{\tau,2}\big( C_2^4 +3( C_8'+C_{5}'^2)  \big) C_{s,1}^4 C_{\tau,3}C_{s,2}^4 \tau^2 \big(\frac{t_2-t_1}{T}\big)^4 +C_{\tau,3} \big( \frac{t_2-t_1}{T}\big)^2.
\end{align*}
Last \eqref{eq:L2 tightness in space CLT} is proved since $\big(\frac{t_2-t_1}{T}\big)^2\le \tau^2$.
\end{proof}

\begin{prop}\label{prop:L2 tightness in time and space CLT} \purple{Let Assumptions \ref{ass:tightness in space LLN} and \ref{ass:continuous data} hold and further suppose the initial conditions in Theorem \ref{thm:mainmean} are satisfied}. For any $d,T>0$, then there exists a constant $C>0$ independent of $d,T$ such that for any $d\ge 1,T>0$, any $i,j\in [d]$
and any \textcolor{black}{$0\le t_1<t_2\le \lfloor{\tau T}\rfloor$},
\textcolor{black}{
\small{
\begin{align}\label{eq:L2 tightness in space and time CLT}
\mb{E}\big[\big| \big(U^{d,T}(\frac{t_2}{T},\frac{i}{d})-U^{d,T}(\frac{t_2}{T},\frac{j}{d})\big)-\big(U^{d,T}(\frac{t_1}{T},\frac{i}{d})-U^{d,T}(\frac{t_1}{T},\frac{j}{d})\big) \big|^4\big] &\le C \big( \frac{i-j}{d} \big)^4 \big( \frac{t_2-t_1}{T} \big)^2.
\end{align}
}
Furthermore, if the initial conditions in Theorem \ref{thm:mainfluc} hold}, then for any $i,j\in [d]$ and any $0\le t\le \lfloor \tau T\rfloor$,
\begin{align}\label{eq:L2 tightness in space CLT}
    \mb{E}\big[\big| U^{d,T}(\frac{t}{T},\frac{i}{d})-U^{d,T}(\frac{t}{T},\frac{j}{d})\big|^4\big]\le 8\left( C\tau^2 +M^4 \right)\big( \frac{i-j}{d} \big)^4,
\end{align}
\purple{where $M$ is the uniform constant in the initial condition in Theorem \ref{thm:mainfluc}}.
\end{prop}

\begin{proof}[Proof of Proposition \ref{prop:L2 tightness in time and space CLT}.]  From \eqref{eq:fluctuation} we have for any $0\le t_1<t_2\le N\coloneqq\lfloor \tau T \rfloor$ and any $i,j\in [d]$,
{\small
\begin{align*} 
    &\quad \big(U^{d,T}(\frac{t_2}{T},\frac{i}{d})-U^{d,T}(\frac{\purple{t_2}}{T},\frac{j}{d})\big)-\big(U^{d,T}(\frac{\purple{t_1}}{T},\frac{i}{d})-U^{d,T}(\frac{\purple{t_1}}{T},\frac{j}{d})\big)\\
    &= -\underbrace{\eta \sum_{t=t_1}^{t_2-1}\sum_{l=1}^d \big( A(\frac{i}{d},\frac{l}{d})-A(\frac{j}{d},\frac{l}{d}) \big)U^{d,T}(\frac{t}{T},\frac{l}{d})}_{M_{i,j,1}} \\
    &\quad-\underbrace{\eta \sum_{t=t_1}^{t_2-1}\sum_{l=1}^d \left( x^t_ix^t_l-\mb{E}\left[x^t_ix^t_l\right]-x^t_jx^t_l+\mb{E}\left[x^t_jx^t_l\right] \right)U^{d,T}(\frac{t}{T},\frac{l}{d})}_{M_{i,j,2}} \\
    &\quad +\underbrace{\alpha\gamma \int_{\frac{t_1}{T}}^{\frac{t_2}{T}}\int_0^1 \big( A(\frac{i}{d},y)-A(\frac{j}{d},y) \big)\Theta(s,y)dyds-\eta \gamma \sum_{l=1}^d \big( A(\frac{i}{d},\frac{l}{d})-A(\frac{j}{d},\frac{l}{d}) \big)\Theta(\frac{t}{T},\frac{l}{d})}_{M_{i,j,3}} \\
    &\quad  -\underbrace{\eta\gamma\sum_{t=t_1}^{t_2-1} \sum_{l=1}^d \left( x^t_ix^t_l-\mb{E}\left[x^t_ix^t_l\right]-x^t_jx^t_l+\mb{E}\left[x^t_jx^t_l\right] \right)\Theta(\frac{t}{T},\frac{l}{d})}_{M_{i,j,4}}+\underbrace{\eta \gamma \sum_{t=t_1}^{t_2-1} \big(x^t_i-x^t_j\big) \varepsilon^t}_{M_{i,j,5}}.
\end{align*}
}
Therefore if we define $\Delta^{d,T}(t,i,j)\coloneqq U^{d,T}(\frac{t}{T},\frac{i}{d})-U^{d,T}(\frac{t}{T},\frac{j}{d})$,  we have
\begin{align}\label{eq:sqaure difference in time and space CLT}
\begin{aligned}
   &\mb{E} \big[ \big| \Delta^{d,T}(t_2,i,j)-\Delta^{d,T}(t_1,i,j) \big|^4 \big]  \le  125 \sum_{l=1}^5 \mb{E}\big[ M_{i,j,l}^{\purple{4}} \big].
\end{aligned}
\end{align}
Next we estimate the right-hand side of \eqref{eq:sqaure difference in time and space CLT} term by term and most terms are bounded based on the proof of Proposition \ref{prop:L2 tightness in time and space LLN}. $\mb{E}[M_{i,j,1}^{\purple{4}}]$ can be upper bounded similar to $\mb{E}[S_{i,j,1}]$ in the proof of Proposition \ref{prop:L2 tightness in time and space LLN}, and we have
\begin{align*}
    \mb{E}\big[ M_{i,j,1}^{\purple{4}} \big]&\le C_{3}^4  \big( \frac{i-j}{d} \big)^4 (t_2-t_1)^3 \eta^4 d^4 \sum_{t=t_1}^{t_2-1} n^d_t= \big( \frac{i-j}{d} \big)^4 \big(\frac{t_2-t_1}{T}\big)^2 O(\eta^4 d^4 T^3\sum_{t=t_1}^{t_2-1} n^d_t).
\end{align*}
$\mb{E}[M_{i,j,2}^{\purple{4}}]$ can be upper bounded similar to $\mb{E}[S_{i,j,2}]$ in the proof of Proposition \ref{prop:L2 tightness in time and space LLN}, and we have
\begin{align*}
    \mb{E}\big[ M_{i,j,2}^{\purple{4}} \big]&\le \big(C_9'^4 +C_{6}'^4 \big) \big( \frac{i-j}{d} \big)^4 (t_2-t_1) \eta^4 d^4 \sum_{t=t_1}^{t_2-1} n^d_t\\
    &=\big( \frac{i-j}{d} \big)^4 \big(\frac{t_2-t_1}{T}\big)^2 O(\eta^4 d^4 T^2\sum_{t=t_1}^{t_2-1} n^d_t).
\end{align*}
Since $\Theta$ is $C^1$ in both variables and $\partial_1 A(x,\cdot)\in C^1([0,1])$ for any $x\in [0,1]$, $\mb{E}[M_{i,j,3}^{\purple{4}}]$ can be estimated as
\begin{align*}
    &\quad\mb{E}\big[ M_{i,j,3}^{\purple{4}} \big]\\
    &=\alpha^4 \gamma^4 \big( \int_{\frac{i}{d}}^{\frac{j}{d}}\int_{\frac{t_1}{T}}^{\frac{t_2}{T}}\int_0^1 \partial_1 A(z,y) \Theta(s,y)\mathrm{d}y\mathrm{d}s \mathrm{d}z-\int_{\frac{i}{d}}^{\frac{j}{d}}\frac{1}{dT}\sum_{l=1}^d \partial_1 A(z,\frac{l}{d})\Theta(\frac{t}{T},\frac{l}{d}) \mathrm{d}z \big) ^4\\
    &=O\bigg(\gamma^4 \big(\frac{i-j}{d}\big)^4 \big( \frac{t_2-t_1}{T} \big)^4 \frac{1}{T^4} \bigg)+O\bigg(\gamma^4 \big(\frac{i-j}{d}\big)^4 \big( \frac{t_2-t_1}{T} \big)^4 \frac{1}{d^4} \bigg)\\
    &=\big(\frac{i-j}{d}\big)^4 \big( \frac{t_2-t_1}{T} \big)^4 O\bigg(  \gamma^4T^{-4}+\gamma^4d^{-4} \bigg).
\end{align*}
Next, $\mb{E}[M_{i,j,4}^{\purple{4}}]$ can be estimated similar to $\mb{E}[M_{i,j,2}^{\purple{4}}]$: 
\begin{align*}
    \mb{E}\big[ M_{i,j,4}^{\purple{4}} \big]&\le C_9'^4\eta^4  \gamma^4 \big( \frac{i-j}{d} \big)^4 \sum_{t=t_1}^{t_2-1}\sum_{l_1,l_2,l_3,l_4=1}^d \mb{E}\big[ \big|\prod_{k=1}^4\Theta(\frac{t}{T},\frac{l_k}{d})  \big| \big] \\
    &\ + C_{6}'^4\eta^4 \gamma^4 \big( \frac{i-j}{d} \big)^4 \sum_{t',t=t_1}^{t_2-1}\sum_{l_1,l_2,l_3,l_4=1}^d \mb{E}\big[ \big|\prod_{k=1}^2 \Theta(\frac{t}{T},\frac{l_k}{d}) \prod_{k=3}^4 \Theta(\frac{t'}{T},\frac{l_k}{d}) \big| \big]\\
    &\le \big( C_9'^4+C_{6}'^4 \big)   \lv \Theta \rv_\infty^4 \big( \frac{i-j}{d} \big)^4 \eta^4\gamma^4 d^4  {(t_2-t_1)^2} \\
    &= \big( \frac{i-j}{d} \big)^4  \big(\frac{t_2-t_1}{T}\big)^2 O(\eta^4 \gamma^4 d^4 T^2).
    \end{align*}
Last $\mb{E}[M_{i,j,5}^{\purple{4}}]$ can be upper bounded similar to $\mb{E}[S_{i,j,3}]$ in the proof of Proposition \ref{prop:L2 tightness in time and space LLN}, and we have
\begin{align*}
    \mb{E}\big[ M_{i,j,5}^{\purple{4}}\big]&\le \big( C_1 C_{7}^4+C_{4}^4 \big)  \big( \frac{i-j}{d} \big)^4 \big(\frac{t_2-t_1}{T}\big)^2 \eta^4 \gamma^4 \sigma^4 T^2.
\end{align*}
In the proof of Proposition \ref{prop:L2 tightness in time CLT}, we have shown that $n^d_t\le C_{\tau,2}C_{d,T}(t)\le C_{\tau,2}C_{\tau,3}C_{s,2}^4 \big(\frac{t}{T}\big)^2=O(1)$. Therefore, apply this estimation to \eqref{eq:sqaure difference in time and space CLT}, along with the estimations on $\mb{E}[M_{i,j,\cdot}^{\purple{4}}]$'s, and we have
\begin{align*}
      &\qquad \mb{E}\big[\big| \Delta^{d,T}(t_2,i,j)-\Delta^{d,T}(t_1,i,j) \big|^4\big] \\
      &\le \big( \frac{i-j}{d} \big)^4 \big(\frac{t_2-t_1}{T}\big)^2 O\bigg(\eta^4 d^4 T^3 \sum_{t=t_1}^{t_2-1}n^d_t+\gamma^4 T^{-4}+\gamma^4 d^{-4}+\eta^4 \gamma^4 d^4 T^2+\eta^4 \gamma^4 \sigma^4 T^2\bigg)\\
      &=\big( \frac{i-j}{d} \big)^4 \big(\frac{t_2-t_1}{T}\big)^2 O\big(\eta^4 d^4 T^4 +\gamma^4 T^{-4}+\gamma^4 d^{-4}+\eta^4 \gamma^4 d^4 T^2+\eta^4 \gamma^4 \sigma^4 T^2\big).
\end{align*}
Since $\max(\eta d T,\eta \sigma T^{\frac{1}{2}})\le C_{s,1}$ and $\max(\gamma T^{-\frac{1}{2}}, \gamma d^{-1}, \gamma \sigma d^{-1}T^{-\frac{1}{2}} )\le C_{s,2}$, 
\[
O\big(\eta^4 d^4 T^4 +\gamma^4 T^{-4}+\gamma^4 d^{-4}+\eta^4 \gamma^4 d^4 T^2+\eta^4 \gamma^4 \sigma^4 T^2\big)=O(1),
\]
and therefore there exists a uniform constant $C$ such that 
\begin{align*}
      \mb{E}\big[\big| \Delta^{d,T}(t_2,i,j)-\Delta^{d,T}(t_1,i,j) \big|^4\big] &\le C \big( \frac{i-j}{d} \big)^4 \big( \frac{t_2-t_1}{T} \big)^2.
\end{align*}
Picking $t_2=t$ and $t_1=0$, we get
\begin{align*}
    &\quad\mb{E}\big[ \big| U^{d,T}(\frac{t}{T},\frac{i}{d})-U^{d,T}(\frac{t}{T},\frac{j}{d}) \big|^4 \big]\\
    &\le 8\mb{E}\big[ \big| U^{d,T}(0,\frac{i}{d})-U^{d,T}(0,\frac{j}{d}) \big|^4 \big]+8C\big( \frac{i-j}{d} \big)^4 \big( \frac{t_2-t_1}{T} \big)^2 \\
    &\le  8 \big( C\tau^2+M^4 \big)\big( \frac{i-j}{d} \big)^4, 
    \end{align*}
    \purple{where the last inequality follows from the initial condition in Theorem \ref{thm:mainfluc}.}
\end{proof}

\begin{proof}[Proof of Theorem \ref{thm:tightness of continuous interpolation CLT}]
    When $\eta=\frac{\alpha}{dT}$ and $\sigma=o(dT^{\frac{1}{2}})$, there exists a uniform constant $C_{s,1}$ such that \textcolor{black}{for all $d\ge 1,T>0$ along the limiting sequence}, 
 $\max(\eta d T, \eta \sigma T^{\frac{1}{2}})\le C_{s,1}$. Therefore, given Proposition \ref{prop:L2 tightness in time CLT} and Proposition \ref{prop:L2 tightness in time and space CLT}, the proof of Theorem \ref{thm:tightness of continuous interpolation CLT} is exactly the same as the proof of Theorem \ref{thm:tightness of continuous interpolation LLN}. Hence, we skip the details.
\end{proof}
\subsection{Limit identification}\label{sec:cltproof}

\begin{proof}[Proof of Theorem~\ref{thm:masterclttheorem}]
For any $0\le t\le \lfloor \tau T \rfloor-1\coloneqq N-1$, $i\in \purple{[d-1]}$, according to \eqref{eq:approximation to Theta LLN} and \eqref{eq:fluctuation}, we have for all $x\in [\frac{i}{d},\frac{i+1}{d}]$,
\begin{align*}
    &\qquad U^{d,T}(\frac{t+1}{T},x)-U^{d,T}(\frac{t}{T},x)\\
    &=- \eta(i+1-dx) \sum_{j=1}^d W(\frac{t}{T},\frac{i}{d})W(\frac{t}{T},\frac{j}{d})U^{d,T}(\frac{t}{T},\frac{j}{d}) \\
    &\quad- \eta(dx-i) \sum_{j=1}^d W(\frac{t}{T},\frac{i+1}{d})W(\frac{t}{T},\frac{j}{d})U^{d,T}(\frac{t}{T},\frac{j}{d}) \\
    &\quad -\eta \gamma (i+1-dx)\sum_{j=1}^d \big(W(\frac{t}{T},\frac{i}{d})W(\frac{t}{T},\frac{j}{d})-\mb{E}\big[W(\frac{t}{T},\frac{i}{d})W(\frac{t}{T},\frac{j}{d})\big]\big)\Theta(\frac{t}{T},\frac{j}{d}) \\
    &\quad -\eta \gamma (dx-i)\sum_{j=1}^d \big(W(\frac{t}{T},\frac{i+1}{d})W(\frac{t}{T},\frac{j}{d})-\mb{E}\big[W(\frac{t}{T},\frac{i+1}{d})W(\frac{t}{T},\frac{j}{d})\big]\big)\Theta(\frac{t}{T},\frac{j}{d}) \\
    &\quad-\gamma \big(\Theta(\frac{t+1}{T},x)- \Theta(\frac{t}{T},x)+\eta \sum_{j=1}^d A(x,\frac{j}{d})\Theta(\frac{t}{T},\frac{j}{d}) \big) \\
    &\quad + \eta \gamma(i+1-dx) W(\frac{t}{T},\frac{i}{d})\varepsilon^t+ \eta\gamma (dx-i)  W(\frac{t}{T},\frac{i+1}{d})\varepsilon^t. 
\end{align*}
If we apply a bounded smooth test function $f:[0,\tau]\to\mb{R}$ on both sides, we get for any $s\in(0,\tau)$:
\begin{align}\label{eq:CLT test function applied sum}
    &\sum_{t=0}^{\lfloor sT\rfloor-1} f(\frac{t}{T})\big( U^{d,T}(\frac{t+1}{T},x)-U^{d,T}(\frac{t}{T},x) \big) \nonumber \\
    =&-(i+1-dx)\sum_{k=1}^3 P_{i,k}^s-(dx-i)\sum_{k=1}^3 P_{i+1,k}^s-P_4^s , 
    \end{align}
    where for any $i\in [d]$,
    \begin{align*}
    P_{i,1}^s&=\eta \sum_{t=0}^{\lfloor sT\rfloor-1}\sum_{j=1}^df(\frac{t}{T}) W(\frac{t}{T},\frac{i}{d})W(\frac{t}{T},\frac{j}{d})U^{d,T}(\frac{t}{T},\frac{j}{d}), \\
     P_{i,2}^s&=-\eta\gamma  \sum_{t=0}^{\lfloor sT\rfloor-1}f(\frac{t}{T})W(\frac{t}{T},\frac{i}{d})\varepsilon^t ,\\
    P_{i,3}^s&=\eta \gamma  \sum_{t=0}^{\lfloor sT\rfloor-1}\sum_{j=1}^df(\frac{t}{T}) \big(W(\frac{t}{T},\frac{i}{d})W(\frac{t}{T},\frac{j}{d})-\mb{E}\big[W(\frac{t}{T},\frac{i}{d})W(\frac{t}{T},\frac{j}{d})\big]\big)\Theta(\frac{t}{T},\frac{j}{d}),\\
    P_{4}^s&=\gamma  \sum_{t=0}^{\lfloor sT\rfloor-1}f(\frac{t}{T})\big(\Theta(\frac{t+1}{T},x)- \Theta(\frac{t}{T},x)+\eta \sum_{j=1}^d  A(x,\frac{j}{d})\Theta(\frac{t}{T},\frac{j}{d}) \big).
\end{align*}
Similar to the way we derive \eqref{eq:LHS of test function applied sum}, we can write the left-hand side of \eqref{eq:CLT test function applied sum} as
\begin{align}\label{eq:LHS CLT test function applied sum}
    \text{LHS}\eqref{eq:CLT test function applied sum}=- \int_0^{\frac{\lfloor sT \rfloor}{T}} U(u,x)f'(u)\mathrm{d}u+f(s)U(s,x)-f(0)U(0,x)+o(1).
\end{align}
Next, we look at the right-hand side of \eqref{eq:CLT test function applied sum}. There are 4 types of terms on the right-hand side of \eqref{eq:CLT test function applied sum}. 
\purple{Based on the same space-time interpolation introduced in Section~\ref{sec:Introduction}, for each $l=1,2,3$, $\{P_{i,l}^s\}$ can be interpolated as a function in $C([0,\tau];C([0,1]))$, denoted as $P^{d,T}_l$. $P_4^s$ is already continuous in variable $x$ and it can be interpolated as a function in $C([0,\tau];C([0,1]))$, denoted as $P^{d,T}_4$.} 
\purple{According to Lemma \ref{lem:joint tightness CLT}, $(U^{d,T}, P_1^{d,T}, P_2^{d,T}, P_3^{d,T}, P_4^{d,T})$ is tight in $C([0,\tau];C([0,1]))^{\otimes 5}$. Therefore, the sequential weak limit exists in $C([0,\tau];C([0,1]))^{\otimes 5}$. The sequential weak limit of $(U^{d,T}, P_1^{d,T}, P_2^{d,T}, P_3^{d,T}, P_4^{d,T})$ in $C([0,\tau];C([0,1]))^{\otimes 5}$ is derived based on the weak limit of each term. Note that for $U^{d,T}$, the weak limit, denoted as $U$, satisfies an equation based the weak limits of other terms, which are identified as follows:
\begin{itemize}
    \item  For $P_1^{d,T}$, weak limit is identified based on the continuous mapping theorem~\cite[Theorem 2.7]{billingsley2013convergence};
    \item  For $P_2^{d,T}$, $P_3^{d,T}$, to identify weak limit of an infinite dimensional process, it suffices to identify the weak limits for each finite dimensional weak limit. We identify the finite dimensional limit via pointwise central limit theorem.
    \item For $P_4^{d,T}$, we use the fact that convergence in probability to zero implies convergence in distribution to zero. 
\end{itemize}
}
Now, we look at the limit of each of the terms.
\begin{itemize}
    \item [(1)] For terms related to $P_{i,1}^s$ for some $i\in[d]$, we can deal with them akin to how we deal with $N_{i,1}^s,N_{2,i}^s$ in the proof of Theorem \ref{thm:LLN}. We get
    \begin{align*}
       -(i+1-dx) P_{i,1}^s-(dx-i) P_{i+1,1}^s = -\alpha \int_0^{\frac{\lfloor sT\rfloor}{T}}\int_0^1 f(u)A(x,y)U(u,y)\mathrm{d}y\mathrm{d}u+o(1).
    \end{align*}
    \item [(2)] For terms related to $P_{i,2}^s$ for some $i\in[d]$, we can deal with them akin to how we deal with $N_{i,2}^s$ in the proofs of Theorem \ref{thm:LLN}. \purple{There exists a centered Gaussian random field process $\{\xi_{2}(s,x)\}_{s\in [0,\tau],x\in [0,1]}$ such that for any $s_1,s_2\in [0,\tau]$ and $x,y\in [0,1]$,
    \begin{align*}
       \mb{E}[\xi_2(s_1,x)\xi_2(s_2,y)]=\sigma_2(s_1,s_2,x,y)\quad\text{with}\quad \sigma_2(s_1,s_2,x,y)= (s_1\wedge s_2) A(x,y),
    \end{align*}}
    and
    \begin{equation*}
        P_{i,2}^s=\left\{\begin{aligned}
            & \quad o(1) \qquad &\text{if } \gamma \sigma d^{-1}T^{-\frac{1}{2}}\to 0,\\
            & \alpha \beta \int_0^s f(u) \mathrm{d}\xi_2(u,\frac{i}{d}) +o(1) \qquad &\text{if } \gamma \sigma d^{-1}T^{-\frac{1}{2}}\to \beta\in (0,\infty).
        \end{aligned}
        \right.
    \end{equation*}
    Therefore
    \begin{itemize}
        \item [(2.1)] When $\gamma \sigma d^{-1}T^{-\frac{1}{2}}\to 0$, $-(i+1-dx) P_{i,2}^s-(dx-i) P_{i+1,2}^s\to 0$ in probability.
        \item [(2.2)] When  $\gamma \sigma d^{-1}T^{-\frac{1}{2}}\to \beta\in (0,\infty)$,  
        \[
        -(i+1-dx) P_{i,2}^s-(dx-i) P_{i+1,2}^s= \alpha \beta \int_0^{\frac{\lfloor sT \rfloor}{T}} f(u) \mathrm{d}\xi_2(u,x)+o(1).
        \]
    \end{itemize}
    
    \item [(3)] To study the terms related to $P_{i,3}$ for some $i\in [d]$, we first define 
    \[
    Z^{d,T}_{\gamma}(\frac{t}{T},\frac{j}{d})\coloneqq\frac{\gamma}{dT}\sum_{l=1}^d f(\frac{t}{T}) \big(W(\frac{t}{T},\frac{j}{d})W(\frac{t}{T},\frac{l}{d})-\mb{E}\big[W(\frac{t}{T},\frac{j}{d})W(\frac{t}{T},\frac{l}{d})\big]\big)\Theta(\frac{t}{T},\frac{l}{d}),
    \]
    for any $t\le \lfloor sT\rfloor$, $j\in [d]$. Then we can write
    \begin{align*}
       P_{i,3}^s= (\eta d T) \sum_{t=0}^{\lfloor sT\rfloor-1} Z^{d,T}_{\gamma}(\frac{t}{T},\frac{i}{d}).
    \end{align*}
    Notice that $\mb{E}\big[ Z^{d,T}_\gamma(\frac{t}{T},\frac{i}{d}) \big]=0$ and $\{Z^{d,T}_\gamma(\frac{t_1}{T},\frac{i}{d})\}_{i\in [d]}$ is independent of $\{Z^{d,T}_\gamma(\frac{t_2}{T},\frac{i}{d})\}_{i\in [d]}$ when $t_1\neq t_2$. For any $0\le s_1<s_2\le \tau $, $i,j\in [d]$, we compute
    \begin{align*}
        &\mb{E}\big[ P_{i,3}^{s_1}P_{j,3}^{s_2}\big]\\
        =&\left(\eta d T\right)^2\mb{E}\big[ \sum_{t=0}^{\lfloor s_1 T\rfloor-1} Z^{d,T}_{\gamma}(\frac{t}{T},\frac{i}{d})Z^{d,T}_{\gamma}(\frac{t}{T},\frac{j}{d}) \big] \\
        =&\left(\eta d T\right)^2 \frac{\gamma^2}{d^2 T^2} \sum_{t=0}^{\lfloor s_1 T\rfloor-1}\sum_{l_1,l_2=1}^d f(\frac{t}{T})^2 \Theta(\frac{t}{T},\frac{l_1}{d})\Theta(\frac{t}{T},\frac{l_2}{d})\Tilde{B}(\frac{i}{d},\frac{l_1}{d},\frac{j}{d},\frac{l_2}{d})\\
        =&\frac{\alpha^2 \gamma^2}{T^2}\sum_{t=0}^{\lfloor s_1 T\rfloor-1} \big( \int_0^1\int_0^1 f(\frac{t}{T})^2 \Theta(\frac{t}{T},x)\Theta(\frac{t}{T},y)\Tilde{B}(\frac{i}{d},x,\frac{j}{d},y)\mathrm{d}x\mathrm{d}y+o(1)\big),
    \end{align*}
    where $\Tilde{B}(x_1,x_2,x_3,x_4)=B(x_1,x_2,x_3,x_4)-A(x_1,x_3)A(x_2,x_4)$ for any $x_1,x_2,x_3,x_4\in[0,1]$. The first identity follows from independence, and the last identity follows from the fact that $f,\Theta,B$ are continuous and bounded. Therefore, we have
    \begin{align}\label{eq:correlation M}
    \begin{aligned}
   &\mb{E}\big[ P_{i,3}^{s_1}P_{j,3}^{s_2}\big]\\
   =&\gamma^2 T^{-1} \alpha^2  \int_0^{\frac{\lfloor s_1 T\rfloor}{T}} \int_0^1\int_0^1 f(u)^2 \Theta(u,x)\Theta(u,y)\Tilde{B}(\frac{i}{d},x,\frac{j}{d},y)\mathrm{d}x\mathrm{d}y\mathrm{d}u+o(\gamma^2 T^{-1}).
   \end{aligned}
    \end{align}
Now, we have the following observations.
    \begin{itemize}
        \item [(3.1)] If $\gamma =o(T^{\frac{1}{2}})$, then $P_{i,3}^s\to 0$ in probability and $ -(i+1-dx) P_{i,2}^s-(dx-i) P_{i+1,2}^s\to 0$ in probability.
        \item [(3.2)] If $\gamma T^{-\frac{1}{2}}\to \zeta\in (0,\infty)$, then, with \eqref{eq:correlation M}, we can first apply the Lindeberg-Feller CLT to $P_{i,3}$. Since 
    \begin{align*}
        &\quad\mb{E}\big[ Z_{{\gamma}}^{d,T}(\frac{t}{T},\frac{i}{d})^2 \big]\\
        &= (\frac{\gamma}{dT})^2 \sum_{l_1,l_2}^d f(\frac{t}{T})^2 \Theta(\frac{t}{T},\frac{l_1}{d})\Theta(\frac{t}{T},\frac{l_2}{d})\Tilde{B}(\frac{i}{d},\frac{l_1}{d},\frac{j}{d},\frac{l_2}{d})\\
        &=\frac{\zeta^2f(\frac{t}{T})^2}{T} \big( \int_0^1\int_0^1 \Theta(\frac{t}{T},z_1)\Theta(\frac{t}{T},z_2)\Tilde{B}(\frac{i}{d},z_1,\frac{i}{d},z_2)\mathrm{d}x\mathrm{d}y+o(1)  \big)
        \end{align*}
        and
        \begin{align*}
        \text{Var}(P_{i,3}^s)=\zeta^2\alpha^2 &\int_0^{\frac{\lfloor s T\rfloor}{T}} \int_0^1\int_0^1 f(u)^2 \Theta(u,z_1)\Theta(u,z_2)\Tilde{B}(\frac{i}{d},x,\frac{i}{d},y)\mathrm{d}z_1\mathrm{d}z_2\mathrm{d}u+o(1).
    \end{align*}
As $T\to\infty $, we have
\[
\frac{ \max_{0\le t\le \lfloor sT \rfloor-1} \mb{E}\big[ Z_{{\gamma}}^{d,T}(\frac{t}{T},\frac{i}{d})^2 \big] }{\text{Var}(P_{i,3}^s)}\to 0 .
\]
 \purple{Therefore, there exists a centered Gaussian random field process $\{\xi_3(s,x)\}_{s\in [0,\tau],x\in [0,1]}$ such that for any $s_1,s_2\in [0,\tau]$ and $x,y\in [0,1]$,
\begin{align*}
&\mb{E}[\xi_3(s_1,x)\xi_3(s_2,y)]= \sigma_3(s_1,s_2,x,y)\\
\text{with} &\quad \sigma_3(s_1,s_2,x,y) =\int_0^{s_1\wedge s_2}\int_0^1\int_0^1  \Theta(s,z_1)\Theta(s,z_2)\Tilde{B}(x,z_1,x,z_2)\mathrm{d}z_{1}\mathrm{d}z_{2}\mathrm{d}s ,
\end{align*}
and
$P_{i,3}^s\to \alpha\zeta \int_0^s f(u)\mathrm{d}\xi_3(u,\frac{i}{d})$ in distribution for all $i\in [d]$. By checking the covariance of the process $\{(P_{1,3}^s,P_{2,3}^s,\cdots,P_{d,3}^s)\}_{s\ge 0}$ via equation \eqref{eq:correlation M}, we have 
 $\{(P_{i,3}^s)_{i\in [d]}\}_{s\ge 0}\to \{(\alpha\zeta \int_0^s f(u)\mathrm{d}\xi_3(u,\frac{i}{d}))_{i\in [d]}\}_{s\ge 0}$, where the limit is a finite dimensional Gaussian process.} Furthermore, using the facts that $A\in C([0,1]^2)$ and $B(\cdot,x,\cdot,y)\in C([0,1]^2)$ we have for any $z_1\in [\frac{i}{d},\frac{i+1}{d}) $, $z_2\in [\frac{j}{d},\frac{j+1}{d})$, 
\begin{align*}
&\quad (i+1-dz_1)(j+1-dz_2)\Tilde{B}(\frac{i}{d},x,\frac{j}{d},y)\\
&+(i+1-dz_1)(dz_2-j)\Tilde{B}(\frac{i}{d},x,\frac{j+1}{d},y)\\
&+(dz_1-i)(j+1-dz_2)\Tilde{B}(\frac{i+1}{d},x,\frac{j+1}{d},y)\\
&+(dz_1-i)(dz_2-j)\Tilde{B}(\frac{i+1}{d},x,\frac{j+1}{d},y)\\
&=\Tilde{B}(z_1,x,z_2,y)+o(1)
\end{align*}
and, that any $x\in [\frac{i}{d},\frac{i+1}{d})$:
\begin{align}\label{eq:CLT limit term}
    -(i+1-dx)P_{i,3}^s-(dx-i)P_{i+1,3}^s\todbt  \alpha \zeta \int_0^s f(u)\mathrm{d}\xi_3(u,x).
\end{align}
    \end{itemize}
 Also notice that $\mb{E}[P_{i,3}^s P_{j,2}^s]=0$ for any $s\in [0,\tau]$ and $i,j\in [d]$, and so we have $\mb{E}[\xi_2(s,x)\xi_3(s,y)]=0$ for any $s\in [0,\tau]$ and $x,y\in [0,1]$. Therefore, the Gaussian random field process $\{\xi_3(s,x)\}_{s\in [0,\tau],x\in [0,1]}$ is independent of the Gaussian random field process $\{\xi_2(s,x)\}_{s\in [0,\tau],x\in [0,1]}$. 
\item [(4)] To study the term $P_{4}^s$, since $\Theta\in C^1([0,\tau];C^1([0,1]))$ and \eqref{eq:mainmean ODE LLN}, we have
\begin{align*}
    &\qquad \Theta(\frac{t+1}{T},x)-\Theta(\frac{t}{T},x)\\
    &=-\alpha\int_{\frac{t}{T}}^{\frac{t+1}{T}}\int_0^1 A(x,y)\Theta(s,y)\mathrm{d}y\mathrm{d}s \\
    &=\frac{\alpha}{d T}\sum_{j=1}^d A(x,\frac{j}{d})\Theta(\frac{t}{T},\frac{j}{d}) +O(d^{-1}T^{-1}+T^{-2})\\
    &=\eta \sum_{j=1}^d \big( (i+1-dx) A(\frac{i}{d},\frac{j}{d})+(dx-i)A(\frac{i+1}{d},\frac{j}{d})O(d^{-1}) \big)\Theta(\frac{t}{T},\frac{j}{d})\\
    &\quad+O(d^{-1}T^{-1}+T^{-2}).
\end{align*}
Since $\eta d T\to \alpha$, we have
\begin{align*}
  P_{4}^s&= \gamma  \sum_{t=1}^{\lfloor sT\rfloor-1}f(\frac{t}{T})  O( \eta+ d^{-1}T^{-1}+T^{-2}) \\
  &=  \frac{1}{T} \sum_{t=1}^{\lfloor sT\rfloor-1}f(\frac{t}{T}) O(d^{-1}\gamma+T^{-1}\gamma). 
\end{align*}
\end{itemize}
Therefore, combining our approximations, we get
\begin{align}\label{eq:RHS of CLT test applied sum}
    &\text{RHS}\eqref{eq:CLT test function applied sum} \nonumber \\
    =&-\alpha \int_0^{\frac{\lfloor sT\rfloor}{T}}\int_0^1 f(u)A(x,y)U(u,y)\mathrm{d}y\mathrm{d}u+\alpha \beta \int_0^s f(u) \mathrm{d}\xi_2(u,x) 1_{\{\gamma \sigma d^{-1}T^{-\frac{1}{2}}\to\beta\}}  \\
    &+\alpha\zeta \int_0^s f(u)\mathrm{d}\xi_3(u,x) 1_{\{\gamma T^{-\frac{1}{2}}\to \zeta\}}+\int_0^{\frac{\lfloor sT\rfloor}{T}}f(u)\mathrm{d}u O(\gamma d^{-1}+\gamma T^{-1}) +o(1), \nonumber
\end{align}
where $\{\xi_2(s,x)\}_{s\in [0,\tau],x\in [0,1]}$ and $\{\xi_3(s,x)\}_{s\in [0,\tau],x\in [0,1]}$ are two independent Gaussian field processes adapted to the same filtration $\{\mc{F}_s\}_{s\in [0,\tau]}$ with covariances given by \eqref{eq:gaussian field covariance CLT} and \eqref{eq:gaussian parameters} respectively. $\beta,\zeta\in [0,\infty) $ because $\max(\gamma T^{-\frac{1}{2}},\gamma d^{-1},\gamma \sigma d^{-1}T^{-\frac{1}{2}} )\le C_{s,2}$. Combine \eqref{eq:LHS CLT test function applied sum} and \eqref{eq:RHS of CLT test applied sum}, and we have for any $s\in [0,\tau]$,
\begin{align*}
    &\qquad - \int_0^s U(u,x)f'(u)\mathrm{d}u+f(s)U(s,x)-f(0)U(0,x)\\
    &=-\alpha \int_0^s \int_0^1 f(u)A(x,y)U(u,y)\mathrm{d}y\mathrm{d}u+\alpha \beta \int_0^s f(u) \mathrm{d}\xi_2(u,x) 1_{\{\gamma \sigma d^{-1}T^{-\frac{1}{2}}\to\beta\}}\\
    &\quad+\alpha \zeta \int_0^s f(u) \mathrm{d}\xi_3(u,x) 1_{\{\gamma T^{-\frac{1}{2}}\to \zeta\}}
    + O(\gamma d^{-1}+\gamma T^{-1}) +o(1),
\end{align*}
giving us the desired result.
\end{proof}

\begin{proof}[Proof of Theorem \ref{thm:mainfluc}.]

\textcolor{black}{ The proof follows from the results established in Theorem \ref{thm:masterclttheorem}, as we explain below.}

\textcolor{black}{The case when the particle interaction dominates and the noise dominates  simply follow from integration by parts according to the consequences (a) and (b) of Theorem \ref{thm:mainfluc} respectively, as we discussed in Section \ref{sec:fluctuation scaling limit}. The uniqueness and existence of solution in $C([0,\tau];C([0,1]))$ follows from part.1 of Theorem \ref{prop:existence and uniqueness of SDE}.}

\textcolor{black}{For the case when the interpolation error dominates, since $f$ is smooth, we know that $$
\int_0^{(\cdot)} U( u ,x)  f'(u) \mathrm{d}u\in C^1([0,\tau])\quad\text{and}\quad \int_0^{(\cdot)} f(u)A(x,y)U(u,y)\mathrm{d}u\mathrm{d}y\in C^1([0,\tau]).$$ 
Therefore according to \eqref{eq:large discretization error}, $f(\cdot)U(\cdot,x)\in C^1([0,\tau])$ for any $x\in[0,1]$ which implies that $U(\cdot,x)\in C^1([0,\tau])$ for any $x\in [0,1]$. We can then apply integration by parts to the left side of \eqref{eq:large discretization error}. Therefore, $U$ satisfies \eqref{eq:mainthm fluc SDE CLT large discretization error}.\\
Since $A$ satisfies Assumption \ref{ass:continuous data}, for any $U_1,U_2\in C([0,\tau];C([0,1]))$, we have
\begin{align*}
     &\quad\alpha\sup_{x\in [0,1]} \bigg|  \int_0^1 A(x,y) \left( U_1(s,y)-U_2(s,y) \right)\mathrm{d}y \bigg| \\
     &\le \alpha \sup_{x,y\in [0,1]} \left|A(x,y)\right| \sup_{y\in [0,1]} \left| U_1(s,y)-U_2(s,y)  \right|\\
    &\le \alpha C_2  \sup_{y\in [0,1]} \left|U_1(s,y)-U_2(s,y)  \right|.
\end{align*}
Therefore the right-hand side of \eqref{eq:mainthm fluc SDE CLT large discretization error} is Lipschitz in $U(s,\cdot)$ for any $s\in [0,\tau]$.  According to the Picard-Lindel\"{o}f theorem (see, for example,~\cite{arnold1992ordinary}), there exists a unique solution in $C([0,\tau];C([0,1]))$ to \eqref{eq:mainthm fluc SDE CLT large discretization error} with any initial condition $U(0,\cdot)=U_0(\cdot)\in C([0,1])$.}

\textcolor{black}{With the uniqueness of solution to \eqref{eq:mainthm fluc SDE CLT large discretization error} and Theorem \ref{thm:masterclttheorem}, we know that every subsequence of $\{U^{d,T}\}_{d\ge 1,T>0}$ has a further subsequence converging weakly to a unique $U\in C([0,\tau];C([0,1]))$. Therefore, $\{U^{d,T}\}_{d\ge 1,T>0}$ converges weakly to $U$ as $d,T\to\infty$.}
\end{proof}

\purple{
\begin{lemma}\label{lem:joint tightness CLT} Let $P^{d,T}_l \in C([0,\tau];C([0,1]))$ be the interpolated function defined in the the proof of Theorem \ref{thm:masterclttheorem} for $l=1,2,3,4$. Under the conditions in Theorem \ref{thm:masterclttheorem}, the sequence denoted by $\{(U^{d,T}, P^{d,T}_1, P^{d,T}_2, P^{d,T}_3,P^{d,T}_4)\}_{d\ge 1, T>0}$ is tight in the space $C([0,\tau];C([0,1]))^{\otimes 5}$.      
\end{lemma}
}
\begin{proof}[Proof of Lemma \ref{lem:joint tightness CLT}] Because of the product structure, it suffices to show tightness for each one of $U^{d,T}, P^{d,T}_1, P^{d,T}_2, P^{d,T}_3,P^{d,T}_4$. The tightness of $\{U^{d,T}\}$ is proved in Theorem \ref{thm:tightness of continuous interpolation CLT}. In this proof, we show tightness for $\{P^{d,T}_l\}$ for $l=1,2,3,4$.

\purple{
\textbf{Tightness of $P_1^{d,T}$.} First, we have $P^{d,T}_1(0,0)=\eta f(0) \sum_{j=1}^d W(0,0)W(0,\frac{j}{d}) U^{d,T}(0,\frac{j}{d})$. Under assumptions in Theorem \ref{thm:masterclttheorem}, we have \begin{align*}
    \mb{P}(|P^{d,T}_1(0,0)| >N )&\le N^{-2} \eta^2 f(0)^2 \sum_{j=1}^d B(0,\frac{j}{d},0,\frac{j}{d})\sum_{j=1}^d \mb{E}[|U^{d,T}(0,\frac{j}{d})|^2] \\
    &\le C_5 \lv f \rv_{\infty}^2 N^{-2}\eta^2 d^2 D^2  .
\end{align*}
Since $\eta d T=\alpha<\infty$, we have $\mb{P}(|P^{d,T}_1(0,0)| >N ) \lesssim N^{-2}\to 0$ as $N\to\infty$. Therefore, $\{P^{d,T}_1(0,0)\}_{d\ge 1, T>0}$ is tight in the probability space. Next, we observe that for any $0\le t_1<t_2\le \lfloor sT \rfloor-1$ and $i_1,i_2\in [d]$,
\begin{align*}
    &\qquad P^{d,T}_1(\frac{t_2}{T},\frac{i_2}{d})-P^{d,T}_1(\frac{t_1}{T},\frac{i_1}{d}) \\
    &= \eta \sum_{t=t_1}^{t_2-1}\sum_{j=1}^d f(\frac{t}{T})W(\frac{t}{T},\frac{i_2}{d})W(\frac{t}{T},\frac{j}{d})U^{d,T}(\frac{t}{T},\frac{j}{d})\\
    &\quad +\eta \sum_{t=0}^{t_1-1}\sum_{j=1}^df(\frac{t}{T})\big(W(\frac{t}{T},\frac{i_2}{d})W(\frac{t}{T},\frac{j}{d})-W(\frac{t}{T}, \frac{i_1}{d})W(\frac{t}{T},\frac{j}{d})\big)U^{d,T}(\frac{t}{T},\frac{j}{d})\\
    &= \eta \sum_{t=t_1}^{t_2-1}\sum_{j=1}^d f(\frac{t}{T})A(\frac{i_2}{d},\frac{j}{d})U^{d,T}(\frac{t}{T},\frac{j}{d})\\
    &\quad +\eta \sum_{t=0}^{t_1-1}\sum_{j=1}^df(\frac{t}{T})\big(A(\frac{i_2}{d},\frac{j}{d})-A(\frac{i_1}{d},\frac{j}{d})\big)U^{d,T}(\frac{t}{T},\frac{j}{d}) \\
    &\quad+ \eta \sum_{t=t_1}^{t_2-1}\sum_{j=1}^d f(\frac{t}{T})\big(W(\frac{t}{T},\frac{i_2}{d})W(\frac{t}{T},\frac{j}{d}) - A(\frac{i_2}{d},\frac{i}{d}) \big)U^{d,T}(\frac{t}{T},\frac{j}{d})\\
    &\quad +\eta \sum_{t=0}^{t_1-1}\sum_{j=1}^df(\frac{t}{T})\bigg(W(\frac{t}{T},\frac{i_2}{d})W(\frac{t}{T},\frac{j}{d})-W(\frac{t}{T}, \frac{i_1}{d})W(\frac{t}{T},\frac{j}{d})\\
    &\qquad\qquad\qquad\qquad\qquad\qquad-A(\frac{i_2}{d},\frac{j}{d})+A(\frac{i_1}{d},\frac{j}{d}))\bigg)U^{d,T}(\frac{t}{T},\frac{j}{d}).
\end{align*}
It follows from the arguments of bounding $\mb{E}[M_{i,1}^4],\mb{E}[M_{i,2}^4]$ and $\mb{E}[M_{i,j,1}^4],\mb{E}[M_{i,j,2}^4]$ in the proofs of Proposition \ref{prop:L2 tightness in time CLT} and Proposition \ref{prop:L2 tightness in time and space CLT} that
\begin{align*}
    &\ \mb{E}[| P^{d,T}_1(\frac{t_2}{T},\frac{i_2}{d})-P^{d,T}_1(\frac{t_1}{T},\frac{i_1}{d}) |^4]\\
    \lesssim &\ \eta^4\lv f \rv_{\infty}^4 d^4 (t_2-t_1)^4+\eta^4\lv f \rv_{\infty}^4 d^4 (t_2-t_1)^2\\
     &~~~~+ \eta^4\lv f \rv_{\infty}^4 (i_2-i_1)^4 t_1^4   +\eta^4\lv f \rv_{\infty}^4 (i_2-i_1)^4 t_1^2  \\
   \lesssim&\ (\eta d T)^4 \big(\frac{t_2-t_1}{T}\big)^2 + (\eta d T)^4 \big(\frac{i_2-i_1}{d}\big)^4.
\end{align*}
Since $P^{d,T}_1$ is piecewise linear in both variables and $\eta d T=\alpha$ for all $d\ge 1, T>0$, for all $0\le s_1<s_2\le \tau$ and $0\le x_1<x_2\le 1$, we have
\begin{align*}
    {\mb{E}\big[ \big| P_1^{d,T}(s_1,x_1)-P_1^{d,T}(s_2,x_2)\big|^4 \big]}
    &\lesssim  (s_2-s_1)^2 + (x_1-x_2)^4.
\end{align*}
The tightness of $P^{d,T}_1$ in $C([0,\tau];C([0,1]))$ follows from the Kolmogorov's tightness criterion.}

\purple{\textbf{Tightness of $P_2^{d,T}$.} First, we have $P^{d,T}_2(0,0)=-\eta \gamma f(0)  W(0,0) \varepsilon^0$. Under assumptions in Theorem \ref{thm:masterclttheorem}, we have \begin{align*}
    \mb{P}(|P^{d,T}_2(0,0)| >N )&\le N^{-2} \eta^2 \gamma^2 \sigma^2 f(0)^2 A(0,0) \\
    &\le C_2 \lv f \rv_{\infty}^2 N^{-2}\eta^2  \gamma^2 \sigma^2 .
\end{align*}
Since $\eta d T=\alpha$ and $\gamma \sigma d^{-1}T^{-\frac{1}{2}}\le C_{s,2}$, $\mb{P}(|P^{d,T}_2(0,0)| >N ) \lesssim N^{-2}\to 0$ as $N\to\infty$. Therefore, $\{P^{d,T}_2(0,0)\}_{d\ge 1, T>0}$ is tight in the probability space. Next, we observe that for any $0\le t_1<t_2\le \lfloor sT \rfloor-1$ and $i_1,i_2\in [d]$,
\small{
\begin{align*}
    &\qquad P^{d,T}_2(\frac{t_2}{T},\frac{i_2}{d})-P^{d,T}_2(\frac{t_1}{T},\frac{i_1}{d}) \\
    &= -\eta\gamma \sum_{t=t_1}^{t_2-1}f(\frac{t}{T})W(\frac{t}{T},\frac{i_2}{d})\varepsilon^t-\eta\gamma \sum_{t=0}^{t_1-1}\sum_{j=1}^df(\frac{t}{T})\big(W(\frac{t}{T},\frac{i_2}{d})-W(\frac{t}{T}, \frac{i_1}{d})\big)\varepsilon^t.
\end{align*}
}
It follows from the arguments of bounding $\mb{E}[M_{i,5}^4]$ and $\mb{E}[M_{i,j,5}^4],$ in the proofs of Proposition \ref{prop:L2 tightness in time CLT} and Proposition \ref{prop:L2 tightness in time and space CLT} that
\begin{align*}
    &\quad\mb{E}[| P^{d,T}_2(\frac{t_2}{T},\frac{i_2}{d})-P^{d,T}_2(\frac{t_1}{T},\frac{i_1}{d}) |^4]\\
    &\lesssim \gamma^4\eta^4\lv f\rv_{\infty}^4\sigma^4 (t_2-t_1)^2+\gamma^4\eta^4\lv f \rv_{\infty}^4 \sigma^4 T^2 \big(\frac{i_2-i_1}{d}\big)^4 \big(\frac{t_1}{T}\big)^2\\
   &\lesssim (\eta \gamma\sigma T^{\frac{1}{2}})^4 \big(\frac{t_2-t_1}{T}\big)^2 + (\eta \gamma\sigma T^{\frac{1}{2}})^4 \big(\frac{i_2-i_1}{d}\big)^4.
\end{align*}
Since $P^{d,T}_2$ is piecewise linear in both variables and $\eta \gamma \sigma T^{-\frac{1}{2}} =\alpha (\eta\gamma d^{-1}T^{-\frac{1}{2}})\le \alpha C_{s,2} $ for all $d\ge 1, T>0$, for all $0\le s_1<s_2\le \tau$ and $0\le x_1<x_2\le 1$, we have
\begin{align*}
    {\mb{E}\big[ \big| P_1^{d,T}(s_1,x_1)-P_1^{d,T}(s_2,x_2)\big|^4 \big]}
    &\lesssim  (s_2-s_1)^2 + (x_1-x_2)^4.
\end{align*}
The tightness of $P^{d,T}_2$ in $C([0,\tau];C([0,1]))$ follows from the Kolmogorov's tightness criterion.
}

\purple{\textbf{Tightness of $P_3^{d,T}$.} Notice that at $(t,x)=(0,0)$, we have
$$P^{d,T}_3(0,0)=\eta\gamma f(0)  \sum_{j=1}^d \big(W(0,0)W(0,\frac{j}{d}) -A(0,\frac{j}{d})\big) \Theta(0,\frac{j}{d}).$$ Under assumptions in Theorem \ref{thm:masterclttheorem}, we have \begin{align*}
    \mb{P}(|P^{d,T}_3(0,0)| >N )&\le N^{-2} \eta^2 \gamma^2  f(0)^2 \big( \sum_{j=1}^d  B(0,\frac{j}{d},0,\frac{j}{d})- A(0,\frac{j}{d})^2 \big) \big(\sum_{j=1}^d \mb{E}[|\Theta(0,\frac{j}{d})|^2] \big) \\
    &\le (C_5+C_2^2) \lv f \rv_{\infty}^2 \lv \Theta(0,\cdot) \rv_{\infty}^2 N^{-2}\eta^2  \gamma^2 d^2 .
\end{align*}
Since $\eta d T=\alpha$ and $\gamma T^{-\frac{1}{2}}\le C_{s,2}$, $\mb{P}(|P^{d,T}_3(0,0)| >N ) \lesssim N^{-2}\to 0$ as $N\to\infty$. Therefore, $\{P^{d,T}_3(0,0)\}_{d\ge 1, T>0}$ is tight in the probability space. Next, we observe that for any $0\le t_1<t_2\le \lfloor sT \rfloor-1$ and $i_1,i_2\in [d]$,
\begin{align*}
    &\qquad P^{d,T}_3(\frac{t_2}{T},\frac{i_2}{d})-P^{d,T}_3(\frac{t_1}{T},\frac{i_1}{d}) \\
    &= \gamma\eta \sum_{t=t_1}^{t_2-1}\sum_{j=1}^d f(\frac{t}{T})\big(W(\frac{t}{T},\frac{i_2}{d})W(\frac{t}{T},\frac{j}{d}) - A(\frac{i_2}{d},\frac{i}{d}) \big)\Theta(\frac{t}{T},\frac{j}{d})\\
    &\ +\gamma\eta \sum_{t=0}^{t_1-1}\sum_{j=1}^df(\frac{t}{T})\bigg(W(\frac{t}{T},\frac{i_2}{d})W(\frac{t}{T},\frac{j}{d})-W(\frac{t}{T}, \frac{i_1}{d})W(\frac{t}{T},\frac{j}{d})\\
    &\qquad\qquad\qquad\qquad\qquad-A(\frac{i_2}{d},\frac{j}{d})+A(\frac{i_1}{d},\frac{j}{d}))\bigg)\Theta(\frac{t}{T},\frac{j}{d}).
\end{align*}
It follows from the arguments of bounding $\mb{E}[M_{i,3}^4]$ and $\mb{E}[M_{i,j,4}^4],$ in the proofs of Proposition \ref{prop:L2 tightness in time CLT} and Proposition \ref{prop:L2 tightness in time and space CLT} that
\begin{align*}
    &\quad\mb{E}[| P^{d,T}_3(\frac{t_2}{T},\frac{i_2}{d})-P^{d,T}_3(\frac{t_1}{T},\frac{i_1}{d}) |^4]\\
    &\lesssim \gamma^4\eta^4\lv f\rv_{\infty}^4 d^4 (t_2-t_1)^2+\gamma^4\eta^4\lv f \rv_{\infty}^4 d^4 T^2 \big(\frac{i_2-i_1}{d}\big)^4 \big(\frac{t_1}{T}\big)^2\\
   &\lesssim (\eta \gamma d T^{\frac{1}{2}})^4 \big(\frac{t_2-t_1}{T}\big)^2 + (\eta \gamma d T^{\frac{1}{2}})^4 \big(\frac{i_2-i_1}{d}\big)^4.
\end{align*}
Since $P^{d,T}_3$ is piecewise linear in both variables and $\eta \gamma d T^{-\frac{1}{2}} =\alpha (\gamma T^{-\frac{1}{2}})\le \alpha C_{s,2} $ for all $d\ge 1, T>0$, for all $0\le s_1<s_2\le \tau$ and $0\le x_1<x_2\le 1$, we have
\begin{align*}
    {\mb{E}\big[ \big| P_3^{d,T}(s_1,x_1)-P_3^{d,T}(s_2,x_2)\big|^4 \big]}
    &\lesssim  (s_2-s_1)^2 + (x_1-x_2)^4.
\end{align*}
The tightness of $P^{d,T}_3$ in $C([0,\tau];C([0,1]))$ follows from the Kolmogorov's tightness criterion.
}

\purple{\textbf{Tightness of $P_4^{d,T}$.} First, we have $P^{d,T}_4(0,0)=\gamma\eta f(0)  \sum_{j=1}^d A(0,\frac{j}{d}) \Theta(0,\frac{j}{d})$. Under assumptions in Theorem \ref{thm:masterclttheorem}, we have \begin{align*}
    \mb{P}(|P^{d,T}_4(0,0)| >N )&\le N^{-2} \eta^2 \gamma^2  f(0)^2 \big( \sum_{j=1}^d  A(0,\frac{j}{d})^2 \big) \big(\sum_{j=1}^d \mb{E}[|\Theta(0,\frac{j}{d})|^2] \big) \\
    &\le C_2^2 \lv f \rv_{\infty}^2 \lv \Theta(0,\cdot) \rv_{\infty}^2 N^{-2}\eta^2  \gamma^2 d^2 .
\end{align*}
Since $\eta d T=\alpha$ and $\gamma T^{-\frac{1}{2}}\le C_{s,2}$, $\mb{P}(|P^{d,T}_4(0,0)| >N ) \lesssim N^{-2}\to 0$ as $N\to\infty$. Therefore, $\{P^{d,T}_4(0,0)\}_{d\ge 1, T>0}$ is tight in the probability space. Next, we observe that for any $0\le t_1<t_2\le \lfloor sT \rfloor-1$ and $i_1,i_2\in [d]$,
\begin{align*}
    &\quad P^{d,T}_4(\frac{t_2}{T},\frac{i_2}{d})-P^{d,T}_4(\frac{t_1}{T},\frac{i_1}{d}) \\
    &= \gamma \sum_{t=t_1}^{t_2-1} f(\frac{t}{T})\big(\Theta(\frac{t+1}{T},\frac{i_2}{d})-\Theta(\frac{t}{T},\frac{i_2}{d}) +\eta \sum_{j=1}^d A(\frac{i_2}{d},\frac{j}{d}) \Theta(\frac{t}{T},\frac{j}{d}) \big)\\
    &\quad+\gamma \sum_{t=0}^{t_1-1}f(\frac{t}{T})\bigg(\Theta(\frac{t+1}{T},\frac{i_2}{d})-\Theta(\frac{t}{T},\frac{i_2}{d})-\Theta(\frac{t+1}{T},\frac{i_1}{d})+\Theta(\frac{t}{T},\frac{i_1}{d})  \\
    &\qquad\qquad\qquad\qquad+\eta \sum_{j=1}^d \big(A(\frac{i_2}{d},\frac{j}{d})-A(\frac{i_1}{d},\frac{j}{d})\big) \Theta(\frac{t}{T},\frac{j}{d}) \bigg).
\end{align*}
It follows from the arguments of bounding $\mb{E}[M_{i,4}^4]$ and $\mb{E}[M_{i,j,3}^4],$ in the proofs of Proposition \ref{prop:L2 tightness in time CLT} and Proposition \ref{prop:L2 tightness in time and space CLT} that
\begin{align*}
    &\quad\mb{E}[| P^{d,T}_4(\frac{t_2}{T},\frac{i_2}{d})-P^{d,T}_4(\frac{t_1}{T},\frac{i_1}{d}) |^4]\\
    &\lesssim \gamma^4\eta^4\lv f\rv_{\infty}^4  (t_2-t_1)^4+\gamma^4\eta^4\lv f \rv_{\infty}^4 d^4 T^{-4} (t_2-t_1)^4\\
    &\ + \gamma^4 T^{-4} \lv f \rv_{\infty}^4\big( \frac{i_2-i_1}{d} \big)^4 \big(\frac{t_1}{T}\big)^4+ \gamma^4 d^{-4} \lv f \rv_{\infty}^4\big( \frac{i_2-i_1}{d} \big)^4 \big(\frac{t_1}{T}\big)^4 \\
   &\lesssim (\gamma d^{-1} )^4 \big(\frac{t_2-t_1}{T}\big)^4 + ( \gamma  T^{-1})^4 \big(\frac{i_2-i_1}{d}\big)^4.
\end{align*}
Since $P^{d,T}_4$ is piecewise linear in both variables and $ \gamma d^{-1}, \gamma T^{-\frac{1}{2}} \le  C_{s,2} $ for all $d\ge 1, T>0$, for all $0\le s_1<s_2\le \tau$ and $0\le x_1<x_2\le 1$, we have
\begin{align*}
    {\mb{E}\big[ \big| P_4^{d,T}(s_1,x_1)-P_4^{d,T}(s_2,x_2)\big|^4 \big]}
    &\lesssim  (s_2-s_1)^4 + (x_1-x_2)^4.
\end{align*}
The tightness of $P^{d,T}_4$ in $C([0,\tau];C([0,1]))$ follows from the Kolmogorov's tightness criterion.
}
\end{proof}

\section{Proofs for the existence and uniqueness of the SDE}\label{sec:existence and uniqueness SDE} 

\purple{Our proofs for the existence and uniqueness of the SDE rely on the proporty that $\xi_1,\xi_2,\xi_3$ are all bounded and continuous with probability 1, which can be checked by applying \cite[Theorem 1.4.1]{adler2007random}. In the following proofs, for simplicity, we utilize this property without referring to it repeatedly.}

\begin{proof}[Proof of part (a) of Theorem \ref{prop:existence and uniqueness of SDE}.]

 \textbf{Existence:} We prove existence based on the Picard iteration argument. Let $U_0(s,x)=U(0,x)$ for all $s\in [0,\tau]$. The Picard iteration is given by
 \begin{align}
 \label{eq:Picard iteration Linfty}
 \begin{aligned}
     U_k(s,x)=&~U(0,x)-\alpha\int_0^s \int_0^1 A(x,y)U_{k-1}(u,y)\mathrm{d}y\mathrm{d}u\\
      &~~+\alpha\beta\int_0^s \mathrm{d}\xi_2(u,x)+\alpha\zeta\int_0^s \mathrm{d}\xi_3(u,x).
\end{aligned}
 \end{align}
According to the definition of $\xi_2,\xi_3$, if $U(0,\cdot)\in C([0,1])$, then $U_k(s,\cdot)\in C([0,1])$ for all $s\in [0,\tau]$ and $k\ge 0$. Next we show that $U_k\in C([0,\tau];C([0,1]))$. For any $K>0$, define the stopping time, 
$$
\tau_{K,\purple{\Upsilon}}\coloneqq\min( \tau, \inf\{ s\in [0,\tau]: \max_{0\le k\le K}\lv U_k(s,\cdot) \rv_\infty\ge \purple{\Upsilon} \}  ).
$$
It is easy to see that $\tau_{K,\purple{\Upsilon}}\to \tau$ almost surely as $\purple{\Upsilon}\to\infty$. Define $U_k^{\purple{\Upsilon}}(s,x)\coloneqq U_k(s\wedge \tau_{K,\purple{\Upsilon}},x )$. We have
 \begin{align*}
     U_k^{\purple{\Upsilon}}(s,x)&=U(0,x)-\alpha\int_0^{s}\int_0^1 A(x,y)U_{k-1}(u\wedge \tau_{\purple{\Upsilon}},y) 1_{(0,\tau_{\purple{\Upsilon}})}(u) \mathrm{d}y\mathrm{d}u \\
     &\quad +\alpha\beta\int_0^s 1_{(0,\tau_{\purple{\Upsilon}})}(u)\mathrm{d}\xi_2(u\wedge \tau_{\purple{\Upsilon}},x)+\alpha\zeta\int_0^s 1_{(0,\tau_{\purple{\Upsilon}})}(u)\mathrm{d}\xi_3(u\wedge \tau_{\purple{\Upsilon}},x).
 \end{align*}
 Therefore under Assumption \ref{ass:tightness in space LLN} we obtain
 \begin{align*}
     &\quad\max_{0\le k\le K}\lv U_k^{\purple{\Upsilon}}(s,\cdot) \rv_\infty\\
     &\le \lv U(0,\cdot) \rv_\infty + C_2 \alpha \int_0^s \int_0^1   \max_{0\le k\le K}\lv U_k^{\purple{\Upsilon}}(u,\cdot) \rv_\infty \mathrm{d}y \mathrm{d}u \\
     &\quad + \alpha \beta\sup_{x\in [0,1]} \big|  \int_0^s 1_{(0,\tau_{\purple{\Upsilon}})}(u)\mathrm{d}\xi_2(u,x) \big|+ \alpha \zeta\sup_{x\in [0,1]} \big|  \int_0^s 1_{(0,\tau_{\purple{\Upsilon}})}(u)\mathrm{d}\xi_3(u,x) \big|. 
 \end{align*}
 Taking the expectation, we get
 \begin{align*}
     &\quad \mb{E}\big[ \max_{0\le k\le K}\sup_{s\in [0,\tau]}\lv U_k^{\purple{\Upsilon}}(s,\cdot) \rv_{\infty} \big]\\
     &\le \mb{E}\big[ \lv U(0,\cdot) \rv_\infty \big]+C_2 \alpha\int_0^\tau \mb{E}\big[ \max_{0\le k\le K} \lv U_k^{\purple{\Upsilon}}(u,\cdot) \rv_\infty \big] \mathrm{d}u\\
     &\quad +  \alpha\beta\mb{E}\big[\sup_{x\in [0,1]} \big( \int_0^\tau 1_{(0,\tau_{K,{\purple{\Upsilon}}})}(u)\mathrm{d}\xi_2(u,x) \big)^2\big]^{\frac{1}{2}}\\
     &\quad+\alpha\zeta\mb{E}\big[\sup_{x\in [0,1]} \big( \int_0^\tau 1_{(0,\tau_{K,{\purple{\Upsilon}}})}(u)\mathrm{d}\xi_3(u,x) \big)^2\big]^{\frac{1}{2}}.
 \end{align*}
 According to \eqref{eq:gaussian field covariance CLT} and \eqref{eq:gaussian parameters}, under Assumption \ref{ass:tightness in space LLN}, we have
 \begin{align*}
      &\mb{E}\big[ \max_{0\le k\le K}\sup_{s\in [0,\tau]}\lv U_k^{\purple{\Upsilon}}(s,\cdot) \rv_{\infty} \big]\\
      \le& \mb{E}\big[ \lv U(0,\cdot) \rv_\infty \big]+C_2\alpha \int_0^\tau \mb{E}\big[ \max_{0\le k\le K} \lv U_k^{\purple{\Upsilon}}(u,\cdot) \rv_\infty \big] \mathrm{d}u\\
      &\quad+ \alpha\beta \mb{E}\big[\sup_{x\in [0,1]}  \int_0^\tau A(x,x)  \mathrm{d}(u\wedge {\tau_{K,{\purple{\Upsilon}}}}) \big]^{\frac{1}{2}} \\
     &\quad + \alpha\zeta \mb{E}\big[\sup_{x\in [0,1]}  \int_0^\tau \int_0^1\int_0^1 \Theta(u,z_1)\Theta(u,z_2)\Tilde{B}(x,z_1,x,z_2)\mathrm{d}z_1\mathrm{d}z_2   \mathrm{d}(u\wedge {\tau_{K,{\purple{\Upsilon}}}}) \big]^{\frac{1}{2}} \\
     \le&  \mb{E}\left[ \lv U(0,\cdot) \rv_\infty \right]+C_2 \alpha \int_0^\tau \mb{E}\big[ \max_{0\le k\le K} \sup_{u\in [0,s]}\lv U_k^{\purple{\Upsilon}}(u,\cdot) \rv_\infty \big] \mathrm{d}s \\
     &\quad +C_2^{\frac{1}{2}} \alpha\beta \tau^{\frac{1}{2}}+ (C_2^2+C_{5})^{\frac{1}{2}}\alpha\zeta \lv \Theta \rv_\infty \tau^{\frac{1}{2}},
 \end{align*}
 where $\Tilde{B}(x,z_1,x,z_2)=B(x,z_1,x,z_2)-A(x,z_1)A(x,z_2)$. With Gronwall's inequality, we then get
 \begin{align*}
     &\mb{E}\big[ \max_{0\le k\le K}\sup_{s\in [0,\tau]}\lv U_k^{\purple{\Upsilon}}(s,\cdot) \rv_\infty \big]\\
     \le &\bigg(\mb{E}\big[ \lv U(0,\cdot) \rv_\infty\big] + C_2^{\frac{1}{2}}\alpha\beta \tau^{\frac{1}{2}}+ (C_{5}+C_2^2)^{\frac{1}{2}}\alpha \zeta  \lv \Theta\rv_\infty \tau^{\frac{1}{2}}  \bigg)\exp(C_2 \alpha \tau).
 \end{align*}
Since $K$ is arbitrary, we can push $K, {\purple{\Upsilon}}\to\infty$ and according to Monotone convergence theorem, we prove for any $k$,
 \begin{align*}
     &\quad\mb{E}\big[ \sup_{s\in [0,\tau]}\lv U_k(s,\cdot) \rv_\infty \big]\\
     &\le \bigg(\mb{E}\big[ \lv U(0,\cdot) \rv_\infty\big] +  C_2^{\frac{1}{2}}\alpha\beta \tau^{\frac{1}{2}}+ (C_{5}+C_2^2)^{\frac{1}{2}}\alpha \zeta  \lv \Theta\rv_\infty \tau^{\frac{1}{2}} \bigg)\exp(C_2 \alpha \tau).
 \end{align*}
 Therefore for any $k$, $U_k\in C\left( [0,\tau];C([0,1])\right)$. Next we show that $\{U_k\}_{k=1}^\infty$ converges in the space $C([0,\tau];C([0,1]))$. From \eqref{eq:Picard iteration Linfty}, we have for any $k\ge 1$ and $s\in [0,\tau]$,
 \begin{align*}
     \lv U_{k+1}(s,\cdot)-U_k(s,\cdot) \rv_\infty &= \alpha\lv \int_0^s \int_0^1 A(\cdot,y)\left( U_k(u,y)-U_{k-1}(u,y) \right)\mathrm{d}y\mathrm{d}u \rv_\infty\\
     &\le C_2 \alpha s \int_0^s \lv U_{k}(u,\cdot)-U_{k-1}(u,\cdot) \rv_\infty \mathrm{d}u,
 \end{align*}
 and 
 \begin{align*}
     &\lv U_{1}(s,\cdot)-U_0(s,\cdot) \rv_\infty \le \alpha \lv\int_0^s \int_0^1 A(\cdot,y)U_0(y)\mathrm{d}y \mathrm{d}u \rv_\infty\\ &\qquad\qquad\qquad\qquad\qquad\qquad+\alpha\beta\sup_{x\in [0,1]} \big| \int_0^s \mathrm{d}\xi_2(u,x) \big|+\alpha\zeta\sup_{x\in [0,1]} \big| \int_0^s \mathrm{d}\xi_3(u,x) \big|.
 \end{align*}
Hence, for any $s\in [0,\tau]$,
 \begin{align*}
     &\mb{E}\big[ \sup_{r\in [0,s]}\lv U_{1}(r,\cdot)-U_0(r,\cdot) \rv_\infty \big]\\
     \le&(1+C_2\alpha s)\mb{E}\big[ \lv U(0,\cdot) \rv_\infty\big]+C_2^{\frac{1}{2}}\alpha\beta s^{\frac{1}{2}}+(C_{5}+C_2^2)^{\frac{1}{2}}\alpha\zeta  \lv \Theta \rv_{\infty} s^{\frac{1}{2}}\\
     \coloneqq& C(s)\le C(\tau)<\infty,
 \end{align*}
 and 
 \begin{align*}
    & \mb{E}\big[ \sup_{r\in [0,s]}\lv U_{k+1}(r,\cdot)-U_k(r,\cdot) \rv_\infty\big] \\
     \le &~C_2\alpha \int_0^s \mb{E}\big[\sup_{u\in [0,r]}\lv U_{k}(u,\cdot)-U_{k-1}(u,\cdot) \rv_\infty\big] \mathrm{d}u.
 \end{align*}
 By induction we get for any $k\ge 1$,
 \begin{align*}
     \mb{E}\big[ \sup_{r\in [0,s]} \lv U_{k+1}(r,\cdot)-U_k(r,\cdot) \rv_\infty \big]\le \frac{C(\tau) \left( C_2 \alpha  s \right)^k }{k!}.
 \end{align*}
 Therefore according to Markov's inequality,
 \begin{align*}
     \mb{P}\big( \purple{\sup_{r\in [0,s]}}\lv U_{k+1}(r,\cdot)-U_k(r,\cdot) \rv_\infty>2^{-k}  \big)\le \frac{C(\tau)\left( 2C_2 \purple{\alpha} s \right)^k}{k!}.
 \end{align*}
 Let $\Omega$ be the path space on which the Gaussian field processes $\xi_2$ and $\xi_3$ are defined. According to Borel-Cantelli Lemma we have for almost every $\omega\in \Omega$, there exists $k(\omega)$ such that $\sup_{r\in [0,s]}\lv U_{k+1}(r,\cdot)-U_k(r,\cdot) \rv_\infty\le \frac{1}{2^k}$ for any $k\ge k(\omega)$. Therefore, with probability $1$, $\{U_k(\cdot,\cdot)\}_{k\ge 1}$ converges in $C([0,\tau];C([0,1]))$ with limit $U(\cdot,\cdot)\in C([0,\tau];C([0,1]))$. Furthermore, \purple{by taking the limit of \eqref{eq:Picard iteration Linfty}}, we can check that $U$ satisfies SDE \eqref{eq:mainthm fluc SDE CLT}. Therefore, existence is proved.\\

 \textbf{Uniqueness:} Suppose that there exist two solutions, $U,\Bar{U}$ to SDE \eqref{eq:mainthm fluc SDE CLT}, from \eqref{eq:integral solution SDE}, we have for any $s\in [0,\tau]$, $x\in [0,1]$,
 \begin{align*}
     \Bar{U}(s,x)-U(s,x)=-\alpha \int_0^s \int_0^1 A(x,y)\left( \Bar{U}(u,y)-U(u,y) \right)\mathrm{d}y \mathrm{d}u,
 \end{align*}
 which implies that
 \begin{align*}
     \mb{E}\big[ \sup_{s\in [0,\tau]} \lv \Bar{U}(s,\cdot)-U(s,\cdot) \rv_\infty \big]\le C_2 \alpha \mb{E}\big[ \int_0^\tau \sup_{r\in [0,s]} \lv \Bar{U}(r,\cdot)-U(r,\cdot) \rv_\infty \mathrm{d}s  \big].
 \end{align*}
 By Gronwall's inequality we have
 \begin{align*}
     \mb{E}\big[ \sup_{s\in [0,\tau]} \lv \Bar{U}(s,\cdot)-U(s,\cdot) \rv_\infty \big]\le \mb{E}\big[ \lv \Bar{U}(0,\cdot)-U(0,\cdot) \rv_\infty \big] \exp\left( C_2 \alpha \tau \right)=0.
 \end{align*}
 Therefore there is a unique solution to \eqref{eq:mainthm fluc SDE CLT} in $C([0,\tau];C([0,1]))$.
\end{proof}

\begin{proof}[Proof of part (b) of Theorem \ref{prop:existence and uniqueness of SDE}.]

\textbf{Stability:} Suppose there exist solutions to the SDE \eqref{eq:mainthm fluc SDE CLT}, then the integral form solution can be written as $ \forall \ s\in [0,\tau]$
\begin{align}\label{eq:integral solution SDE}
\begin{aligned}
    U(s,x)=&~U(0,x)-\alpha\int_0^s \int_0^1 A(x,y)U(u,y)\mathrm{d}y \mathrm{d}u\\
    &~+ \alpha \beta\int_0^s \mathrm{d}\xi_2(u,x)+\alpha \zeta \int_0^s \mathrm{d}\xi_3(u,x).
\end{aligned}
\end{align}
 Define the stopping time $\tau_N\coloneqq\min( \tau, \inf\{ s\in [0,\tau]: \lv U(s,\cdot) \rv_{L^2([0,1])}\ge N \}  )$. It is easy to see that $\tau_N\to \tau$ almost surely as $N\to\infty$. Define $U^N(s,x)\coloneqq U(s\wedge \tau_N,x )$. We have
 \begin{align*}
     U^N(s,x)&=U(0,x)-\alpha\int_0^{s}\int_0^1 A(x,y)U(u\wedge \tau_N,y) 1_{(0,\tau_N)}(u) \mathrm{d}y\mathrm{d}u \\
     &\quad +\alpha \beta\int_0^s 1_{(0,\tau_N)}(u)\mathrm{d}\xi_2(u\wedge \tau_N,x)+\alpha \zeta\int_0^s 1_{(0,\tau_N)}(u)\mathrm{d}\xi_3(u\wedge \tau_N,x).
 \end{align*}
 Therefore we obtain an estimation of $U^N(s,\cdot)$ in the space of $L^2([0,1])$,
 \begin{align*}
     &\int_0^1 \big| U^N(s,x) \big|^2 \mathrm{d}x\le 4\lv U(0,\cdot) \rv_{L^2([0,1])}^2 +4\alpha^2s \int_0^s \int_0^1 \big( \int_0^1 A(x,y) U^N(u,y)\mathrm{d}y  \big)^2 \mathrm{d}x\mathrm{d}u\\
     &+ 4\alpha^2\beta^2 \int_0^1 \big| \int_0^s 1_{(0,\tau_N)}(u)\mathrm{d}\xi_2(u,x) \big|^2 \mathrm{d}x +  4\alpha^2\zeta^2 \int_0^1 \big| \int_0^s 1_{(0,\tau_N)}(u)\mathrm{d}\xi_3(u,x) \big|^2 \mathrm{d}x.  
 \end{align*}
 Under Assumption \ref{ass:continuous data}, we have that for the integral operator $$\mc{A}: g\in L^2([0,1])\mapsto \int_0^1 A(x,y)g(y)\mathrm{d}y\in L^2([0,1]),$$ 
 \begin{align*}
     \lv \mc{A} \rv_{op}\coloneqq\sup_{\lv g \rv_{L^2([0,1])}=1} \int_0^1 \big( \int_0^1 A(x,y)g(y) \mathrm{d}y \big)^2 \mathrm{d}x\le C_2.
 \end{align*}
 Therefore we get
 \begin{align*}
     &\lv U^N(s,\cdot) \rv_{L^2([0,1])}^2\le 4\lv U(0,\cdot) \rv_{L^2([0,1])}^2+4C_2^2\alpha^2 s \int_0^s \lv U^N(u,\cdot) \rv_{L^2([0,1])}^2 \mathrm{d}s\\
     & + 4\alpha^2\beta^2 \int_0^1 \left( \int_0^s 1_{(0,\tau_N)}(u)\mathrm{d}\xi_2(u,x) \right)^2 \mathrm{d}x +4\alpha^2\zeta^2 \int_0^1 \left( \int_0^s 1_{(0,\tau_N)}(u)\mathrm{d}\xi_3(u,x) \right)^2 \mathrm{d}x. \\
 \end{align*}
 Taking the supreme over $[0,\tau]$ and taking the expectation, we get
 \begin{align*}
     &\quad\mb{E}\big[ \sup_{s\in [0,\tau]} \lv U^N(s,\cdot) \rv_{L^2([0,1])}^2 \big]\\
     &\le 4\mb{E}\big[ \lv U(0,\cdot) \rv_{L^2([0,1])}^2 \big]+4C_2^2 \alpha^2 \tau \int_0^\tau \mb{E}\big[ \lv U^N(s,\cdot) \rv_{L^2([0,1])}^2 \big] \mathrm{d}s\\
     &\quad + 4\alpha^2 \beta^2 \sup_{s\in[0,\tau]}\int_0^1 \mb{E}\big[\big( \int_0^s 1_{(0,\tau_N)}(u)\mathrm{d}\xi_2(u,x) \big)^2\big] \mathrm{d}x\\
     &\quad  4\alpha^2 \zeta^2 \sup_{s\in[0,\tau]}\int_0^1 \mb{E}\big[\big( \int_0^s 1_{(0,\tau_N)}(u)\mathrm{d}\xi_3(u,x) \big)^2\big] \mathrm{d}x.
 \end{align*}
\purple{Noticing that for any fixed $x\in [0,1]$, $\{\xi_i(s,x)\}_{s\in [0,\tau]}$ is a Gaussian process on $[0,\tau]$ for $i=2,3$. According to \eqref{eq:gaussian field covariance CLT} and \eqref{eq:gaussian parameters}, the quadratic variations of $\{\xi_2(s,x)\}_{s\in [0,\tau]}$ and $\{\xi_3(s,x)\}_{s\in [0,\tau]}$ are
\begin{align*}
    [\mathrm{d}\xi_2(\cdot,x) ]_s &=A(x,x) s, \\
    [\mathrm{d}\xi_3(\cdot,x) ]_s &=\int_0^1\int_0^1 \Theta(s,z_1)\Theta(s,z_2)\Tilde{B}(x,z_1,x,z_2)\mathrm{d}z_1\mathrm{d}z_2  ,
\end{align*}
where $\Tilde{B}(x,z_1,x,z_2)=B(x,x_1,x,z_2)-A(x,z_1)A(x,z_2)$. Under Assumption \ref{ass:continuous data} and the It\^{o}'s formula} \cite[Theorem~3.3]{RevuzYor1999}, we get
 \begin{align*}
      &~\quad\mb{E}\big[ \sup_{s\in [0,\tau]} \lv U^N(s,\cdot) \rv_{L^2([0,1])}^2 \big]\\&\le 4\mb{E}\big[ \lv U(0,\cdot) \rv_{L^2([0,1])}^2 \big]+4C_2^2 \alpha^2\tau \int_0^\tau \mb{E}\big[ \lv U^N(s,\cdot) \rv_{L^2([0,1])}^2 \big] \mathrm{d}s\\
      &\quad+ 4\alpha^2\beta^2 \int_0^1 \int_0^\tau   A(x,x) \mathrm{d}(s\wedge \tau_N)  \mathrm{d}x\\
     &\quad + 4\alpha^2\zeta^2 \int_0^1 \int_0^\tau \int_0^1 \int_0^1 \Theta(s,z_1)\Theta(s,z_2)\Tilde{B}(x,z_1,x,z_2)\mathrm{d}z_1\mathrm{d}z_2 d(s\wedge \tau_N)  \mathrm{d}x\\
     &\le 4\mb{E}\big[ \lv U(0,\cdot) \rv_{L^2([0,1])}^2 \big]+4C_2^2\alpha^2 \tau \int_0^\tau \mb{E}\big[ \lv U^N(s,\cdot) \rv_{L^2([0,1])}^2 \big] \mathrm{d}s\\
     &\quad +4C_2 \alpha^2\beta^2 \tau+ 4\left(C_{5}+C_2^2\right)\alpha^2\zeta^2 \int_0^\tau \lv \Theta(s,\cdot) \rv_{L^2([0,1])}^2 \mathrm{d}s \\
      &\le 4\mb{E}\big[ \lv U(0,\cdot) \rv_{L^2([0,1])}^2 \big]+4C_2^2 \alpha^2 \tau \int_0^\tau \mb{E}\big[ \sup_{r\in [0,s]}\lv U^N(r,\cdot) \rv_{L^2([0,1])}^2 \big] \mathrm{d}s\\
     & \quad+4C_2 \alpha^2\beta^2 \tau + 4(C_{5}+C_2^2) \alpha^2 \zeta^2\int_0^\tau  \lv \Theta(s,\cdot) \rv_{L^2([0,1])}^2 \mathrm{d}s.
 \end{align*}
  With Gronwall's inequality we get
 \begin{align*}
     &\mb{E}\big[ \sup_{s\in [0,\tau]} \lv U^N(s,\cdot) \rv_{L^2([0,1])}^2 \big] \\
     \le& 4\bigg(\mb{E}\big[ \lv U(0,\cdot) \rv_{L^2([0,1])}^2\big] +C_2\alpha^2 \beta^2 \tau \\& +(C_{5}+C_2^2)\alpha^2\zeta^2 \int_0^\tau  \lv \Theta(s,\cdot) \rv_{L^2([0,1])}^2 \mathrm{d}s  \bigg)\times \exp\left(4C_2^2\alpha^2 \tau^2\right).
 \end{align*}
 Last letting $N\to\infty$ and according to Monotone convergence theorem, we prove \eqref{eq:L2 estimation SDE solution}.\\

 \textbf{Existence:} As before, we use the Picard iteration argument to show existence. Let $U_0(s,x)=U(0,x)$ for all $s\in [0,\tau]$. The Picard iteration is given by
 \begin{align}\label{eq:Picard iteration}
 \begin{aligned}
      U_k(s,x)=&~U(0,x)-\alpha\int_0^s \int_0^1 A(x,y)U_{k-1}(u,y)\mathrm{d}y\mathrm{d}u \\
      &~+\alpha\beta\int_0^s \mathrm{d}\xi_2(u,x)+\alpha\zeta\int_0^s \mathrm{d}\xi_3(u,x).
\end{aligned}
 \end{align}
 With similar argument as in the stability part, we get for any $K\in \mb{N}$,
 \begin{align*}
     &\mb{E}\big[ \max_{1\le k\le K} \lv U_k(s,\cdot) \rv_{L^2([0,1])}^2 \big]\le  4C_2^2 \alpha^2 s \int_0^s \max_{1\le k\le K} \mb{E}\big[ \lv U_k(u,\cdot) \rv_{L^2([0,1])}^2 \big] \mathrm{d}u\\ 
     & +4\big( \mb{E}\big[ \lv U(0,\cdot) \rv_{L^2([0,1])}^2 \big]+C_2\alpha^2 \beta^2 s+(C_{5}+C_2^2)\alpha^2 \zeta^2  \int_0^s \lv \Theta(u,\cdot) \rv_{L^2([0,1])}^2 \mathrm{d}u \big).
\end{align*}
Gronwall's inequality implies that for any $k\in \mb{N}$,
 \begin{align*}
     &\qquad \mb{E}\big[  \lv U_k(s,\cdot) \rv_{L^2([0,1])}^2 \big]\\
     &\le  4\bigg( \mb{E}\big[ \lv U(0,\cdot) \rv_{L^2([0,1])}^2 \big]+C_2\alpha^2\beta^2 s+(C_{5}+C_2^2)\alpha^2\zeta^2  \int_0^s \lv \Theta(u,\cdot) \rv_{L^2([0,1])}^2 \mathrm{d}u \bigg)\\
     &\qquad\qquad\qquad\qquad\qquad\qquad\qquad\qquad\qquad\qquad\qquad\qquad\times \exp\left( 4C_2^2 \alpha^2 s^2 \right).
 \end{align*}
 Therefore for any $k$ and $s\in [0,\tau]$, $U(s,\cdot)\in L^2([0,1])$. Next we show that $\{U_k\}_{k=1}^\infty$ converges in the space $C([0,\tau];L^2([0,1]))$. From \eqref{eq:Picard iteration}, we have for any $k\ge 1$ and $s\in [0,\tau]$,
 \begin{align*}
     &\quad\lv U_{k+1}(s,\cdot)-U_k(s,\cdot) \rv_{L^2([0,1])}^2 \\&= \alpha^2\lv \int_0^s \int_0^1 A(x,y)\left( U_k(u,y)-U_{k-1}(u,y) \right)\mathrm{d}y\mathrm{d}u \rv_{L^2([0,1])}^2\\
     &\le C_2^2 \alpha^2 s \int_0^s \lv U_{k}(u,\cdot)-U_{k-1}(u,\cdot) \rv_{L^2([0,1])}^2 \mathrm{d}u,
 \end{align*}
 and 
 \begin{align*}
     &\lv U_{1}(s,\cdot)-U_0(s,\cdot) \rv_{L^2([0,1])}^2 \\
     \le& 3 \alpha^2\int_0^1 \big( \int_0^s \int_0^1 A(x,y)U_0(y)\mathrm{d}y \mathrm{d}u \big)^2 \mathrm{d}x+3\alpha^2\beta^2\int_0^1 \big( \int_0^s \mathrm{d}\xi_2(u,x) \big)^2 \mathrm{d}x\\
     & +3\alpha^2\zeta^2\int_0^1 \big( \int_0^s \mathrm{d}\xi_3(u,x) \big)^2 \mathrm{d}x.
 \end{align*}
 Therefore, for any $s\in [0,\tau]$, we have that
 \begin{align*}
     \mb{E}\big[ \sup_{r\in [0,s]}\lv U_{1}(r,\cdot)-U_0(r,\cdot) \rv_{L^2([0,1])}^2 \big]&\le  3 C_2^2 \alpha^2 s^2 \mb{E}\big[ \lv U(0,\cdot) \rv_{L^2([0,1])}^2 \big] +3C_2\alpha^2\beta^2 s \\
     &\quad +3\left(C_{5}+C_2^2\right)\alpha^2\zeta^2  \int_0^s \lv \Theta(u,\cdot) \rv_{L^2([0,1])}^2 \mathrm{d}u\\
     &\coloneqq C(s)\le C(\tau)<\infty,
 \end{align*}
 and 
 \begin{align*}
     &\mb{E}\big[ \sup_{r\in [0,s]}\lv U_{k+1}(r,\cdot)-U_k(r,\cdot) \rv_{L^2([0,1])}^2\big] \\
     \le& C_2^2  \alpha^2 s \int_0^s \mb{E}\big[\sup_{u\in [0,r]}\lv U_{k}(u,\cdot)-U_{k-1}(u,\cdot) \rv_{L^2([0,1])}^2\big] \mathrm{d}u.
 \end{align*}
 Hence, by induction, we get that for any $k\ge 1$,
 \begin{align*}
     \mb{E}\big[ \sup_{r\in [0,s]} \lv U_{k+1}(r,\cdot)-U_k(r,\cdot) \rv_{L^2([0,1])}^2 \big]\le \frac{C(\tau) \left( C_2^2 \alpha^2 \tau s \right)^k }{k!}.
 \end{align*}
 Therefore according to Markov inequality,
 \begin{align*}
     \mb{P}\big( \purple{\sup_{r\in [0,s]}}\lv U_{k+1}(r,\cdot)-U_k(r,\cdot) \rv_{L^2([0,1])}^2>2^{-k}  \big)\le \frac{C(\tau)\left( 2C_2^2 \alpha^2 \tau s \right)^k}{k!}.
 \end{align*}
 Let $\Omega$ be the path space on which the Gaussian field processes $\xi_2$ and $\xi_3$ are defined. According to Borel-Cantelli Lemma, we have that for almost every $\omega\in \Omega$, there exists $k(\omega)$ such that $\sup_{r\in [0,s]}\lv U_{k+1}(r,\cdot)-U_k(r,\cdot) \rv_{L^2([0,1])}^2\le \frac{1}{2^k}$ for any $k\ge k(\omega)$. Therefore, with probability $1$, $\{U_k(\cdot,\cdot)\}_{k\ge 1}$ converges in $C([0,\tau];L^2([0,1]))$ with limit $U(\cdot,\cdot)\in C([0,\tau];L^2([0,1]))$. Furthermore, \purple{by taking the limit of \eqref{eq:Picard iteration}}, we can check that $U$ satisfies SDE \eqref{eq:mainthm fluc SDE CLT}. Therefore, the solution to \eqref{eq:mainthm fluc SDE CLT} exists. \\

 \textbf{Uniqueness:} Suppose that there exist two solutions, $U,\Bar{U}$ to SDE \eqref{eq:mainthm fluc SDE CLT}, from \eqref{eq:integral solution SDE}, we have for any $s\in [0,\tau]$, $x\in [0,1]$,
 \begin{align*}
     \Bar{U}(s,x)-U(s,x)=-\alpha\int_0^s \int_0^1 A(x,y)\left( \Bar{U}(u,y)-U(u,y) \right)\mathrm{d}y \mathrm{d}u,
 \end{align*}
 which implies that
 \begin{align*}
     &\mb{E}\big[ \sup_{s\in [0,\tau]} \lv \Bar{U}(s,\cdot)-U(s,\cdot) \rv_{L^2([0,1])} \big]\\
     \le &C_2^2 \alpha^2 \tau \mb{E}\big[ \int_0^\tau \sup_{r\in [0,s]} \lv \Bar{U}(r,\cdot)-U(r,\cdot) \rv_{L^2([0,1])} \mathrm{d}s  \big].
 \end{align*}
 By Gronwall's inequality we have
 \begin{align*}
     &\mb{E}\big[ \sup_{s\in [0,\tau]} \lv \Bar{U}(s,\cdot)-U(s,\cdot) \rv_{L^2([0,1])} \big]\\
     \le & \mb{E}\big[ \lv \Bar{U}(0,\cdot)-U(0,\cdot) \rv_{L^2([0,1])}^2 \big] \exp\left( C_2^2 \alpha^2 \tau^2 \right)=0.
 \end{align*}
 Therefore this is a unique solution to \eqref{eq:mainthm fluc SDE CLT} in $C([0,\tau];L^2([0,1]))$.
\end{proof}

\section{Proofs for applications from Section \ref{sec:applications}}

\begin{proof}[Proof of Proposition \ref{prop:general scaling limits}.]
First, according to the definition of $\MSE^{d,T}$, we have
\begin{align*}
    \MSE^{d,T}(s)&=\frac{1}{d}\sum_{i=1}^d \Bar{\Theta}^{d,T}(\frac{\lfloor sT \rfloor}{T},\frac{i}{d})^2 = \int_0^1 \Bar{\Theta}^{d,T}(s,y)^2 \mathrm{d}y+o(1) \\
    &\to \MSE(s)\coloneqq\int_0^1 \Theta(s,y)^2 \mathrm{d}y \quad \text{as }d,T\to\infty,
\end{align*}
where the second identity follows from the fact that $\Bar{\Theta}^{d,T}\in C([0,\tau];C([0,1]))$. The last step follows from Theorem \ref{thm:mainmean}. Next, for the predictive error, according to the definition of $\PrE^{d,T}$, 
\begin{align*}
    \PrE^{d,T}(s)&=\frac{1}{d^2}\sum_{i,j=1}^d A(\frac{i}{d},\frac{j}{d})\Bar{\Theta}^{d,T}(\frac{\lfloor sT \rfloor}{T},\frac{i}{d})\Bar{\Theta}^{d,T}(\frac{\lfloor sT \rfloor}{T},\frac{j}{d})\\
    &=\int_0^1\int_0^1 A(x,y)\Bar{\Theta}^{d,T}(s,x)\Bar{\Theta}^{d,T}(s,y)\mathrm{d}x\mathrm{d}y +o(1)\\
    &\to \PrE(s) \quad\text{as }d,T\to\infty,
\end{align*}
where the second identity follows from the fact that $\Bar{\Theta}^{d,T}\in C([0,\tau];C([0,1]))$. The last step follows from Theorem \ref{thm:mainmean}. Last, different equations that characterize $\Theta$ follows directly from Theorem \ref{thm:mainmean}.
\end{proof}

\begin{proof}[Proof of Proposition \ref{prop:fluctuation of errors}.] Scaling conditions in $(1)$, $(2)$ and $(3)$ corresponds to the different scalings in Theorem \ref{thm:mainfluc}. Therefore, for each case, we have $\Bar{\Theta}^{d,T}\topb \Theta$, $U^{d,T}\todbt U$ with $\Theta$ being the solution to \eqref{eq:mainmean ODE LLN}. $U$ solves \eqref{eq:mainthm fluc SDE CLT}, \eqref{eq:mainthm fluc SDE CLT relative large variance}, \eqref{eq:mainthm fluc SDE CLT large discretization error} in $(1)$, $(2)$ and $(3)$ respectively. Next we apply these convergence results to prove $(i)$.
\begin{align*}
    &\qquad \gamma\big( \MSE^{d,T}(s)-\MSE(s) \big)\\
    &=\underbrace{-\gamma \big( \int_0^1 \Theta(s,x)^2 \mathrm{d}x-\frac{1}{d}\sum_{i=1}^d \Theta(s,\frac{i}{d})^2  \big)}_{N_1^s}-\underbrace{\frac{\gamma}{d}\sum_{i=1}^d \big( \Theta(s,\frac{i}{d})^2-\Theta(\frac{\lfloor sT \rfloor}{T},\frac{i}{d})^2 \big)}_{N_2^s} \\
    &\quad -\underbrace{\frac{2\gamma}{d} \sum_{i=1}^d \Theta(\frac{\lfloor sT \rfloor}{T},\frac{i}{d})\big( \Theta(\frac{\lfloor sT \rfloor}{T},\frac{i}{d}) -\Bar{\Theta}^{d,T}(\frac{\lfloor sT \rfloor}{T},\frac{i}{d}) \big)}_{N_3^s} \\
    &\quad+\underbrace{\frac{\gamma}{d}\sum_{i=1}^d \big( \Theta(\frac{\lfloor sT \rfloor}{T},\frac{i}{d})-\Bar{\Theta}^{d,T}(\frac{\lfloor sT \rfloor}{T},\frac{i}{d}) \big)^2}_{N_4^s}.
\end{align*}
\begin{itemize}
    \item [(a)] For $N_1^s$, since $\Theta(s,\cdot)\in C^1([0,1])\cap L^\infty([0,1])$ for any $s\in [0,\tau]$, we have
    \begin{align*}
     \big| \int_0^1 \Theta(s,x)^2 \mathrm{d}x-\frac{1}{d}\sum_{i=1}^d \Theta(s,\frac{i}{d})^2  \big|\le 2\lv \Theta(s,\cdot) \rv_\infty \lv \partial_x \Theta(s,\cdot) \rv_\infty d^{-1}   
    \end{align*}

    Therefore 
    \begin{itemize}
        \item in $(1)$, $N_1^s=O(T^{\frac{1}{2}d^{-1}})=o(1)$ because $T=o(d^2)$. 
        \item in $(2)$, $N_1^s=O(\sigma^{-1}T^{\frac{1}{2}})=o(1)$ because $\max(d,T^{\frac{1}{2}})\ll \sigma$.
        \item in $(3)$, $N_1^s=O(\gamma d^{-1})=o(1)$ because $\gamma\ll d$.
    \end{itemize}
    \item [(b)] For $N_2^s$, since $\Theta(\cdot,x)\in C^1([0,\tau])\cap L^\infty ([0,\tau]) $ for any $x\in [0,1]$, we have
    \begin{align*}
        \big|  \Theta(s,\frac{i}{d})^2-\Theta(\frac{\lfloor sT \rfloor}{T},\frac{i}{d})^2 \big|\le 2\lv \Theta(\cdot,x) \rv_\infty \lv \partial_s \Theta(\cdot,x) \rv_\infty T^{-1}.
    \end{align*}
    Therefore
    \begin{itemize}
        \item  in $(1)$, $N_2^s=O(\gamma T^{-1})=O(T^{-\frac{1}{2}})=o(1)$.
        \item in $(2)$, $N_2^s=O(\gamma T^{-1})=O(\sigma^{-1}dT^{-\frac{1}{2}})=o(1)$ because $\max(d,T^{\frac{1}{2}})\ll \sigma$.
        \item in $(3)$, $N_2^s=O(\gamma T^{-1})=o(dT^{-1})=o(1)$ because $d=O(T^{\frac{1}{2}})$.
    \end{itemize}
    \item [(c)] For $N_3^s$, according to the definition of $U^{d,T}$, we have
    \begin{align*}
        N_3^s&=-\frac{2}{ d}\sum_{i=1}^d \Theta(\frac{\lfloor sT \rfloor}{T},\frac{i}{d})U^{d,T}(\frac{\lfloor sT \rfloor}{T},\frac{i}{d})\\
        &=-2\int_0^1 \Theta(s,x)U(s,x)\mathrm{d}x\\
        &\quad+2\big( \int_0^1 \Theta(s,x)U(s,x)\mathrm{d}x-\frac{1}{d}\sum_{i=1}^d \Theta(\frac{\lfloor sT \rfloor}{T},\frac{i}{d})U(\frac{\lfloor sT \rfloor}{T},\frac{i}{d}) \big)\\
        &\quad +\frac{2}{d}\sum_{i=1}^d \Theta(\frac{\lfloor sT \rfloor}{T},\frac{i}{d})\big( U(\frac{\lfloor sT \rfloor}{T},\frac{i}{d})-U^{d,T}(\frac{\lfloor sT \rfloor}{T},\frac{i}{d})\big).
    \end{align*}
    As $U\in C([0,\tau];C([0,1]))$, $U^{d,T}\todbt U$ and $\Theta\in C^1([0,\tau];C^1([0,1]))$, we have that  
    \begin{align*}
    \int_0^1 \Theta(s,x)U(s,x)\mathrm{d}x-\frac{1}{d}\sum_{i=1}^d \Theta(\frac{\lfloor sT \rfloor}{T},\frac{i}{d})U(\frac{\lfloor sT \rfloor}{T},\frac{i}{d})&=o(1)\\
    \frac{2}{d}\sum_{i=1}^d \Theta(\frac{\lfloor sT \rfloor}{T},\frac{i}{d})\big( U(\frac{\lfloor sT \rfloor}{T},\frac{i}{d})-U^{d,T}(\frac{\lfloor sT \rfloor}{T},\frac{i}{d})\big)&=o(1).
    \end{align*}
    Therefore, we get
    \begin{align*}
        N_3^s\todbt - 2\int_0^1 \Theta(s,x)U(s,x)\mathrm{d}x. 
    \end{align*} 
    \item [(d)] For $N_4^s$, according to the definition of $U^{d,T}$, we have
    \begin{align*}
        N_4^s&=-\frac{1}{\gamma d} \sum_{i=1}^d U^{d,T}(\frac{\lfloor sT \rfloor}{T},\frac{i}{d})^2\\
        &=-\gamma^{-1} \int_0^1 U(s,x)^2 \mathrm{d}x-\gamma^{-1}\big( \int_0^1 U(s,x)^2 \mathrm{d}x-\frac{1}{d}\sum_{i=1}^d U(\frac{\lfloor sT \rfloor}{T},\frac{i}{d})^2 \big)\\
        &\quad -\gamma^{-1} d^{-1}\sum_{i=1}^d \big( U^{d,T}(\frac{\lfloor sT \rfloor}{T},\frac{i}{d})^2-U(\frac{\lfloor sT \rfloor}{T},\frac{i}{d})^2 \big).
    \end{align*}
    According to \eqref{eq:L2 estimation SDE solution} and the fact that $\gamma\gg 1$, $\gamma ^{-1} \int_0^1 U(s,x)^2 dx=o(1)$. As $U^{d,T}\todbt U\in C([0,\tau];C([0,1]))$, we have that
    \begin{align*}
    \int_0^1 U(s,x)^2 \mathrm{d}x-\frac{1}{d}\sum_{i=1}^d U(\frac{\lfloor sT \rfloor}{T},\frac{i}{d})^2&=o(1)\\
    \frac{1}{d}\sum_{i=1}^d \big( U^{d,T}(\frac{\lfloor sT \rfloor}{T},\frac{i}{d})^2-U(\frac{\lfloor sT \rfloor}{T},\frac{i}{d})^2 \big)&=o(1).
    \end{align*}
    Therefore, $N_4^s=o(1)$.
\end{itemize}
According to points $\textup{(a)}$, $\textup{(b)}$, $\textup{(c)}$, and $\textup{(d)}$ above, $\textup{(i)}$ is proved. The statement $\textup{(ii)}$ can be proved similarly under Assumption \ref{ass:continuous data}. We will leave it to the readers.
\end{proof}

\begin{proof}[Proof of Lemma \ref{lem:Sinusoidal graphon A}.]  Let $\bar{A}(x)=a_0+\sum_{k=1}^\infty b_k \cos(2\pi k x)$ for all $x\in [0,1]$. Then $A(x,y)=\bar{A}(|x-y|)$ for all $x,y\in [0,1]$. To prove that $A$ satisfies Assumption \ref{ass:continuous data}, it suffices to show that $\bar{A}\in C^2([0,1])$. Note that $\bar{A}$ is given in the form of Fourier series with orthonormal basis $\{1,\sqrt{2}\cos(2\pi k x), \sqrt{2}\sin(2\pi kx) \}_{k\ge 1}$. Under condition \eqref{eq:Sinusoidal graphon A parameter}, we have $\bar{A}\in H^{(5+\varepsilon)/2}([0,1])$, where $H^{(5+\varepsilon)/2}([0,1])$ is the Sobolev space of functions on $[0,1]$ with square-integrable weak derivatives up to order $(5+\varepsilon)/2$; see, for example, \cite{adams2003sobolev} for details about Sobolev spaces. By the  Sobolev embedding theorem \citep[Chapter 4]{adams2003sobolev}, we hence have that $\bar{A}\in C^2([0,1])$, thus implying the desired result.

According to \eqref{eq:Isserlis thm}, we have
\begin{align}\label{eq:Isserlis B}
    B(x_1,x_2,x_3,x_4)=\sum_{p\in P_4^2}\prod_{(i,j)\in p} A(x_i,x_j)
\end{align}
 From Lemma \ref{lem:Sinusoidal graphon A}, we know that $A$ is bounded and twice continuously differentiable. Therefore, $B$ is continuous and bounded. Furthermore, with \eqref{eq:Isserlis B}, we have
 \begin{align*}
    &\big| B(x_1,x_3,x_1,x_3)+B(x_2,x_3,x_2,x_3)-2B(x_1,x_3,x_2,x_3)  \big|\\
    =& \big| A(x_1,x_3) \left( A(x_1,x_1)+A(x_2,x_2)-2A(x_1,x_2) \right)-2\left( A(x_1,x_3)-A(x_2,x_3) \right)^2  \big| \\
    \le& C |x_1-x_2|^2,
\end{align*}
and 
\begin{align*}
    &\big| B(x_1,x_1,x_1,x_1)+B(x_2,x_2,x_2,x_2)+6B(x_1,x_1,x_2,x_2)\\
    &\qquad\qquad\qquad\qquad-4B(x_1,x_1,x_1,x_2)-4B(x_1,x_2,x_2,x_2)\big|\\
    =&3\big| A(x_1,x_1)+A(x_2,x_2)-2A(x_1,x_2) \big|^2\le C | x_1-x_2|^4,
 \end{align*}
 where the last inequality follows from Lemma \ref{lem:Sinusoidal graphon A}. Therefore, the function $B$ satisfies Assumption \ref{ass:continuous data}. For the eighth moments, similarly we have
 \begin{align}\label{eq:Isserlis E}
     E(x_1,\cdots,x_8)=\sum_{p\in P_8^2}\prod_{(i,j)\in p} A(x_i,x_j)
 \end{align}
Therefore according to Lemma \ref{lem:Sinusoidal graphon A}, $E$ is continuous and bounded. Furthermore, with \eqref{eq:Isserlis E}, letting $\tilde{x}=(x_3,x_4,x_5,x_6)$ we have that
 \begin{align*}
     &~\big| E(x_1,x_1,x_1,x_1,\tilde{x})+E(x_2,x_2,x_2,x_2,\tilde{x})+6E(x_1,x_1,x_2,x_2,\tilde{x})\\&\qquad\qquad\qquad -4E(x_1,x_1,x_1,x_2,\tilde{x})-4E(x_1,x_2,x_2,x_2,\tilde{x})\big|\\
     =&~\big| 3 ( A(x_1,x_1)+A(x_2,x_2)-2A(x_1,x_2) )^2 \sum_{p\in P_{3-6}^2}\prod_{(i,j)\in p} A(x_i,x_j)\\
     &\quad+\sum_{\sigma(\mathbf{i})} \prod_{j=1}^4(A(x_1,x_{i_j})-A(x_2,x_{i_j})) \\
     &+3( A(x_1,x_1)+A(x_2,x_2)-2A(x_1,x_2) ) \sum_{\sigma(\mathbf{i})} A(x_{i_3},x_{i_4})\prod_{j=1}^2\big(A(x_1,x_{i_j})-A(x_2,x_{i_j})\big)  \big|\\
     \le&~C\big|x_1-x_2\big|^4,
 \end{align*}
 where we use $P_{3-6}^2$ to denote the set of all pairings in $\{3,4,5,6\}$ and $\underset{\sigma(\mathbf{i})}{\sum}$ to denote summing the 4-tuple $\mathbf{i}\coloneqq(i_1,i_2,i_3,i_4)$ over all permutations of the set $\{3,4,5,6\}$. The last inequality above follows from Lemma \ref{lem:Sinusoidal graphon A}.  Therefore, the function $E$ satisfies Assumption \ref{ass:continuous data}.
\end{proof}

\begin{proof}[Proof of Proposition \ref{prop:exp decay of MSE}.]  We start by proving part (b). According to \eqref{eq:mainmean ODE LLN},
\begin{align*}
    \frac{\mathrm{d}}{\mathrm{d}s} \MSE(s)&=\frac{\mathrm{d}}{\mathrm{d}s} \int_0^1 \Theta(s,x)^2 \mathrm{d}x =-2\alpha \int_0^1\int_0^1 \Theta(s,x)A(x,y)\Theta(s,y) \mathrm{d}x\mathrm{d}y \\
    &=-2\alpha \sum_{i=1}^\infty \lambda_i \langle  \Theta(s,\cdot),\phi_i \rangle_{L^2([0,1])}^2\\
    &\le -2\alpha \lambda \MSE(s),
    \end{align*}
    where the equality in the second line follows from our assumption on $A$ and the inequality in the last line follows from the assumption that $\lambda\coloneqq\inf_{i} \lambda_i>0 $. The final claim follows from Gronwall's inequality.

 Next, we prove part (a). Define the functional $\mc{F}:L^2([0,1])\to \mb{R}$ as
 \[
 \mc{F}[\phi]\coloneqq \int_0^1 \int_0^1 A(x,y)\phi(x)\phi(y)\mathrm{d}x\mathrm{d}y,\qquad \forall \phi\in L^2([0,1]).
 \] 
 Its functional gradient $\frac{\delta \mc{F}[\phi]}{\delta \phi}:[0,1]\to \mb{R}$ and its functional Hessian $\frac{\delta^2 \mc{F}[\phi]}{\delta \phi^2}:[0,1]^2\to \mb{R}$ are given by
\begin{align*}
    \frac{\delta \mc{F}[\phi]}{\delta \phi}(x)=2\int_0^1 A(x,y)\phi(y)\mathrm{d}y, \quad  \frac{\delta^2 \mc{F}[\phi]}{\delta \phi^2}(x,y)=2A(x,y)
\end{align*}
Let $\Theta$ be the solution to \eqref{eq:mainmean ODE LLN}. For any $s\in [0,\tau]$, we have
\begin{align*}
    \PrE'(s)&=\frac{\mathrm{d}}{\mathrm{d}s} \mc{F}[\Theta(s,\cdot)] =-2\alpha \int_0^1\int_0^1 \int_0^1 A(x,y)A(y,z)\Theta(s,x)\Theta(s,z)\mathrm{d}x\mathrm{d}y\mathrm{d}z \\
    &=-\frac{\alpha}{2} \int_0^1 \left|\frac{\delta \mc{F}[\Theta(s,\cdot)]}{\delta \Theta(s,\cdot)}(y)  \right|^2 \mathrm{d}y,
\end{align*}
where the second identity follows from \eqref{eq:mainmean ODE LLN}. Integrate \purple{over} the interval $[0,s]$ and we get
\begin{align}\label{eq:inter step}
    \PrE(s)&=\PrE(0)-\frac{\alpha}{2} \int_0^s \int_0^1 \left|\frac{\delta \mc{F}[\Theta(r,\cdot)]}{\delta \Theta(r,\cdot)}(y)  \right|^2 \mathrm{d}y\mathrm{d}r \nonumber\\
    &=\mc{F}[\Theta(0,\cdot)]-\mc{F}[0]-\frac{\alpha}{2} \int_0^s \int_0^1 \left|\frac{\delta \mc{F}[\Theta(r,\cdot)]}{\delta \Theta(r,\cdot)}(y)  \right|^2 \mathrm{d}y\mathrm{d}r ,
\end{align}
where the last step follows from the fact that $\mc{F}[0]=0$. Notice that,
\begin{align}\label{eq:relation between function gradient and time derivative}
    \frac{\delta \mc{F}[\Theta(r,\cdot)]}{\delta \Theta(r,\cdot)}(y)=2\int_0^1 A(x,y)\Theta(r,y)\mathrm{d}y=-\frac{2}{\alpha} \partial_r \Theta(r,x).
\end{align}
Therefore, we have
\begin{align}\label{eq:convexity of functional}
    \PrE(0)-\mc{F}[0]&=-\Big\langle \frac{\delta \mc{F}[\Theta(0,\cdot)]}{\delta \Theta(0,\cdot)}, -\Theta(0,\cdot) \Big\rangle_{L^2([0,1])}- \int_0^1\int_0^1 \Theta(0,x) A(x,y) \Theta(0,y)\mathrm{d}x\mathrm{d}y  \nonumber \nonumber \\
    &\le -\frac{2}{\alpha} \langle \partial_s \Theta(s,x) |_{s=0}, \Theta(0,\cdot) \rangle_{L^2([0,1])} \nonumber\\
    &= -\frac{1}{\alpha} \left(\frac{\mathrm{d}}{\mathrm{d}s} \int_0^1 \Theta(s,y)^2 \mathrm{d}y\right)\bigg|_{s=0},
\end{align}
where the first step follows from Taylor expansion and the inequality follows from our assumption on $A$ and \eqref{eq:relation between function gradient and time derivative}. Meanwhile, according to \eqref{eq:relation between function gradient and time derivative},
\begin{align}\label{eq:square term}
    &\quad-\frac{\alpha}{2} \int_0^s \int_0^1 \left|\frac{\delta \mc{F}[\Theta(r,\cdot)]}{\delta \Theta(r,\cdot)}(y)  \right|^2 \mathrm{d}y\mathrm{d}r\\
    &=-\frac{2}{\alpha} \int_0^s \int_0^1 \left( \partial_r \Theta(r,y) \right)^2 \mathrm{d}y\mathrm{d}r =-\frac{2}{\alpha} \int_0^1 \int_0^s \partial_r \Theta(r,y) \Theta(r,y) \mathrm{d}r\mathrm{d}y \nonumber\\
    &=-\frac{2}{\alpha} \int_0^1  \Theta(s,y)\partial_r\Theta(r,y)|_{r=s}-\Theta(0,y)\partial_r\Theta(r,y)|_{r=0} \mathrm{d}y \nonumber \\
    &\quad +\frac{2}{\alpha}\int_0^1 \int_0^s \Theta(r,y)\partial^2_{rr} \Theta(r,y) \mathrm{d}r\mathrm{d}y \nonumber\\
    &=\frac{2}{\alpha}\left( \frac{\mathrm{d}}{\mathrm{d}r} \int_0^1 \Theta(r,x)^2 \mathrm{d}x\bigg|_{r=0}-\frac{\mathrm{d}}{\mathrm{d}r} \int_0^1 \Theta(r,x)^2 \mathrm{d}x\bigg|_{r=s}  \right)\nonumber \\
    &\quad +\frac{\alpha}{2} \int_0^s \int_0^1 \left|\frac{\delta \mc{F}[\Theta(r,\cdot)]}{\delta \Theta(r,\cdot)}(y)  \right|^2 \mathrm{d}y\mathrm{d}r
\end{align}
where the last step follows because
\begin{align*}
    &\quad\frac{2}{\alpha}\int_0^1 \int_0^s \Theta(r,y)\partial^2_{rr} \Theta(r,y) \mathrm{d}r \mathrm{d} y\\
    &=-2\int_0^1 \int_0^s \Theta(r,y) \int_0^1 A(y,z)\partial_r \Theta(r,z)\mathrm{d}z \mathrm{d}r\mathrm{d}y \\
    &=2\alpha \int_0^1 \int_0^s \Theta(r,y) \int_0^1 A(y,z) \int_0^1 A(z,x)\Theta(r,x)\mathrm{d}x \mathrm{d}z \mathrm{d}r \mathrm{d}y  \\
    &= 2\alpha \int_0^s \int_0^1 \left( \int_0^1 A(y,z)\Theta(r,z)\mathrm{d}z \right)^2 \mathrm{d}y \mathrm{d}s\\
    &=\frac{\alpha}{2} \int_0^s \int_0^1 \left|\frac{\delta \mc{F}[\Theta(r,\cdot)]}{\delta \Theta(r,\cdot)}(y)  \right|^2 \mathrm{d}y\mathrm{d}r.
\end{align*}
Therefore based on combining \eqref{eq:inter step}, \eqref{eq:convexity of functional} and \eqref{eq:square term}, we have that
\begin{align*}
    &\quad\PrE(s)\\
    &\le -\frac{1}{\alpha} \left(\frac{\mathrm{d}}{\mathrm{d}s} \int_0^1 \Theta(s,y)^2 \mathrm{d}y\right)\bigg|_{s=0}+\frac{1}{\alpha}\left( \frac{\mathrm{d}}{\mathrm{d}r} \int_0^1 \Theta(r,x)^2 \mathrm{d}x\bigg|_{r=0}-\frac{\mathrm{d}}{\mathrm{d}r} \int_0^1 \Theta(r,x)^2 \mathrm{d}x\bigg|_{r=s}  \right)\\
    &=-\frac{1}{\alpha} \frac{\mathrm{d}}{\mathrm{d}s} \int_0^1 \Theta(s,x)^2 \mathrm{d}x.
\end{align*}
Integrating over $[0,\tau]$, we have
\begin{align*}
    \int_0^\tau \PrE(s)\mathrm{d}s \le -\frac{\MSE(\tau)}{\alpha}+\frac{\MSE(0)}{\alpha}\le \frac{\MSE(0)}{\alpha}.
\end{align*}
From previous calculations we know that $\PrE'(s) =-\frac{\alpha}{2} \int_0^1 \left|\frac{\delta \mc{F}[\Theta(s,\cdot)]}{\delta \Theta(s,\cdot)}(y)  \right|^2 \mathrm{d}y\le 0$. Therefore \begin{align*}
   \PrE(\tau)\le \frac{\int_0^\tau \PrE(s)\mathrm{d}s}{\tau} \le \frac{\MSE(0)}{\alpha \tau}.
\end{align*}
\end{proof}

\bibliography{citation}       


\end{document}